\documentclass[11pt]{article}
\usepackage[T1]{fontenc}
\usepackage{verbatim}
\usepackage{amsmath,amsthm}
\usepackage{xspace}
\usepackage{pifont}
\usepackage{graphicx}
\usepackage{amssymb}
\usepackage{epic, eepic}
\usepackage{dsfont}
\usepackage{amssymb}
\usepackage{makeidx}
\usepackage{mathrsfs}
\usepackage{xcolor, colortbl}
\usepackage{subcaption}
\usepackage[numbers]{natbib}
\usepackage{enumerate}
\usepackage{enumitem}
\usepackage{afterpage}
\usepackage{mathtools}
\usepackage{epsf}
\usepackage{graphics,color}
\usepackage[export]{adjustbox}
\usepackage{tcolorbox}
\usepackage{authblk}
\usepackage[normalem]{ulem}
\RequirePackage[colorlinks,citecolor=blue,urlcolor=blue]{hyperref}
\mathtoolsset{showonlyrefs}

\def\eps{\varepsilon}
\def\lam{\lambda}
\def\al{\alpha}
\def\wt{\widetilde}
\def\Q{\mathbb{Q}}
\def\P{\mathbb{P}}
\def\E{{\mathbb{E}}}
\newcommand{\R}{{\mathbb R}}
\newcommand{\F}{{\mathbb F}}
\newcommand{\N}{{\mathbb N}}

\DeclareMathOperator*{\argmax}{arg\,max}

\newcommand{\bd}{\begin{displaymath}}
\newcommand{\ed}{\end{displaymath}}
\newcommand{\be}{\begin{equation}}
\newcommand{\ee}{\end{equation}}
\newcommand{\bq}{\begin{eqnarray}}
\newcommand{\eq}{\end{eqnarray}}
\newcommand{\bn}{\begin{eqnarray*}}
\newcommand{\en}{\end{eqnarray*}}
\newcommand{\dl}{\delta}

\newtheorem{theorem}{Theorem}[section]
\newtheorem{lemma}[theorem]{Lemma}
\newtheorem{proposition}[theorem]{Proposition}

\newtheorem{corollary}[theorem]{Corollary}

\newtheorem{remark}[theorem]{Remark}

\newtheorem{example}[theorem]{Example}

\newtheorem{definition}[theorem]{Definition}
\newtheorem{assumption}[theorem]{Assumption}
\numberwithin{equation}{section}
            
\date{\today}
\title{Stochastic Graphon Games with Interventions}
 
\author[]{Eyal Neuman}
\author[]{Sturmius Tuschmann\thanks{ST is supported by the EPSRC Centre for Doctoral Training in Mathematics of Random \mbox{Systems}: Analysis, Modelling and Simulation (EP/S023925/1).}}
\affil[]{Department of Mathematics, Imperial College London}

\begin{document}
\maketitle

\begin{abstract} 
We consider a class of targeted intervention problems in dynamic network and graphon games. First, we study a general dynamic network game in which players interact over a graph and maximize their heterogeneous, concave goal functionals, which depend on both their own actions and their interactions with their neighbors.
We establish the existence and uniqueness of the Nash equilibrium in both the finite-player network game and the corresponding infinite-player graphon game. We also prove the convergence of the Nash equilibrium in the network game to the one in the graphon game, providing explicit bounds on the convergence rate.

Using this framework, we introduce a central planner who implements a dynamic targeted intervention. Given a fixed budget, the planner maximizes the average welfare at equilibrium by perturbing the players' heterogeneous objectives, thereby influencing the resulting Nash equilibrium.
Using a novel fixed-point argument, we prove the existence and uniqueness of an optimal intervention in the graphon setting, and show that it achieves near-optimal performance in large finite networks, again with explicit bounds on the convergence rate.
As an application, we study the special case of linear-quadratic objectives and exploit the spectral decomposition of the graphon operator to derive semi-explicit solutions for the optimal intervention. This spectral approach provides key insights into the design of optimal interventions in dynamic environments.



\end{abstract} 

\begin{description}
\item[Mathematics Subject Classification (2020):] 	91A07, 91A15, 91A43, 93E20 
\item[Keywords:] graphon games, network games, targeted interventions, central planner, Nash equilibrium, stochastic control
\end{description}

\section{Introduction}

Dynamic network games are  non-cooperative games in which players interact strategically through a dynamic system to optimize their objectives,  with their interdependencies specified by a connectivity network modeled as a graph. The players’ actions affect both the evolution of the environment and their own payoffs. These games provide a framework for modeling strategic interactions among agents in competitive real-world systems, including autonomous driving \cite{hang2020human},  real-time bidding \cite{sayedi2018real}, and dynamic production management \cite{leng2005game}. The rigorous mathematical analysis of   dynamic network games is particularly challenging due to the heterogeneity among players and their interactions, the scale of players involved, and the nonlinear feedback effects of players' actions on the environment.

The scaling limits of graph sequences, particularly for dense sequences, have garnered significant attention. \citet{lovasz2006limits} proved that dense graph sequences with converging subgraph densities converge to a natural limit known as a graphon, which is a symmetric measurable function $W:[0,1]^2\to [0,1]$. \citet{borgs2008convergent,borgs2012convergent} further substantiated the role of graphons as the appropriate limiting objects for dense graph sequences and introduced the cut distance, which characterizes the convergence of these sequences (see \citet{lovasz2012large}).  Lovász and Szegedy's seminal work gave rise to the concept of graphon games, which serve as infinite-population approximations to finite-player network games. In these games, a continuum of players interacts according to a graphon $W$, where $W(x,y)$ describes the interaction of the two infinitesimal players $x,y\in [0,1]$. \citet{parise2023graphon} demonstrated that Nash equilibria of static graphon games can approximate Nash equilibria of network games on large graphs, which are sampled from the graphon (see also \cite{carmona2022stochastic, parise2021analysis}). 

Dynamic network games present greater analytical challenges than their static counterparts, as each player's payoff depends not only on the actions of others but also on an evolving system state, whose dynamics are jointly affected by the actions of all players. Dynamic graphon games were first studied by \citet{gao2020linear}, who derived an approximate Nash equilibrium for the corresponding finite-player game on large graphs in the linear-quadratic case (see also \citet{aurell2022stochastic}). These games were studied in greater generality by \citet{bayraktar2023propagation} and \citet{tangpi2024optimal}, who established the convergence of the Nash equilibrium of network games to that of the limiting graphon game using propagation of chaos. Finally, a general class of linear-quadratic stochastic games with heterogeneous interactions, where both the goal functionals are non-Markovian, was solved by \citet{neuman2024stochastic}. 

Although significant progress has been made in analyzing network games and their graphon limits, challenges persist in regulating or intervening in economic behavior across large-scale networks. One major difficulty lies in the intervention problem faced by a central planner, which grows increasingly complex as the network expands. Also, the assumption that the central planner has full knowledge of the network’s structure is often unrealistic, as gathering precise network data can be prohibitively expensive or entirely infeasible due to privacy and proprietary restrictions. For static games, some work has been done in the linear-quadratic case by \citet{galeotti2020targeting}, who decomposed the intervention into principal components of the network and derived the optimal intervention in terms of the associated eigenvalues. A more general framework for interventions in static games with continuously differentiable and concave payoff functions was developed by \citet{parise2023graphon}. 

In this paper, we address the central planner’s intervention problem for a general class of dynamic, non-Markovian games on large networks. The main challenge arises from the heterogeneity of the players’ goal functionals, the network-based nature of their interactions, and the fact that the central planner’s objective depends implicitly on the players’ actions at equilibrium. Below, we outline our main contributions.

\textbf{Existence, uniqueness, and convergence of equilibria in the underlying game. }
We consider a general framework in which players’ utility functionals are heterogeneous, non-Markovian, continuously differentiable, and concave. We derive existence and uniqueness results for the Nash equilibrium in the graphon game (see Theorem~\ref{thm:NE}) and in the network game (see Corollary~\ref{cor:NE-network}). Moreover, we prove the convergence of the equilibrium in the network game to the one in the graphon game, both when a given graph sequence converges to the graphon in the cut norm (see Theorem~\ref{thm:convergence}) and when the graph sequence is sampled from the graphon (see~Theorem \ref{thm:convergence-sampled}), providing explicit bounds on the convergence rate in both cases. As described in Example \ref{example:U} and Remarks~\ref{rem:convergence} and \ref{rem:convergence-sampled}, those results extend the framework of non-Markovian graphon games from \cite{abijaber2023equilibrium,neuman2024stochastic} along with existence, uniqueness, and convergence results from \citet{aurell2022stochastic,bayraktar2023propagation,carmona2022stochastic,gao2020linear,lacker2023label,parise2023graphon,tangpi2024optimal}, and others. 

\textbf{Existence, uniqueness, and convergence of optimal interventions.}
We then turn our attention to the central planner’s intervention problem (see \eqref{eq:network-int-problem}, \eqref{eq:graphon-int-problem}). In this problem, the central planner optimizes the average welfare of all players by modifying their utility functionals, subject to limited resources and under the constraint that the players are at equilibrium. 
In the static formulation of the problem, the existence of an optimal graphon intervention can typically be derived via standard compactness and continuity arguments. By contrast, proving existence in the dynamic setting is considerably more involved, since the sets of admissible actions and interventions now reside in infinite-dimensional Hilbert spaces, remaining closed and bounded but no longer compact. 

In Theorem~\ref{thm:intervention-existence}, we present the first known result on the existence and uniqueness of an optimal intervention in a dynamic framework. We prove this result for the graphon intervention problem~\eqref{eq:graphon-int-problem} using a novel fixed-point argument. The main challenge lies in the fact that the central planner’s intervention influences the players’ actions and vice versa. To account for this interdependence, we introduce a product operator $\boldsymbol{P}$ that maps a pair $(\hat\theta, a)$, consisting of an intervention $\hat\theta$ and an action profile $a$, to the pair $\boldsymbol{P}(\hat\theta, a)$, which consists of the central planner’s best response to $a$ and the players’ best response to $\hat\theta$. By showing that $\boldsymbol{P}$ admits a unique fixed point, we establish the existence and uniqueness of an optimal intervention.
We then provide convergence results which show that the optimal intervention in the network game converges to the optimal intervention in the corresponding graphon game, again both when a given graph sequence converges to the graphon (see Theorem~\ref{thm:intervention-convergence}) and when the graph sequence is sampled from the graphon (see~Corollary~\ref{cor:intervention-convergence-sampled}), with explicit bounds on the convergence rates. These results extend the work of \citet{parise2023graphon} to dynamic games.

\textbf{Semi-explicit optimal interventions in the linear-quadratic setting.}
In the second part of the paper, we focus on a special case where players’ utility functionals are linear-quadratic. Motivated by \citet{galeotti2020targeting}, we solve the network intervention problem \eqref{eq:network-int-problem-LQ} semi-explicitly using principal components (see Theorem~\ref{thm:LQ}), allowing for a detailed characterization of the optimal intervention. We further demonstrate that this analysis becomes particularly valuable in the case of an infinite network, where the dimensionality of the central planner’s problem is substantially reduced. Specifically, the network intervention problem \eqref{eq:network-int-problem-LQ} scales with the population size, while the associated graphon intervention problem \eqref{eq:graphon-int-problem-LQ} can often be solved more efficiently. A key difference between the network and the graphon setting is that, in the infinite-player version, we work with a spectral decomposition of the graphon operator $\boldsymbol{W}$ on $L^2([0,1], \mathbb{R})$, rather than with one of the adjacency matrix $G^N\in [0,1]^{N\times N}$ of the finite network. Unlike in the finite-dimensional case, the spectrum of $\boldsymbol{W}$ may contain either a finite or infinite number of distinct eigenvalues, each with its own multiplicity. By leveraging the spectral properties of $\boldsymbol{W}$, we address these complexities and derive the optimal graphon intervention in semi-explicit form (see Theorem~\ref{thm:LQ-graphon}).

One of the main conclusions from this analysis is that, within each principal component, the optimal interventions are scalings of the underlying heterogeneity processes. This observation offers insights into the similarity between the optimal intervention and the eigenfunctions corresponding to the spectrum of the graphon operator (see Corollary~\ref{corollary:similarity-ratio-graphon}). Moreover, it yields explicit asymptotics for small and large budgets of the central planner (see Proposition~\ref{prop:convergence-graphon}) and demonstrates how these asymptotics approximate the optimal intervention (see Proposition~\ref{prop:convergence-rate-graphon}).

\paragraph{Organization of the paper.} In Section~\ref{sec:underlying-game}, we present the finite-player network game and the corresponding graphon game. We then establish existence and uniqueness results for both games, as well as convergence results that show how the network game equilibrium approximates the graphon game equilibrium. In Section~\ref{sec:interventions}, we introduce the central planner's intervention problem, prove existence and uniqueness of the optimal intervention, and present the corresponding convergence results. Finally, in Section~\ref{sec:LQ}, we present our results for the linear-quadratic case. Sections~\ref{sec:underlying-game-proofs}--\ref{sec:LQ-proofs} are dedicated to the proofs of our main results.

\section{The Underlying Game}\label{sec:underlying-game}
In this section we present a dynamic network game and the corresponding graphon game, where each player seeks to maximize a general utility functional. We derive results on the existence and uniqueness of the Nash equilibrium in both cases, and prove the convergence of the equilibria when the number of players tends to infinity.  
\subsection{The Network Game}\label{subsec:definition-finite-player-game}

Let $T>0$ denote a finite deterministic time horizon and let $N\in \mathbb{N}$ be the number of players in the system. We consider a network represented by a graph with symmetric adjacency matrix $G^N\in [0,1]^{N\times N}$, where $G^N_{ij}$ represents the levels of interaction  between players $i$ and $j$. We fix a filtered probability space $(\Omega, \mathcal F, \mathbb F:=(\mathcal F_t)_{0 \leq t \leq T}, \mathbb P)$ satisfying the usual conditions of right-continuity and completeness, and consider a dynamic stochastic game among the $N$ players. Each player $i \in \{1, \ldots, N\}$ selects their action $a^{i,N}$ from their set of admissible actions $\mathcal A^{i,N}$, which is a subset of
\be\label{eq:A}  
  \mathcal A:=\left\{\tilde{a}:\Omega\times [0,T]\to\R \, \Big| \, \tilde{a} \textrm{ is } \mathbb{F}\textrm{-progressively measurable, } \,\|\tilde a\|^2_{\mathcal{A}}:=\int_0^T \mathbb E\big[\tilde{a}_t^2 \big] dt  <\infty \right\},
\ee
where throughout, variables with a tilde denote elements in $\mathcal{A}$. Let $a^N:=(a^{1,N},\ldots, a^{N,N})$ be the vector of all actions and define the set of admissible action profiles
\be\label{eq:AadN}
\mathcal{A}_{ad}^N:=\prod_{i=1}^N\mathcal{A}^{i,N}.
\ee
Moreover, let $\langle\cdot,\cdot\rangle_{\mathcal{A}}$ denote the corresponding inner product on $\mathcal{A}$ given by
\be\label{eq:inner-product-A}
\big\langle  \tilde a,  \tilde b\big\rangle_{\mathcal{A}}:=\E\left[\int_0^T \tilde a_t \tilde b_tdt\right],\quad \tilde a, \tilde b\in\mathcal{A}.
\ee
Analogously, denote by  $\left\langle\cdot,\cdot\right\rangle_{\mathcal{A}^N}$ and $\|\cdot\|_{\mathcal{A}^N}$ the inner product and induced norm on $\mathcal{A}^N$. Consider a universal utility functional 
\be \label{eq:U}
U:\mathcal{A}\times\mathcal{A}\times\mathcal{A}\to\R.
\ee
Each player $i \in \{1, \ldots, N\}$ seeks to maximize their individual utility functional on $\mathcal{A}^{i,N}$ given by
\be \label{eq:Ui}
a^{i,N}\mapsto U\left(a^{i,N},z^{i,N}(a^N),\theta^{i,N}\right),
\ee
where the local aggregate
\be \label{eq:z^{i,N}}
z^{i,N}(a^N):=\frac{1}{N}\sum_{j=1}^NG^N_{ij}a^{j,N}
\ee
is defined as a weighted average of the other player's actions computed according to the heterogeneous interaction weights of the network $G^N$. The stochastic process $\theta^{i,N}\in\mathcal{A}$ incorporates heterogeneity in the utility functionals of different players. Denote by ${\theta}^N\in \mathcal{A}^N$ the vector of all heterogeneity processes. In line with \cite{parise2023graphon}, we denote this network game by $\mathcal{G}(\mathcal A_{ad}^N,U,\theta^N,G^N)$.
\begin{definition} \label{def:Nash-finite-player}
A vector of actions $\bar a^N\in\mathcal{A}_{ad}^N$ is called a Nash equilibrium of the network game $\mathcal{G}(\mathcal A_{ad}^N,U,\theta^N,G^N)$ if for every $i\in \{1,\ldots, N\}$ the action $\bar a^{i,N}$ satisfies 
\be \label{obj-n}
\bar a^{i,N}=\argmax_{\tilde {a}\in\mathcal{A}^{i,N}} \ U\left(\tilde{a}, z^{i,N}(\bar a^N),\theta^{i,N}\right).
\ee
\end{definition}

\begin{example}\label{example:U}
Notice that our formulation does not include state processes, and that the utility functional $U$ from \eqref{eq:U} takes as inputs stochastic processes in $\mathcal{A}$. This makes it quite versatile, including the following examples:
\begin{itemize}
\item[(i)] A simple class of utility functionals is given by 
$$
U(\tilde a,\tilde z,\tilde\theta)= \E\left[\int_0^T f_U(\tilde a_t,\tilde z_t,\tilde\theta_t)dt \right],
$$
where $f_U:\R^3\to\R$ is a function incorporating a running cost. In particular, this class includes the dynamic formulation of the static game studied by \citet{galeotti2020targeting}, which we will analyze in detail in Section~\ref{sec:LQ}.
\item[(ii)] Another example is the class of linear-quadratic utility functionals given by 
$$
U(\tilde a,\tilde z,\tilde\theta)=\langle\tilde a,\boldsymbol{A}_1\tilde a\rangle_\mathcal{A}+\langle\tilde a,\boldsymbol{A}_2\tilde z\rangle_\mathcal{A}+\langle\tilde z,\boldsymbol{A}_2\tilde a\rangle_\mathcal{A}+\langle\tilde z,\boldsymbol{A}_3\tilde z\rangle_\mathcal{A}+\langle\tilde a,\tilde \theta\rangle_\mathcal{A}+\langle\tilde z, \eta^*\rangle_\mathcal{A},
$$
where $\boldsymbol{A}_1,\boldsymbol{A}_2,\boldsymbol{A}_3$ are Volterra operators on $L^2([0,T],\R)$ and $\eta^*\in\mathcal{A}$ is a process incorporating common noise. Notably, this class allows for non-Markovian dynamics and was studied in detail in \cite{abijaber2023equilibrium,neuman2024stochastic}.
\item[(iii)] Given square-integrable Volterra kernels $K,L:[0,T]^2\to\R$ and a family of processes $(\eta^{i,N})_{i=1}^N\in\mathcal{A}$, assume that the players have state processes $X^{i,N}$ and local aggregate state processes $Z^{i,N}$ with linear dynamics given by 
\be\begin{aligned}\label{eq:Xi-Zi}
X^{i,N}_t&=\int_0^t K(t,s)X^{i,N}_sds+\int_0^t L(t,s)a^{i,N}_sds+\eta^{i,N}_t,\\
Z^{i,N}_t&=\frac{1}{N}\sum_{j=1}^N G_{ij}^NX^{j,N}_t,\quad\  i=1,\ldots, N.
\end{aligned}\ee
Then, the dynamics \eqref{eq:Xi-Zi} have a unique explicit solution given by
\be\begin{aligned}\label{eq:Xi-Zi-resolved}
X_t^i&=\int_0^t \left(L(t,s)-\big(R\star L\big)(t,s)\right)a^{i,N}_sds-\int_0^t R(t,s)\eta^{i,N}_sds+\eta^{i,N}_t,\\
Z_t^i&=\int_0^t \left(L(t,s)-\big(R\star L\big)(t,s)\right)z^{i,N}_s(a^N)ds-\frac{1}{N}\sum_{j=1}^NG_{ij}^N\left(\int_0^t R(t,s)\eta^{j,N}_sds-\eta^{j,N}_t\right),
\end{aligned}\ee
where $R:[0,T]^2\to\R$ denotes the resolvent of the kernel $(-K)$ and the $\star$-product $(R\star L)(t,s)$ is given by $\smash{\int_0^T R(t,u) L(u,s)du}$ (see \cite{gripenberg1990volterra}, Chapter~9.3, Theorem~3.6). Consider goal functionals of the classical form 
\be\label{eq:J^{i,N}}
J^{i,N}(a^{i,N})=\E\left[\int_0^Tf_J(t,X^{i,N}_t,Z^{i,N}_t, a^{i,N}_t)dt+g_J(X^{i,N}_T,Z^{i,N}_T)\right],
\ee
for a running cost $f_J:\R^4\to\R$ and a terminal cost $g_J:\R^2\to\R$. Then, using \eqref{eq:Xi-Zi-resolved}, the functionals $J^{i,N}$ in \eqref{eq:J^{i,N}} have an explicit representation $U\left(a^{i,N},z^{i,N}(a^N),\theta^{i,N}\right)$ for a universal function $U:\mathcal{A}^3\to\R$ as in \eqref{eq:U}, \eqref{eq:Ui} and suitably chosen heterogeneity processes $(\theta^{i,N})_{i=1}^N$. This shows that our framework aligns with the formulations of \citet{aurell2022stochastic,lacker2023label,tangpi2024optimal} and others, who study cost functionals which include state processes.
\end{itemize}
\end{example}

\subsection{The Graphon Game}\label{subsec:definition-infinite-player-game}
In line with \cite{aurell2022stochastic,neuman2024stochastic,tangpi2024optimal}, the corresponding infinite-player graphon game will be modeled by the following setup. We label the players amid the unit interval by $x\in [0,1]$. 
Let $\mathcal{B}_{[0,1]}$ be the Borel $\sigma$-algebra of $[0,1]$, and $\nu_{[0,1]}$ denote the Lebesgue measure on $[0,1]$. Let $(\Omega,\mathcal{F},\P)$ be the sample space and $([0,1],\mathcal{I},\nu)$ be a probability space that extends the Lebesgue measure space $([0,1],\mathcal{B}_{[0,1]},\nu_{[0,1]})$ as in Theorem~1 in \cite{sun2009individual}. 
We will consider a rich Fubini extension $([0,1]\times\Omega,\mathcal{I}\boxtimes \mathcal{F},\nu\boxtimes\P)$ of the standard product space $([0,1]\times\Omega,\mathcal{I}\otimes \mathcal{F},\nu\otimes\P)$ (see \cite{sun2006exact,sun2009individual} for an overview). Namely, by Theorem~1 in \cite{sun2009individual}, it is possible to define $\mathcal{I  }\boxtimes \mathcal{F}$-measurable processes $\theta: [0,1]\times\Omega \to L^2([0,T],\R)$ such that the random variables $(\theta^x)_{x\in [0,1]}$ are essentially pairwise independent, and for each $x\in [0,1]$, the process $\theta^x=(\theta_t^x)_{0\leq t\leq T}$ is a stochastic process in $L^2(\Omega\times [0,T],\R)$. Here, the processes $(\theta^x)_{x\in [0,1]}$ have to satisfy the technical condition that the map from $([0,1],\mathcal{I},\nu)$ to the space of Borel probability measures on $L^2([0,T],\R)\times\R$ which assigns to each $x\in [0,1]$ the distribution of $\theta^x$ is measurable. Throughout the rest of this subsection, denote by $\F\vcentcolon=(\mathcal{F}_t)_{0\leq t\leq T}$ the augmentation of the filtration generated by $(\theta^x)_{x\in [0,1]}$.
\begin{remark}
By Definition~2.2 in \cite{sun2006exact}, the Fubini property holds on the rich Fubini extension $([0,1]\times\Omega,\mathcal{I}\boxtimes \mathcal{F},\nu\boxtimes\P)$. That is, for any $\nu\boxtimes\P$-integrable function $f:[0,1]\times\Omega\to\R$ it holds that
\be
\int_{[0,1]\times\Omega}f(x,\omega)(\nu\boxtimes\P)(dx,d\omega)=\int_0^1\E[f(x)]\nu(dx)=\E\Big[\int_0^1f(x)\nu(dx)\Big].
\ee
Also see Lemma~2.3 in \cite{sun2006exact} for a generalized Fubini property. 
We usually denote $\nu(dx)=dx$ for ease of notation and tacitly employ the Fubini property without referring to this remark again.
\end{remark}

\begin{definition}\label{def:action-profile}
An action profile is a family $a=(a^x)_{x\in [0,1]}$ of actions $a^x\in\mathcal{A}$ such that the map $(x,t,\omega)\to a_t^x(\omega)$ is $\mathcal{I}\otimes\mathcal{B}([0,T])\otimes\mathcal{F}$-measurable.
Define the set of feasible action profiles as
\be\label{eq:Ainfty}
\mathcal{A}^\infty\vcentcolon=\left\{a-\textrm{action profile}\,\Big| \,\|a\|^2_{\mathcal{A}^\infty}:=\int_0^1\int_0^T\E\big[(a_t^x)^2\big]dtdx<\infty\right\},
\ee
and denote by $\langle\cdot,\cdot\rangle_{\mathcal{A}^\infty}$ the corresponding inner product.  
\end{definition}
We consider a dynamic stochastic game among a continuum of players, where each player $x\in [0,1]$ selects an action $a^x$  from their set of admissible actions $\mathcal{A}^x\subset\mathcal{A}$. 
Define the set of admissible action profiles as
\be\label{eq:Aadinfty}
\mathcal{A}_{ad}^\infty:=\left\{a\in\mathcal{A}^\infty\Big| a^x\in\mathcal{A}^x\textrm{ for all }x\in [0,1]\right\}.
\ee
While the local aggregate in the network game is defined in terms of the adjacency matrix $G^N$, a natural way to define interactions among infinitely many players is through a graphon, that is, a symmetric and measurable function $W:[0,1]^2\to[0,1]$. Here, $W(x,y)$ denotes the level of interaction between players $x$ and $y$. Given a graphon $W$ from the set of graphons
$$
\mathcal{W}_0:=\left\{W:[0,1]^2\to[0,1]\mid W\textrm{ is symmetric and measurable}\right\},
$$
define the interaction effects experienced by player $x$ as the local aggregate
\be \label{eq:graphon-operator}
z^x(a):=(\boldsymbol{W}a)(x):=\int_0^1W(x,y)a^ydy.
\ee
Here, $\boldsymbol{W}$ denotes the bounded linear operator on $L^2([0,1],\R)$ induced by the graphon $W$. Similar to the network game, the utility functional of player $x$ on $\mathcal{A}^x$ is given by 
\be \label{eq:Ux}
a^x\mapsto U\left(a^x,z^x(a),\theta^x\right),
\ee
where $\theta^x\in\mathcal{A}$ is a heterogeneity process such that $\theta\in\mathcal{A}^\infty$. 
We note that the utility functional \eqref{eq:Ux} in the graphon game has the same structure as the utility functional \eqref{eq:Ui} in the network game, as they both contain the same universal functional $U$ from \eqref{eq:U}. We denote the graphon game by $\mathcal{G}(\mathcal{A}_{ad}^\infty,U,\theta,W)$.
\begin{definition} \label{def:Nash-infinite-player}
An admissible action profile $(\bar{a}^{x})_{x\in [0,1]}\in\mathcal{A}_{ad}^\infty$ is called a Nash equilibrium of the graphon game $\mathcal{G}(\mathcal{A}_{ad}^\infty,U,\theta,W)$ if for every $x\in [0,1]$ the action $\bar a^x$ satisfies 
\be \label{graph-game}
\bar a^x=\argmax_{\tilde {a}\in\mathcal{A}^x} \ U\left(\tilde{a}, z^x(\bar a),\theta^x\right).
\ee
\end{definition}
 
\subsection{Correspondence between Network Games and Graphon Games}\label{subsec:correspondence}
For every $N\in\N$, assume that the augmentation of the filtration generated by $\{\theta^{i,N}\}_{i=1}^N$ from Section~\ref{subsec:definition-finite-player-game} is contained in the filtration $\F$ from Section~\ref{subsec:definition-infinite-player-game}.
We now demonstrate that any network game can be reformulated as a graphon game. In the network game, a Nash equilibrium is an $N$-tuple of processes in $\mathcal{A}$, whereas in the graphon game, a Nash equilibrium is an element in $\mathcal{A}^\infty$. To compare these two objects we introduce the uniform partition of $[0,1]$ given by
\be \label{eq:partition}
\mathcal{P}^N=\{\mathcal{P}_1^N,\ldots,\mathcal{P}_N^N\}, \quad
\mathcal{P}_i^N:=\begin{cases}
    [\tfrac{i-1}{N},\tfrac{i}{N}),\phantom{\tfrac{N-1}{N},1]}\hspace{-6mm}\textrm{for } 1\leq i\leq N-1,\\
    [\tfrac{N-1}{N},1],\phantom{\tfrac{i-1}{N},\tfrac{i}{N})}\hspace{-6mm}\textrm{for } i= N.
\end{cases}
\ee
The idea is to pair each player $i$ in the network game with the interval $\mathcal{P}_i^N\subset [0,1]$. Namely, for a family $a^N\in\mathcal{A}^N$ of actions $a^{i,N}\in \mathcal{A}$, define the corresponding step function action profile $a^{N}_{\operatorname{step}}\in \mathcal{A}^\infty$ by
\be\label{eq:step-function}
a^{x,N}_{\operatorname{step}}:=a^{i,N},\quad \textrm{for all } x\in\mathcal{P}_i^N,\  i=1,\ldots, N.
\ee
Similarly, the partition $\mathcal{P}^N$ allows to define for any network with symmetric adjacency matrix $G^N\in [0,1]^{N\times N}$ a corresponding step graphon $W_{G^N}:[0,1]^2\to [0,1]$ given by
\be \label{eq:step-graphon}
W_{G^N}(x,y)\vcentcolon=G_{ij}^N,\quad\textrm{for all } \,(x,y)\in\mathcal{P}_i^N\times\mathcal{P}_j^N,\ i,j=1,\ldots N.
\ee
\begin{proposition}\label{prop:correspondence}
A vector of actions $\bar a^N\in\mathcal{A}^N$ is a Nash equilibrium of the network game $\mathcal{G}(\mathcal A_{ad}^N,U,\theta^N,G^N)$ if and only if the corresponding step function action profile $a^{N}_{\operatorname{step}}\in \mathcal{A}^\infty$ defined in \eqref{eq:step-function} is a Nash equilibrium of the graphon game $\smash{\mathcal{G}(\mathcal{A}_{ad}^\infty,U,\theta^N_{\operatorname{step}},W_{G^N})}$ with action sets $\mathcal{A}^x:=\mathcal{A}^{i,N}$ for all $x\in\mathcal{P}_i^N$, step function heterogeneity profile $\smash{\theta^{N}_{\operatorname{step}}}$ corresponding to $\theta^N$ as in \eqref{eq:step-function}, and underlying step graphon $W_{G^N}$ corresponding to $G^N$ as in \eqref{eq:step-graphon}.
\end{proposition}
The proof of Proposition~\ref{prop:correspondence} is given in Section~\ref{sec:underlying-game-proofs}.
\subsection{Equilibrium Results for the Underlying Game}
Motivated by \citet{parise2023graphon}, we focus on continuously differentiable and strongly concave utility functionals to derive Nash equilibrium properties of the network game and the graphon game. Since our goal is to study interventions by a central planner, we are particularly interested in the existence of unique Nash equilibria.

\begin{assumption}\label{assum:U}
For all $\tilde z,\tilde\theta\in\mathcal{A}$, the utility functional $U(\tilde a,\tilde z,\tilde \theta)$ in \eqref{eq:U} is continuously G\^ateaux differentiable in $\tilde a$ and strongly concave in $\tilde a$ with strong concavity constant $\alpha_U>0$, that is, $\smash{U(\tilde a,\tilde z,\tilde \theta)+\frac{\alpha_U}{2}\|\tilde a\|^2_{\mathcal{A}}}$ is concave in $\tilde a$. Moreover, $\smash{\nabla_{\tilde a} U(\cdot,\tilde z,\tilde \theta)}$ is Lipschitz continuous in $\tilde z,\tilde\theta$ with constants $\ell_U,\ell_\theta$. 
\end{assumption}

\begin{assumption}\label{assum:graphon}
For each $x\in [0,1]$, the set of admissible actions $\mathcal{A}^x$ from \eqref{eq:Aadinfty} is nonempty, convex, and closed. Moreover, there exist $\tilde z_0, \tilde\theta_0 \in\mathcal{A}$ such that 
$$\big(\argmax_{\tilde{a}\in \mathcal{A}^x} U(\tilde{a}, \tilde z_0,\tilde\theta_0)\big)_{x\in [0,1]}\in\mathcal{A}^\infty.$$
\end{assumption}

For a graphon $W\in\mathcal{W}_0$, denote by $\lambda_{1}(\boldsymbol{W})$ the largest eigenvalue of the graphon operator $\boldsymbol{W}$ from \eqref{eq:graphon-operator}. Note that $\lambda_{1}(\boldsymbol{W})\in [0,1]$ (see \cite{avella2018centrality}, Lemma~1 and \cite{lovasz2006limits}, Chapter~7.5, equation~7.20).

\begin{theorem}\label{thm:NE}
Assume that Assumptions~\ref{assum:U} and \ref{assum:graphon} are satisfied.
\begin{itemize}
\item[(i)]  Fix a graphon $W\in\mathcal{W}_0$. If $\ell_U\cdot\lambda_{1}(\boldsymbol{W})<\alpha_U$, the graphon game $\mathcal{G}(\mathcal{A}_{ad}^\infty,U,\theta,W)$ admits a unique Nash equilibrium.
\item[(ii)]  Consequently, if $\ell_U<\alpha_U$, the graphon game $\mathcal{G}(\mathcal{A}_{ad}^\infty,U,\theta,W)$ admits a unique Nash equilibrium for every graphon $W\in\mathcal{W}_0$.
\end{itemize}
\end{theorem}

The proof of Theorem~\ref{thm:NE} is given in Section~\ref{sec:underlying-game-proofs}. By means of Proposition~\ref{prop:correspondence}, we can obtain an analogous result for network games.

\begin{assumption}\label{assum:network}
For each $i\in \{1,\ldots, N\}$, the set of admissible actions $\mathcal{A}^{i,N}$ from \eqref{eq:AadN} is nonempty, convex, and closed.
\end{assumption}
For a symmetric matrix $G^N\in[0,1]^{N\times N}$, denote by $\lambda_{1}(G^N)$ its largest eigenvalue.

\begin{corollary}
\label{cor:NE-network}
Let Assumptions~\ref{assum:U} and \ref{assum:network} be satisfied.
\begin{itemize}
    \item[(i)]  Fix a graph with symmetric adjacency matrix $G^N\in[0,1]^{N\times N}$. If $\ell_U\cdot\lambda_{1}(G^N)<\alpha_U \cdot N$, the network game $\mathcal{G}(\mathcal A_{ad}^N,U,\theta^N,G^N)$ admits a unique Nash equilibrium.
    \item[(ii)]  Consequently, if $\ell_U<\alpha_U$, the network game $\mathcal{G}(\mathcal A_{ad}^N,U,\theta^N,G^N)$ admits a unique Nash equilibrium for every graph with symmetric adjacency matrix $G^N\in[0,1]^{N\times N}$.
\end{itemize}
\end{corollary}
The proof of Corollary~\ref{cor:NE-network} is given in Section~\ref{sec:underlying-game-proofs}.

\subsection{Convergence Results for the Underlying Game}\label{subsec:convergence}

In this section, we address the convergence of the network game from Section~\ref{subsec:definition-finite-player-game} to the graphon game from Section~\ref{subsec:definition-infinite-player-game}, as the number of players approaches infinity. In previous work on graphon games, there have mainly been two approaches to this problem. First, one can start from a sequence of weighted graphs with adjacency matrices $(G^N)_{N\in\N}$ that converge to a graphon $W$ in a suitable sense, and prove the convergence of the corresponding equilibria (see \cite{bayraktar2023propagation,carmona2022stochastic,gao2020linear,lacker2023label,neuman2024stochastic,tangpi2024optimal}, among others). Second, one can fix a graphon game with underlying graphon $W$, sample from $W$ either weighted or simple graphs and thereby corresponding network games, and then show the convergence of the sampled network game equilibria to the graphon game equilibrium (see \cite{aurell2022stochastic,carmona2022stochastic,neuman2024stochastic,parise2023graphon}). While the first approach is more general and intuitive, the second approach yields strong convergence properties and does not rely on a predefined graph sequence. For the sake of completeness, we follow both approaches and give corresponding convergence results in Theorems~\ref{thm:convergence} and~\ref{thm:convergence-sampled}.

We start with the first approach, where a graph sequence is given a priori. An important ingredient to define the convergence of such graph sequences is the cut norm (see \cite{lovasz2012large}, Chapter~8.2).
\begin{definition}
Let $\mathcal{W}$ denote the space of all bounded symmetric measurable kernels $W:[0,1]^2\to\R$. For a kernel $W\in\mathcal{W}$, define its cut norm by
\be\label{eq:cut-norm}
\|W\|_\Box \vcentcolon= \sup_{S_1,S_2\subset [0,1]}\left|\int_{S_1}\int_{S_2} W(x,y)dxdy\right|,
\ee
where the supremum is taken over all Borel-measurable subsets $S_1,S_2$. If one identifies functions that are almost everywhere equal, the cut norm is indeed a norm.
\end{definition} 

The seminal works of \citet{lovasz2006limits} and \citet{borgs2008convergent,borgs2012convergent} employed the cut norm and introduced the related cut distance to characterize the convergence of dense graph sequences to graphons. Namely, given a sequence of graphs with adjacency matrices $(G^N)_{N\in\N}$, we say that they converge in cut norm to a graphon $W$ if and only if the corresponding step graphons $W_{G^N}$ defined in \eqref{eq:step-graphon} satisfy $\|W-W_{G^N}\|_\Box\to 0$. 

\begin{assumption}\label{assum:Ax}
There is a constant $M$ such that for each $x\in[0,1]$, the set of admissible actions $\mathcal{A}^x$ from \eqref{eq:Aadinfty} satisfies
\be \label{ad-set-a-m}
\mathcal{A}^x \subset \mathcal{A}_M:=\left\{ \tilde a \in \mathcal A \mid  \|\tilde a\|_\mathcal{A}   \leq M \right\}.
\ee
\end{assumption}
For a set $\mathcal{E}$, denote its infinite product indexed by $[0,1]$ by 
\be\label{eq:product}
\mathcal{E}^{[0,1]}:=\prod_{x\in [0,1]}\mathcal{E}.
\ee
\begin{theorem}\label{thm:convergence}
Consider a graphon game $\mathcal{G}((\mathcal{A}^0)^{[0,1]},U,\theta,W)$ in which the players have homogeneous action sets, that is, $\mathcal{A}^x=\mathcal{A}^0$ for all $x\in [0,1]$. Suppose that the game satisfies Assumptions~\ref{assum:U}, \ref{assum:graphon}, and \ref{assum:Ax} with $\ell_U\lambda_{1}(\boldsymbol{W})<\alpha_U$, and denote by $\bar a$ its unique Nash equilibrium. Consider a sequence of graphs with symmetric adjacency matrices $(G^N)_{N\in\N}$ that converges to $W$ in cut norm, that is, 
$$\|W-W_{G^N}\|_\Box\to 0,\quad \textrm{as }N\to\infty.$$ For each $N\in\N$, let ${\theta}^N\in \mathcal{A}^N$ be a vector of heterogeneity processes and $\smash{\theta^{N}_{\operatorname{step}}}$ be the step function heterogeneity profile corresponding to $\theta^N$ as in \eqref{eq:step-function}. Assume that 
$$
\|\theta-\theta^N_{\operatorname{step}}\|_{\mathcal{A}^\infty}\to 0,\quad \textrm{as }N\to\infty.
$$
Then, there exists $N_0\in\N$ such that the network game $\mathcal{G}((\mathcal{A}^0)^N,U,\theta^N,G^N)$  admits a unique Nash equilibrium $\bar a^N$ for all $N\geq N_0$. Moreover, it holds for all $N\geq N_0$ that
$$
\|\bar a-\bar a^N_{\operatorname{step}}\|_{\mathcal{A}^\infty}\leq C_W\|W-W_{G^N}\|_{\Box}^{1/2}+C_\theta\|\theta-\theta^N_{\operatorname{step}}\|_{\mathcal{A}^\infty}\xrightarrow{N\to \infty}0,
$$
where $\bar a^{N}_{\operatorname{step}}$ is the step function action profile corresponding to $\bar a^N$ as in \eqref{eq:step-function} and $$C_W:=\frac{\sqrt{8}\ell_UM}{\alpha_U-\ell_U\lambda_{1}(\boldsymbol{W})},\quad C_\theta:=\frac{\ell_\theta}{\alpha_U-\ell_U\lambda_{1}(\boldsymbol{W})}.$$
\end{theorem}

\begin{remark}\label{rem:convergence}
Theorem~\ref{thm:convergence} extends the convergence results of \citet{bayraktar2023propagation,carmona2022stochastic,gao2020linear,lacker2023label,tangpi2024optimal} and others to general (in particular non-Markovian) utility functionals as in \eqref{eq:U}, without requiring any assumptions on the underlying graphon. Additionally, compared to the previous works, it allows for much more general idiosyncratic and common noise, which can be incorporated through the heterogeneity processes and through the definition of the utility functional in \eqref{eq:U}, respectively. Furthermore, Theorem~\ref{thm:convergence} extends Theorem~4.4 from \cite{neuman2024stochastic} beyond the linear-quadratic case.
\end{remark}

The proof of Theorem~\ref{thm:convergence} is given in Section~\ref{sec:underlying-game-proofs}.

Next, we proceed with the second approach, where a random sequence of weighted or simple graphs is sampled from a prespecified graphon (see \cite{lovasz2012large}, Chapter~10.1). Consequently, network games can thereby be sampled from a given graphon game.
\begin{definition} \label{def:sampling}
Fix a graphon $W\in\mathcal{W}_0$, a number of desired nodes $N\in\N$, and let $(x_1,\dots, x_N)$ be an ordered $N$-tuple of independent uniform random points from $[0,1]$. We define random graphs sampled according to the following two sampling procedures:
\begin{enumerate}[label=(P{{\arabic*}})]
    \item\label{P1}  the weighted graph $G_w^N(W)$ obtained by taking $N$ isolated nodes $i\in\{1,\ldots , N\}$ and adding undirected edges with weights $W(x_i,x_j)$ between nodes $i$ and $j$ for $i,j=1,\ldots N$,
    \item\label{P2}  the simple graph $G_s^N(W)$ obtained by taking $N$ isolated nodes $i\in\{1,\ldots , N\}$ and adding undirected edges between nodes $i$ and $j$ at random with probability $\kappa_NW(x_i,x_j)$ for $1\leq i< j \leq N$, where $(\kappa_N)_{N\geq 1}\subset (0,1]$ is a sequence of density parameters.   
\end{enumerate}
\begin{remark}\label{rem:sampling}
We will always assume that the sampling is carried out on another probability space $(\Omega',\mathcal{F}',\Q)$ independently of the randomness in the network game and the graphon game modeled by the probability measure $\P$.
\end{remark}
\begin{remark}\label{rem:correspondence-with-kappa_N}
In Definition~\ref{def:sampling}, the expected number of edges per node in $G_s^N(W)$ is of order $\kappa_NN$. In line with \cite{neuman2024stochastic,parise2023graphon}, we will assume later that $\smash{\tfrac{\log N}{\kappa_NN}\to 0}$ as $N \rightarrow \infty$, which allows for graph sequences that gradually become sparser for large $N$. Here, the introduction of a density parameter $\kappa_N$ only affects how the sampled network games are obtained from the graphon, without affecting the graphon game limit itself. As a consequence, one has to slightly adjust the local aggregate $z^{i,N}(a^N)$ in \eqref{eq:z^{i,N}} for the network game on the sampled simple graph $G_s^N(W)$, in order to account for
the fact that the number of edges per node may now grow sublinearly,
\be\label{eq:aggregate-sparse}
z^{i,N}(a^N):=\frac{1}{\kappa_NN}\sum_{j=1}^NG_s^N(W)_{ij}a^{j,N},\quad i=1,\ldots,N.
\ee
Note that Proposition~\ref{prop:correspondence} still holds in the modified network game with $G^N$ replaced by $\kappa_N^{-1}G_s^N(W)$ and $W_{G^N}$ replaced by $\kappa_N^{-1}W_{G_s^N(W)}$.
\end{remark}
\end{definition}
\begin{assumption}\label{assum:Lipschitz-graphon}
Assume that the graphon $W$ is blockwise Lipschitz continuous, that is, there exists a constant $L$ and a finite partition $\{I_1,\ldots,I_{K+1}\}$ of $[0,1]$ such that for any $1\leq k,l\leq K+1$, any set $I_k\times I_l$, and any pair $(x,y),(x',y')\in I_k\times I_l$ it holds that
$$
|W(x,y)-W(x',y')|\leq L(|x-x'|+|y-y'|).
$$
\end{assumption}
\begin{theorem}\label{thm:convergence-sampled}
Consider a graphon game $\mathcal{G}((\mathcal{A}^0)^{[0,1]},U,\theta,W)$ in which the players have homogeneous action sets, that is, $\mathcal{A}^x=\mathcal{A}^0$ for all $x\in [0,1]$. Suppose that the game satisfies Assumptions~\ref{assum:U}, \ref{assum:graphon}, \ref{assum:Ax}, and \ref{assum:Lipschitz-graphon} with $\ell_U<\alpha_U$, and denote by $\bar a$ its unique Nash equilibrium. Let $(\kappa_N)_{N\in\N }\subset(0,1]$ be a sequence of density parameters satisfying $\smash{\tfrac{\log N}{\kappa_NN}\to 0}$ as $N \rightarrow \infty$. For any $N\in\N$, let ${\theta}^N\in \mathcal{A}^N$ be a vector of heterogeneity processes, and assume that 
$$
\|\theta-\theta^N_{\operatorname{step}}\|_{\mathcal{A}^\infty}\to 0,\quad \textrm{as }N\to\infty.
$$
Then, the following statements hold for the sampled network games:
\begin{itemize}
\item[(i)]  For weighted graphs as in \ref{P1}, the sampled network game
$\mathcal{G}((\mathcal{A}^0)^N,U,\theta^N,G_w^N(W))$ admits a unique Nash equilibrium $\bar a^N$ for every $N\in\N$. Moreover, for any $0<\dl<e^{-1}$, it holds with $\Q$-probability at least $1-\dl$ that
\be\label{eq:sampled-convergence-1}
    \left\|\bar a-\bar a^N_{\operatorname{step}}\right\|_{\mathcal{A}^\infty}=\mathcal{O}\left(\left(\frac{\log(N/\dl)}{N}\right)^\frac{1}{4}\vee\|\theta-\theta^N_{\operatorname{step}}\|_{\mathcal{A}^\infty}\right).
\ee
In particular, the left-hand side of \eqref{eq:sampled-convergence-1} converges $\Q$-almost-surely to 0 as $N\to\infty$.

\item[(ii)]  For simple graphs as in \ref{P2}, for any $0<\dl<e^{-1}$, there exists an $N_\dl\in\N$ such that the sampled network game $\mathcal{G}((\mathcal{A}^0)^N,U,\theta^N,\kappa_N^{-1}G_s^N(W))$ admits a unique Nash equilibrium $\bar b^N$ with $\Q$-probability at least $1-2\dl$ for all $N\geq N_\dl$. Moreover, it holds with $\Q$-probability at least $1-2\dl$ that
\be\label{eq:sampled-convergence-2}
    \left\|\bar b-\bar b^N_{\operatorname{step}}\right\|_{\mathcal{A}^\infty}=\mathcal{O}\left(\left(\frac{\log(N/\dl)}{N}\right)^\frac{1}{4}\vee\left(\frac{\log(N/\dl)}{\kappa_NN}\right)^\frac{1}{2}\vee\|\theta-\theta^N_{\operatorname{step}}\|_{\mathcal{A}^\infty}\right).
\ee
In particular, the left-hand side of \eqref{eq:sampled-convergence-2} converges $\Q$-almost-surely to 0 as $N\to\infty$.
\end{itemize}
\end{theorem}
The proof of Theorem~\ref{thm:convergence-sampled} is given in Section~\ref{sec:underlying-game-proofs}.
\begin{remark}\label{rem:sampling-option}
Note that Theorem~\ref{thm:convergence-sampled} allows the choice of heterogeneity processes that are sampled from the heterogeneity profile $\theta$, as assumed in \cite{parise2023graphon}. Namely, assume that $\theta^N=(\theta^{x_i})_{i=1}^N$, where $(x_1,\ldots,x_N)$ is an ordered $N$-tuple of independent uniform points from $[0,1]$ as in Definition~\ref{def:sampling}.
In that case, given that $\theta$ is 
sufficiently regular, the convergence rate of $\|\theta-\theta^N_{\operatorname{step}}\|_{\mathcal{A}^\infty}$ can be bounded. In particular, the right-hand side of \eqref{eq:sampled-convergence-1} reduces to its first term and the right-hand side of \eqref{eq:sampled-convergence-2} reduces to its first two terms. This generalizes Theorem~2 of \citet{parise2023graphon} to the dynamic setting.
\end{remark}

\begin{remark}\label{rem:convergence-sampled}
Theorem~\ref{thm:convergence-sampled} additionally extends the convergence results of \citet{aurell2022stochastic} and \citet{carmona2022stochastic} to general (in particular non-Markovian) utility functionals as in \eqref{eq:U}, while assuming blockwise Lipschitz continuity of the underlying graphon. Finally, it also generalizes Theorem~4.13 from \cite{neuman2024stochastic} beyond the linear-quadratic setting.
\end{remark}

\section{Targeted Interventions}\label{sec:interventions}
Motivated by \citet{galeotti2020targeting} and \citet{parise2023graphon}, we now introduce a second layer on top of the underlying game. Namely, we assume that a central planner seeks to maximize the average utility of the population at equilibrium through so-called targeted interventions.
\subsection{Network and Graphon Interventions} 
In the $N$-player network game from Section~\ref{subsec:definition-finite-player-game}, a targeted intervention is a modification of each player's heterogeneity parameter $\theta^{i,N}$ to $\smash{\theta^{i,N}+\hat{\theta}^{i,N}}$ in \eqref{eq:Ui}, subject to a budget constraint, yielding the altered utility functional 
$$
U\left(a^{i,N},z^{i,N}(a^N),\theta^{i,N}+\hat{\theta}^{i,N}\right).
$$
The intervention is assumed to happen before the game is played, so that the players choose their actions with respect to the modified parameters. In line with the static interventions framework \cite{parise2023graphon}, given a budget $C_B>0$ for the cost of intervention, the central planner's optimization problem in the network game can be formulated as follows:
\be\begin{aligned}\label{eq:network-int-problem}
    \bar\theta^N\in&\argmax_{\hat{\theta}^N\in\mathcal{A}^N}\ T^N(\hat{\theta}^N)= \argmax_{\hat{\theta}^N\in\mathcal{A}^N}\frac{1}{N}\sum_{i=1}^N U\big(\bar a^{i,N},\bar z^{i,N},\theta^{i,N}+\hat{\theta}^{i,N}\big),\\
    \quad \textrm{s.t.}\quad & \bar a^N\textrm{ is a Nash equilibrium of }\mathcal{G}(\mathcal A_{ad}^N,U,\theta^N+\hat{\theta}^N,G^N),\quad \bar z^{N}=\frac{1}{N} G^N\bar a^N,\\
    & \frac{1}{N}\|\hat{\theta}^N\|^2_{\mathcal{A}^N}\leq C_B .
\end{aligned}\ee
Similarly, in the graphon game from Section~\ref{subsec:definition-infinite-player-game}, a targeted intervention is a modification of each player's heterogeneity parameter $\theta^{x}$ to $\theta^x+\hat{\theta}^x$ in \eqref{eq:Ux}, subject to a budget constraint, yielding the altered utility functional 
$$
U\left(a^x,z^x(a),\theta^x+\hat{\theta}^x\right).
$$
Consequently, the central planner's optimization problem in the graphon game can be formulated as follows:
\be\begin{aligned}\label{eq:graphon-int-problem}
    \bar\theta\in&\argmax_{\hat{\theta}\in\mathcal{A}^\infty}\ T(\hat{\theta})= \argmax_{\hat{\theta}\in\mathcal{A}^\infty}\int_0^1 U\big(\bar a^x_{\hat{\theta}},\bar z^x_{\hat{\theta}},\theta^x+\hat{\theta}^x\big)dx,\\
    \quad \textrm{s.t.}\quad & \bar a_{\hat{\theta}}\textrm{ is a Nash equilibrium of }\mathcal{G}(\mathcal A_{ad}^\infty,U,\theta+\hat{\theta},W),\quad \bar z_{\hat{\theta}}=\boldsymbol{W} \bar a_{\hat{\theta}}, \\
    & \|\hat{\theta}\|^2_{\mathcal{A}^\infty}\leq C_B.
\end{aligned}\ee
\subsection{Existence and Uniqueness Results for Interventions}
In the static formulation of \eqref{eq:graphon-int-problem} (see Section~6 of \cite{parise2023graphon}), existence of an optimal graphon intervention can usually be derived via standard compactness and continuity arguments. By contrast, proving existence in the dynamic setting is considerably more involved, since the sets of admissible actions and interventions now live in infinite-dimensional Hilbert spaces. Those sets remain closed and bounded, but are no longer compact. One approach to obtain an existence result nonetheless is to exploit their weak compactness, which however requires very restrictive assumptions on $U$ in \eqref{eq:U} and is therefore not followed here. Instead, we introduce a novel fixed point approach that establishes the existence and uniqueness of a solution to the graphon intervention problem \eqref{eq:graphon-int-problem} under a less restrictive assumption.

Recall the definition \eqref{ad-set-a-m} of the sets $\mathcal{A}_M\subset\mathcal{A}$ from Assumption~\ref{assum:Ax}.
\begin{assumption}\label{assum:U2}
The utility functional $U(\tilde a,\tilde z,\tilde \theta)$ in \eqref{eq:U} is uniformly bounded from above on \mbox{$\mathcal{A}_M\times \mathcal{A}_M\times\mathcal{A}$}. There is an $\ell_0$, such that for any $M'$, $U(\tilde a,\tilde z,\tilde \theta)$ is Lipschitz continuous in $\tilde \theta$ on $\mathcal{A}_M\times \mathcal{A}_M\times\mathcal{A}_{M'}$ with Lipschitz constant $\ell_{M'}=\ell_0(1+M')$. 
For all $\tilde a,\tilde z\in\mathcal{A}_M$, $U(\tilde a,\tilde z,\tilde \theta)$ is continuously G\^ateaux differentiable in $\tilde \theta$ and strongly concave in $\tilde \theta$ with strong concavity constant $\beta_U>0$, that is, $\smash{U(\tilde a,\tilde z,\tilde \theta)+\frac{\beta_U}{2}\|\tilde \theta\|^2_{\mathcal{A}}}$ is concave in $\tilde \theta$. Moreover, $\nabla_{\tilde \theta} U(\tilde a,\tilde z,\cdot)$ is Lipschitz continuous in $\tilde a,\tilde z$ with constants $\ell_a,\ell_z$, and 
\be \label{gt}
\max \Big(\frac{\ell_a+\ell_z\lambda_1(\boldsymbol{W})}{\beta_U},\frac{\ell_\theta}{\alpha_U-\ell_U\lambda_{1}(\boldsymbol{W})}\Big)\leq 1.
\ee
\end{assumption}

\begin{theorem}\label{thm:intervention-existence}
Consider a graphon game $\mathcal{G}((\mathcal{A}^0)^{[0,1]},U,\theta,W)$ in which the players have homogeneous action sets, that is, $\mathcal{A}^x=\mathcal{A}^0$ for all $x\in [0,1]$. Suppose that the game  satisfies Assumptions~\ref{assum:U}, \ref{assum:graphon}, \ref{assum:Ax}, and \ref{assum:U2} with $\ell_U\cdot\lambda_1(\boldsymbol{W})<\alpha_U$.
Then, the graphon intervention problem \eqref{eq:graphon-int-problem} admits a solution $\bar\theta\in\mathcal{A}^\infty$. If the inequality in Assumption~\ref{assum:U2} is strict, the optimal intervention is unique.
\end{theorem}
The proof of Theorem~\ref{thm:intervention-existence} is given in Section~\ref{sec:interventions-proofs}. 

\begin{remark}
To the best of our knowledge, Theorem~\ref{thm:intervention-existence} is the first result on the existence and uniqueness of an optimal intervention in a dynamic framework. We prove it in Section~\ref{sec:interventions-proofs} using a novel fixed point argument that leverages the strong concavity of $U$ in $\tilde\theta$ from Assumption~\ref{assum:U2}. The main challenge here lies in the fact that the central planner's intervention influences the players' actions in \eqref{eq:graphon-int-problem} and vice versa. To account for this interdependence, we will introduce a product operator $\boldsymbol{P}$ on a subset of $\mathcal{A}^\infty\times\mathcal{A}^\infty$ that maps a pair $(\hat\theta,a)$ consisting of an intervention and an action profile to the pair $\boldsymbol{P}(\hat\theta,a)$ consisting of the central planner's best response to $a$ and the players' best response to $\hat\theta$. By showing that $\boldsymbol{P}$ admits a (unique) fixed point, we then obtain the existence (and uniqueness) of an optimal intervention (see Section~\ref{sec:interventions-proofs} for details). This approach is inspired by the proof of Theorem~\ref{thm:NE}, where the strong concavity of $U$ in $\tilde a$ from Assumption~\ref{assum:U} is exploited. In particular, this explains the resemblance of Assumptions~\ref{assum:U} and \ref{assum:U2}.
\end{remark}

\subsection{Convergence Results for Interventions}\label{subsec:int-convergence}
As noted in \cite{parise2023graphon}, the network intervention problem \eqref{eq:network-int-problem} scales with the population size $N$, becoming gradually more computationally expensive. In contrast, although solving the graphon intervention problem \eqref{eq:graphon-int-problem} can be costly in general, in many cases it can actually be solved more efficiently than the network intervention problem, for instance when the graphon is of finite rank, that is, the graphon has only finitely many nonzero eigenvalues (see Propositions~1 and 2 in \cite{parise2023graphon}).  Therefore, given a sequence of graphs converging to a graphon such as in the settings of Theorems~\ref{thm:convergence} and \ref{thm:convergence-sampled}, it is desirable to quantify how well the corresponding graphon intervention approximates the network interventions. Such a result is established in the following theorem for a given graph sequence. We consider a sequence of graphs with symmetric adjacency matrices $(G^N)_{N\in\N}$ that converges to $W$ in cut norm, that is, 
\be \label{eq:convergence-ass-1}  
\|W-W_{G^N}\|_\Box\to 0,\quad \textrm{as }N\to\infty.
\ee 
For each $N\in\N$, let ${\theta}^N\in \mathcal{A}^N$ be a vector of heterogeneity processes and assume that 
\be \label{eq:convergence-ass-2}  
\|\theta-\theta^N_{\operatorname{step}}\|_{\mathcal{A}^\infty}\to 0,\quad \textrm{as }N\to\infty,
\ee
where $\theta\in\mathcal{A}^\infty$ is a fixed heterogeneity profile. 
Throughout this section we assume that \eqref{eq:convergence-ass-1}  and \eqref{eq:convergence-ass-2}  hold. 
Given an optimal graphon intervention $\bar\theta\in\mathcal{A}^\infty$, we define the approximate network intervention candidate $\bar\theta^N_W$ by
\be\label{eq:approx-intervention}
(\bar\theta^N_W)_i:=\frac{\bar\theta (\frac{i}{N})}{\gamma^N},\quad i=1,\ldots, N,
\ee
where $\gamma^N$ is a normalizing constant ensuring that $\|\bar\theta^N_W\|_{\mathcal{A}^N}^2=\|\bar\theta\|_{\mathcal{A^\infty}}^2$, so that in particular the budget constraint from \eqref{eq:network-int-problem} is satisfied. Recall that the set $\mathcal A_{M}$ was defined in \eqref{ad-set-a-m}. 

\begin{theorem}\label{thm:intervention-convergence}
Consider a graphon game $\mathcal{G}((\mathcal{A}^0)^{[0,1]},U,\theta,W)$ in which the players have homogeneous action sets, that is, $\mathcal{A}^x=\mathcal{A}^0$ for all $x\in [0,1]$. Suppose that the game satisfies Assumptions~\ref{assum:U}, \ref{assum:graphon}, and \ref{assum:Ax} with $\ell_U\lambda_{1}(\boldsymbol{W})<\alpha_U$. Moreover, assume that the utility functional $U(\tilde a,\tilde z,\tilde \theta)$ is jointly Lipschitz in $(\tilde a,\tilde z,\tilde\theta)$. Suppose that there exists a solution $\bar\theta$ to the graphon intervention problem \eqref{eq:graphon-int-problem}, such that $\bar\theta$ satisfies $\bar\theta^x\in \mathcal{A}_{\bar\theta_{\operatorname{max}}}$ for all $x\in[0,1]$ and a constant $\bar\theta_{\operatorname{max}}$, and there exist a constant $L_{\bar\theta}$ and a finite partition $\{I_1,\ldots,I_{K_{\bar\theta}+1}\}$ of $[0,1]$ such that for any $1\leq k\leq K_{\bar\theta}+1$ and $x,x'\in I_k$ it holds that $\|\bar\theta^x-\bar\theta^{x'}\|_{\mathcal{A}}\leq L_{\bar\theta}|x-x'|$. Then the following hold:
\begin{itemize} 
\item[(i)]  There is an $N_0\in\N$ such that the network game $\mathcal{G}((\mathcal{A}^0)^N,U,\theta^N+\hat{\theta}^N,G^N)$ admits a unique Nash equilibrium for all $\hat\theta^N\in\mathcal{A}^N$ and all $N\geq N_0$. 

\item[(ii)]   Let $T^N_{\operatorname{opt}}$ be the optimal average utility of the population at equilibrium defined in \eqref{eq:network-int-problem}. For all $N\geq N_0$ it holds that
\be
T^N_{\operatorname{opt}}-T^N(\bar\theta^N_W)=\mathcal{O}\left(\|W-W_{G^N}\|_{\Box}^{1/2}\vee\|\theta-\theta_{\operatorname{step}}^N\|_{\mathcal{A}^\infty}\vee \frac{1}{\sqrt{N}}\right).
\ee
In particular, it holds that $T^N(\bar\theta^N_W)\to T^N_{\operatorname{opt}}$ as $N\to\infty$.
\end{itemize} 
\end{theorem}
The proof of Theorem~\ref{thm:intervention-convergence} is given in Section~\ref{sec:interventions-proofs}. 

As discussed in Section~\ref{subsec:convergence}, it is also possible to sample a random sequence of weighted or simple graphs from a prespecified graphon (see Definition~\ref{def:sampling}). The following corollary shows that the optimal graphon intervention yields nearly optimal network interventions in this setting as well.
\begin{corollary}\label{cor:intervention-convergence-sampled}
Consider a graphon game $\mathcal{G}((\mathcal{A}^0)^{[0,1]},U,\theta,W)$. Suppose that the game satisfies Assumptions~\ref{assum:U}, \ref{assum:graphon}, \ref{assum:Ax}, and \ref{assum:Lipschitz-graphon} with $\ell_U<\alpha_U$. Moreover, assume that the utility functional $U(\tilde a,\tilde z,\tilde \theta)$ is jointly Lipschitz in $(\tilde a,\tilde z,\tilde\theta)$, and that there exists an optimal graphon intervention $\bar\theta$ as in Theorem~\ref{thm:intervention-convergence}. Let $\bar\theta^N_W$ in \eqref{eq:approx-intervention} be the approximate network intervention with respect to $\bar \theta$.
\begin{itemize}
\item[(i)]  For weighted graphs as in \ref{P1}, the network game
$\mathcal{G}((\mathcal{A}^0)^N,U,\theta^N+\hat\theta^N,G_w^N(W))$ admits a unique Nash equilibrium for all $\hat\theta^N\in\mathcal{A}^N$ and all $N\in\N$. 
For every $0<\dl<e^{-1}$ with $\Q$-probability at least $1-\dl$ it holds,  
\be\label{eq:int-convergence1}
T^N_{\operatorname{opt}}-T^N(\bar\theta^N_W)=\mathcal{O}\Big(\Big(\frac{\log(N/\dl)}{N}\Big)^\frac{1}{4}\vee\|\theta-\theta^N_{\operatorname{step}}\|_{\mathcal{A}^\infty}\Big).
\ee
In particular, the left-hand side of \eqref{eq:int-convergence1} converges $\Q$-almost-surely to 0 as $N\to\infty$.

\item[(ii)]  For simple graphs as in \ref{P2}, for any $0<\dl<e^{-1}$, there exists an $N_\dl\in\N$ such that the network game $\mathcal{G}((\mathcal{A}^0)^N,U,\theta^N+\hat\theta^N,G_s^N(W))$ admits a unique Nash equilibrium with $\Q$-probability at least $1-2\dl$ for all $\hat\theta^N\in\mathcal{A}^N$ and for all $N\geq N_\dl$. 
It holds with $\Q$-probability at least $1-2\dl$ that
\be\label{eq:int-convergence2}
T^N_{\operatorname{opt}}-T^N(\bar\theta^N_W)=\mathcal{O}\Big(\Big(\frac{\log(N/\dl)}{N}\Big)^\frac{1}{4}\vee\|\theta-\theta^N_{\operatorname{step}}\|_{\mathcal{A}^\infty}\Big).
\ee
In particular, the left-hand side of \eqref{eq:int-convergence2} converges $\Q$-almost-surely to 0 as $N\to\infty$.
\end{itemize}
\end{corollary}
Corollary~\ref{cor:intervention-convergence-sampled} follows from Theorems~\ref{thm:convergence-sampled} and \ref{thm:intervention-convergence} and it is proved in Section~\ref{sec:interventions-proofs}.
\begin{remark}
Recall that the sequence of density parameters $(\kappa_{N})_{N\in\N}$ was introduced in Definition~\ref{def:sampling} and Remark~\ref{rem:correspondence-with-kappa_N}. For the sake of simplicity, we set $\kappa_N\equiv 1$ in Corollary~\ref{cor:intervention-convergence-sampled}(ii), since the introduction of density parameters requires redefining the sampled network intervention problem in \eqref{eq:network-int-problem} accordingly. Nonetheless, assuming that the problem is defined properly and that $\smash{\tfrac{\log N}{\kappa_NN}\to 0}$, such an extended analysis yields almost sure convergence, where with $\Q$-probability at least $1-2\dl$,
\be\label{eq:int-convergence3}
T^N_{\operatorname{opt}}-T^N(\bar\theta^N_W)=\mathcal{O}\Big(\Big(\frac{\log(N/\dl)}{N}\Big)^\frac{1}{4}\vee\Big(\frac{\log(N/\dl)}{\kappa_NN}\Big)^\frac{1}{2}\vee\|\theta-\theta^N_{\operatorname{step}}\|_{\mathcal{A}^\infty}\Big).
\ee 
Moreover, notice that Corollary~\ref{cor:intervention-convergence-sampled} allows the use of sampled heterogeneity processes in the sense of Remark \ref{rem:sampling-option}. In that case, the right-hand sides of \eqref{eq:int-convergence1} and \eqref{eq:int-convergence2} reduce to their first term and the right-hand side of \eqref{eq:int-convergence3} reduces to its first two terms. This generalizes Theorem~3 of \citet{parise2023graphon} to the dynamic setting. 
\end{remark}

\section{The Linear-Quadratic Case}\label{sec:LQ}
We now consider a linear-quadratic special case of our model, which extends the spectral intervention theory from \citet{galeotti2020targeting} to both the dynamic and the infinite-player setting. Linear-quadratic utility functionals make it possible to solve the underlying network and graphon games explicitly, which in turn allows us to derive the corresponding optimal interventions in semi-explicit form (see Theorems~\ref{thm:LQ} and \ref{thm:LQ-graphon}). These theorems characterize the optimal intervention in terms of the spectrum of the graph's adjacency matrix and the graphon operator, respectively.  One of the main conclusions from this analysis is that in each principal component, the optimal intervention is a scaling of the underlying heterogeneity processes, as shown explicitly in \eqref{eq:optimal-zeta-LQ} and \eqref{eq:optimal-zeta-graphon-LQ}. This observation yields insights on the similarity between the optimal intervention and the eigenfunctions corresponding to the spectrum (see Corollary~\ref{corollary:similarity-ratio-graphon}). Moreover, it provides explicit asymptotics for small and large budgets (see Proposition~\ref{prop:convergence-graphon}) and demonstrates how these asymptotics approximate the optimal intervention (see Proposition~\ref{prop:convergence-rate-graphon}).

\subsection{The Network Game}
\paragraph{Model setup.} We start with the network game. Consider the setup from Section~\ref{subsec:definition-finite-player-game} and assume that the universal utility functional $U$ in \eqref{eq:U} is of the specific form 
\be\label{eq:U-LQ}\begin{aligned}
U(\tilde a,\tilde z,\tilde\theta)&= \E\left[\int_0^T \tilde a_t\left(\tilde\theta_t+\beta\tilde z_t\right)-\frac{1}{2}\tilde a_t^2\,dt \right]+P(\tilde z,\tilde\theta)\\
&=\left\langle\tilde a,\tilde\theta+\beta\tilde z\right\rangle_{\mathcal{A}}-\frac{1}{2}\left\|\tilde a\right\|^2_{\mathcal{A}}+P(\tilde z,\tilde\theta),
\end{aligned}\ee
for a constant $\beta \in \mathbb{R}$ and a functional $P:\mathcal{A}\times\mathcal{A}\to\R$. 
Each player $i \in \{1, \ldots, N\}$ seeks to maximize their individual utility functional on $\mathcal{A}^{i,N}:=\mathcal{A}$ given by
\be \label{eq:Ui-LQ}
a^{i,N}\mapsto \left\langle a^{i,N},\theta^{i,N}+\beta z^{i,N}(a^N)\right\rangle_{\mathcal{A}}-\frac{1}{2}\| a^{i,N}\|^2_{\mathcal{A}}+P\left(z^{i,N}(a^N),\theta^{i,N}\right),
\ee 
where the local aggregate $z^{i,N}(a^N)$ is defined in \eqref{eq:z^{i,N}}, and $G^N\in [0,1]^{N\times N}$ is symmetric. Here, the inner product in \eqref{eq:Ui-LQ} represents the private marginal returns of player $i$ and can be broken down into two parts. The first part $\langle a^{i,N},\theta^{i,N}\rangle_{\mathcal{A}}$ of $i$'s marginal return is independent of the others' actions, and is called $i$'s standalone return. The heterogeneity process $\theta^{i,N}\in\mathcal{A}$ quantifies how much player $i$ benefits from increasing their action $a^{i,N}$. The second part $\beta \langle a^{i,N},z^{i,N}(a^N)\rangle_{\mathcal{A}}$ of $i$'s marginal return is the contribution of the others' actions. Here $G^N_{ij}$ measures strength of interaction between players $i$ and $j$, where we assume that $G^N_{ii}=0$ for all $i=1,\ldots, N$. If $\beta >0$, then actions are strategic complements, that is, they mutually reinforce one another, and if $\beta <0$, then actions are strategic substitutes, that is, they mutually offset one another. Moreover, the term $-\frac{1}{2}\|a^{i,N}\|^2_{\mathcal{A}}$ captures the private costs of player $i$'s action. Finally, the term $P(z^{i,N}(a^N),\theta^{i,N})$ incorporates pure externalities, that is, spillovers due to neighbors' actions and idiosyncratic noise, which do not depend on player $i$'s action.

For fixed ${a}^{-i,N}\in\mathcal{A}^{N-1}$, the utility functional of player $i$'s best response to all others' actions ${a}^{-i,N}$ in \eqref{eq:Ui-LQ} is strongly concave with constant $\alpha_U=1$. Hence, taking the G\^ateaux derivative shows that the first-order condition for player $i$'s action to be a best response is
\be\label{eq-LQ:foc}
a^{i,N}=\theta^{i,N}+\frac{\beta}{N} \sum_{j=1} G^N_{ij}a^{j,N},\quad \P \otimes dt \textrm{-a.e.~on } \Omega \times [0,T].
\ee
Therefore, any Nash equilibrium $\bar{a}\in\mathcal{A}^N$ as in Definition~\ref{def:Nash-finite-player} must satisfy
\be\label{eq:Nash-condition-LQ}
(I^N-\frac{\beta}{N}{G}^N)\bar a^N=\theta^N,\quad \P \otimes dt \textrm{-a.e.~on } \Omega \times [0,T],
\ee
where $I^N$ denotes the $N$-dimensional identity matrix. Given that the local aggregate in \cite{galeotti2020targeting} is defined without normalization, the following assumption coincides with Assumption~2 therein.
Recall that $\lambda_{1}(G^N)$ denotes the largest eigenvalue of $G^N$. 
\begin{assumption}\label{assum:spectral-radius}
It holds that $\beta \lambda_{1}(G^N)<N$.
\end{assumption}
Note that $U$ in \eqref{eq:U-LQ} satisfies Assumption~\ref{assum:U} with Lipschitz constants $\ell_U=\beta$ and $\ell_\theta=1$. Thus, due to Assumption~\ref{assum:spectral-radius}, there exists a unique Nash equilibrium of the network game by Corollary~\ref{cor:NE-network}. It can be derived explicitly from \eqref{eq:Nash-condition-LQ} and is given by 
\be\label{eq:Nash-equilibrium-LQ}
\bar{a}^N=(I^N-\frac{\beta}{N}G^N)^{-1}\theta^N,\quad \P \otimes dt \textrm{-a.e.~on } \Omega \times [0,T],
\ee
where $(I^N-\frac{\beta}{N}G^N)$  is invertible since its eigenvalues are positive by Assumption~\ref{assum:spectral-radius}. 
Before the players take action, a central planner maximizes the average utility at equilibrium, by changing a given vector of status quo standalone returns $\theta\in\mathcal{A}^N$ to a vector $\theta^N+\hat\theta^N\in\mathcal{A}^N$, subject to a budget constraint. Namely, given a budget $C_B>0$, the planner's optimization problem from \eqref{eq:network-int-problem} now takes the form
\be\begin{aligned}\label{eq:network-int-problem-LQ}
    \bar\theta^N\in&\argmax_{\hat{\theta}^N\in\mathcal{A}^N}\ T^N(\hat{\theta}^N)= \argmax_{\hat{\theta}^N\in\mathcal{A}^N}\frac{1}{N}\sum_{i=1}^N U\big(\bar a^{i,N},\bar z^{i,N},\theta^{i,N}+\hat{\theta}^{i,N}\big),\\
    \quad \textrm{s.t.}\quad & \bar{a}^N=(I^N-\frac{\beta}{N}G^N)^{-1}(\theta^N+\hat\theta^N),\quad \bar z^{N}=\frac{1}{N} G^N\bar a^N,\quad \P \otimes dt \textrm{-a.e.},\\
    & \frac{1}{N}\|\hat{\theta}^N\|^2_{\mathcal{A}^N}\leq C_B.
\end{aligned}\ee
The analysis of the optimal intervention $\bar\theta^N$ uses the following spectral properties of the symmetric adjacency matrix $G^N$.  
 
\paragraph{Spectral properties of $G^N$.} Since $G^N$ is symmetric, it admits a spectral decomposition $G^N=U^N\Lambda^N (U^N)^\top$, where:
\begin{itemize}
    \item[(i)] $\Lambda^N\in\R^{N\times N}$ is a diagonal matrix whose diagonal entries $\Lambda^N_{ k  k}=\lam_k(G^N)=:\lambda^N_{ k}$ are the eigenvalues of $G^N$, arranged in descending order: $\lambda^N_1\geq\lambda^N_2\geq\ldots\geq\lambda^N_N$.
    \item[(ii)] $U^N$ is an orthogonal matrix, whose $ k$-th column $U^N_{\bullet k}$ is a real normalized eigenvector of $G^N$ corresponding to the eigenvalue $\lambda^N_ k$.
\end{itemize}
 
For a vector $d \in \R^N$, define $\underline{d} = (U^N)^\top d$. We call the $k$-th component of $\underline d$, that is, $\underline{d}_k$, the projection of $d$ onto the $k$-th principal component. Substituting the expression $G^N=U^N\Lambda^N (U^N)^\top$ into equation \eqref{eq:Nash-condition-LQ}, which characterizes the Nash equilibrium, we obtain
\be \label{eq:theta-LQ} 
\big(I^N - \frac{\beta}{N} U^N\Lambda^N (U^N)^\top\big) \bar{a}^N = \theta^N,\quad \P \otimes dt \textrm{-a.e.~on } \Omega \times [0,T].
\ee 
Multiplying both sides of \eqref{eq:theta-LQ} by $(U^N)^\top$ gives us an equivalent representation to \eqref{eq:Nash-equilibrium-LQ}, 
$$
(I^N - \frac{\beta}{N} \Lambda^N) \bar{\underline{a}}^N = \underline{{\theta}}^N\ \Leftrightarrow\ \bar{\underline{a}}^N = (I^N - \frac{\beta}{N} \Lambda^N)^{-1} \underline{\theta}^N.
$$
Note that the $k$-th diagonal entry of the diagonal matrix $(I^N - \frac{\beta}{N} \Lambda^N)^{-1}$ is $\smash{\frac{1}{1 - \beta \lambda^N_k/N}}$, therefore is follows that, 
\be\label{eq:Nash-in-PC-LQ}
\bar{\underline{a}}^N_{k} = \frac{1}{1 - \beta \lambda^N_k/N} \underline{\theta}^N_{k}, \quad  \textrm{for all } k=1,...,N. 
\ee
This shows that the equilibrium action $\underline{\bar{a}}^N_{k}$ in the $k$-th principal component of $G^N$ is the product of an amplification factor (determined by the strategic parameter $\beta$ and the eigenvalue $\lambda^N_k$) and $\underline{\theta}^N_{k}$, the projection of $\theta^N$ onto that principal component. Under Assumption~\ref{assum:spectral-radius}, we have $1 - \beta \lambda^N_k/N > 0$ for all $k$. When $\beta > 0$ ($\beta < 0$), the amplification factor is decreasing (increasing) in $k$. Finally, define 
\be\label{eq:alpha-LQ}
\alpha^N_k:=\frac{1}{(1-\beta\lambda_k^N/N)^2}>0,\quad k=1,\ldots, N.
\ee


\paragraph{Optimal interventions.}  Now we turn to the analysis of the optimal intervention. The following assumption extends Property A and Assumption~3 from \cite{galeotti2020targeting} to the dynamic setting, and holds for various network games of interest. 

\begin{assumption}\label{assum:property}
Assume there exists $\tilde{w}\in\R$ such that the pure externalities satisfy 
$$
\frac{1}{N}\sum_{i=1}^N P\left(z^{i,N}(\bar{a}^N),\theta^{i,N}\right)=\tilde{w} \|\bar a^N\|^2_{\mathcal{A}^N},\quad\textrm{for all }\,\theta^N\in\mathcal{A}^N, 
$$
where either $\tilde{w}<-\frac{1}{2N}$ and $\frac{1}{N}\|\theta^N\|_{\mathcal{A}^N}^2> C_B$, or $\tilde{w}>-\frac{1}{2N}$. Moreover, assume that $\underline{\theta}^N_k\neq 0$, $\P \otimes dt$-a.e., for each $k$.
\end{assumption}

\begin{remark}\label{rem:property}
Plugging \eqref{eq:Nash-condition-LQ} into \eqref{eq:Ui-LQ} shows that under the first part of Assumption~\ref{assum:property} it holds for  $w :=\tilde{w}+\frac{1}{2N}$ that
$$
T^N(\hat\theta^N)=w\|\bar{a}^N\|_{\mathcal{A}^N}^2,\quad\textrm{for all }\,\hat{\theta}^N\in\mathcal{A}^N.
$$
That is, the average utility at equilibrium is proportional to the squared norm of the players' actions, ensuring the tractability of the network intervention problem \eqref{eq:network-int-problem-LQ}. The second part of Assumption~\ref{assum:property} translates into the assumption that either $w<0$ and $\frac{1}{N}\|\theta^N\|_{\mathcal{A}^N}^2> C_B$, or $w>0$. This excludes the trivial case of \eqref{eq:network-int-problem-LQ} where $w<0$ and $\frac{1}{N}\|\theta^N\|_{\mathcal{A}^N}^2\leq C_B$, in which the planner will always choose the optimal intervention $\bar\theta^N=-\theta^N$.  The last part of Assumption~\ref{assum:property} is of technical nature, and will be needed for the proof of Theorem~\ref{thm:LQ}, where we consider interventions relative to status quo standalone returns.
\end{remark}
The following definition of cosine similarity allows us to describe optimal interventions in terms of the similarity to the principal components of $G^N$.
\begin{definition}
The cosine similarity of two nonzero vectors $c,d\in\R^N$ is given by, 
$$
\rho(c, d) = \frac{\langle c , d\rangle_{\R^N}}{\|c\|_{\R^N} \|{d}\|_{\R^N}}.
$$
\end{definition}
The following theorem solves the network intervention problem \eqref{eq:network-int-problem-LQ}. 

\begin{theorem}\label{thm:LQ}
    Under Assumptions~\ref{assum:spectral-radius} and \ref{assum:property}, the cosine similarity between the optimal intervention ${\bar\theta^N}$ and the $k$-th principal component of $G^N$ is given by
    \be\label{eq-LQ:main-cosine-similarity}
    \rho\big({\bar\theta^N},U^N_{\bullet k}\big)
    =\frac{\|\theta^N\|_{\R^N}}{\|{\bar\theta^N}\|_{\R^N}}\rho\big(\theta^N,U^N_{\bullet k}\big)\frac{w\al^N_ k}{\mu-w\al^N_ k},\quad \P \otimes dt \textrm{-a.e.},\quad  k=1,\ldots, N,
    \ee
    where $\mu$ is uniquely determined as the solution to
    $$C_B=\frac{1}{N}\sum_{ k=1}^N\big(\frac{w\al^N_ k}{\mu-w\al^N_ k}\big)^2\|\underline{\theta}^N_ k\|_{\mathcal{A}}^2,$$
    satisfying $\mu>w\al^N_ k$ for all $ k$. In particular, the optimal intervention in the $ k$-th principal component of $G^N$ is explicitly given by 
    \be\label{eq:optimal-zeta-LQ}
    \underline{\bar\theta}^N_{ k}=\frac{w\al^N_ k}{\mu-w\al^N_ k}\underline{\theta}^N_{ k},\quad \P \otimes dt \textrm{-a.e.},\quad  k=1,\ldots, N.
    \ee
\end{theorem}
The proof of Theorem~\ref{thm:LQ} is given in Section~\ref{sec:LQ-proofs}.  
\begin{remark}
Theorem~\ref{thm:LQ} extends Theorem~1 from \citet{galeotti2020targeting} for static  games to the dynamic setting. One of the interesting insights is that the projected optimal intervention $\smash{\underline{\bar{\theta}}^N_ k}$ is just a scalar factor of the projected status quo standalone returns $\underline{\theta}^N_ k$, without further dependence on $t\in [0,T]$ or $\omega\in\Omega$. This deterministic  factor $\al^N_ k/(\mu-w\al^N_ k)$ is given explicitly in terms of the eigenvalues of the network $G^N$, the parameter $\beta$ characterizing the strategic spillovers, the constant $w$ from Remark \ref{rem:property}, and the Lagrange multiplier $\mu$.
\end{remark}

\subsection{The Infinite-Player Game}

As noted in Section~\ref{subsec:int-convergence}, the network intervention problem \eqref{eq:network-int-problem-LQ} scales with the population size $N$, and the graphon intervention problem can be solved more efficiently in many cases. Motivated by our results on the approximation of the optimal network intervention by the optimal graphon intervention (see Theorems~\ref{thm:convergence} and \ref{thm:convergence-sampled}), we now extend our analysis to the infinite-player setting. One of the main differences from the finite-player case is that, in the infinite-player setting, we must work with a spectral decomposition of the graphon operator $\boldsymbol{W}$ on $L^2([0,1], \mathbb{R})$ rather than with one of the adjacency matrix $G^N$ of the finite network. Unlike in the finite-dimensional case, the spectrum of $\boldsymbol{W}$ may contain either a finite or an infinite number of distinct eigenvalues, each with its own multiplicity. By using the spectral properties of $\boldsymbol{W}$, we address these complexities and establish the optimal intervention in semi-explicit form in Theorem~\ref{thm:LQ-graphon}.

\paragraph{Model setup.} 
Consider the setup from Section~\ref{subsec:definition-infinite-player-game}, and assume that the utility functional $U$ in \eqref{eq:U} is given by \eqref{eq:U-LQ},
so that each player $x\in[0,1]$ seeks to maximize their individual utility functional on $\mathcal{A}^x:=\mathcal{A}$ given by
\be \label{eq:Ux-LQ}
a^x\mapsto \left\langle a^x,\theta^x+\beta z^x(a)\right\rangle_{\mathcal{A}}-\frac{1}{2}\| a^x\|^2_{\mathcal{A}}+P\left(z^x(a),\theta^x\right),
\ee 
where the local aggregate $z^x(a)$ is defined in \eqref{eq:graphon-operator}, $W\in\mathcal{W}_0$ is a graphon, $\theta\in\mathcal{A}^\infty$ is a heterogeneity profile, and $\beta$ and $P$ are as before.

For $a\in\mathcal{A}^\infty$ and $x\in[0,1]$, we write $a^{-x}:=(\al^{y})_{y\neq x}\in\mathcal{A}^\infty$. Now for fixed $a^{-x}\in\mathcal{A}^{\infty}$, the utility functional in \eqref{eq:Ux-LQ} of player $x$'s best response to all others' actions $a^{-x}$ is strictly concave. Hence, taking the G\^ateaux derivative shows that the first-order condition for player $x$'s action to be a best response is
\be
a^x=\theta^x+\beta\int_0^1 W(x,y)a^ydy,\quad \P \otimes dt \textrm{-a.e.~on } \Omega \times [0,T].
\ee
Therefore, any Nash equilibrium $\bar{a}\in\mathcal{A}^\infty$ as in Definition~\ref{def:Nash-infinite-player} must satisfy
\be\label{eq-LQ:Nash-condition-graphon}
(\boldsymbol{I-\beta\boldsymbol{W}})\bar{a}=\theta,\quad \P \otimes dt \otimes d\nu \textrm{-a.e.~on } \Omega \times [0,T]\times [0,1],
\ee
where $\boldsymbol{I}$ denotes the identity operator on $L^2([0,1],\R)$. The following assumption is the continuum analogue of Assumption~\ref{assum:spectral-radius}.
\begin{assumption}\label{assum:spectral-radius-graphon}
It holds that $\beta\lambda_1(\boldsymbol{W}) <1$.
\end{assumption}
Recalling that $U$ in \eqref{eq:U-LQ} satisfies Assumption~\ref{assum:U} with Lipschitz constants $\ell_U=\beta$ and $\ell_\theta=1$, there exists a unique Nash equilibrium of the graphon game by Theorem~\ref{thm:NE} and Assumption~\ref{assum:spectral-radius-graphon}. Moreover, the operator $\boldsymbol{I-\beta\boldsymbol{W}}$ is invertible (see \cite{hall2013quantum}, Chapter~7.1, Lemma~7.6), and the Nash equilibrium is explicitly given by
\be\label{eq-LQ:graphon-Nash-equilibrium}
\bar{a}=(\boldsymbol{I-\beta\boldsymbol{W}})^{-1}\theta,\quad \P \otimes dt \otimes d\nu \textrm{-a.e.~on } \Omega \times [0,T]\times [0,1].
\ee
A central planner seeks to maximize the average utility at equilibrium, by changing some given status quo standalone returns $\theta\in\mathcal{A}^\infty$ to $\theta+\hat\theta\in\mathcal{A}^\infty$, subject to a budget constraint. Namely, given a budget $C_B>0$, the graphon intervention problem in \eqref{eq:graphon-int-problem} now becomes
\be\begin{aligned}\label{eq:graphon-int-problem-LQ}
    \bar\theta\in&\argmax_{\hat{\theta}\in\mathcal{A}^\infty}\ T(\hat{\theta})= \argmax_{\hat{\theta}\in\mathcal{A}^\infty}\int_0^1 U\big(\bar a^x,\bar z^x,\theta^x+\hat{\theta}^x\big)dx,\\
    \quad \textrm{s.t.}\quad & \bar{a}=(\boldsymbol{I-\beta\boldsymbol{W}})^{-1}(\theta+\hat\theta),\quad \bar z=\boldsymbol{W} \bar a,\quad \P \otimes dt \otimes d\nu \textrm{-a.e.}, \\
    & \|\hat{\theta}\|^2_{\mathcal{A}^\infty}\leq C_B.
\end{aligned}\ee

Denote by $\langle\cdot,\cdot\rangle_{L^2}$ and $\|\cdot\|_{L^2}$ the inner product and norm on $L^2([0,1],\R)$. Notice that the self-adjoint operator $\boldsymbol{W}$ on $L^2([0,1],\R)$ is a  Hilbert-Schmidt operator and thus compact. The following spectral decomposition is the infinite-dimensional analogue to the decomposition of the matrix $G^N$ presented before. 

\paragraph{Spectral properties of $\boldsymbol{W}$.}
     \begin{itemize}
    \item[(i)] The spectrum $\sigma(\boldsymbol{W})$ of $\boldsymbol{W}$ is given by $\sigma(\boldsymbol{W})=\{0\}\cup\{\lambda_i\}_{i\in I}\subset\R$, where $I$ is a countable (possibly finite) index set and $\lambda_i\neq 0$ for all $i\in I$.  Every nonzero $\lambda\in\sigma(\boldsymbol{W})$ is an isolated eigenvalue with finite multiplicity denoted by $m(\lambda)\in\N$. It is possible that $m(0)=\infty$. (See \cite{conway2019course}, Chapter~7.7, Theorem~7.1.)
    \item[(ii)] There exists a countable orthonormal basis of eigenfunctions $\smash{\{\{ e_{\lambda,j}\}_{j=1}^{m(\lambda)}\}_{\lambda\in \sigma(\boldsymbol{W})}}$ of $\boldsymbol{W}$  corresponding to the eigenvalues.
    In particular, any function $f\in L^2([0,1],\R)$ can be decomposed as 
    $$
    f(x)=\sum_{\lambda\in\sigma(\boldsymbol{W})}\sum_{j=1}^{m(\lambda)} \langle f,e_{\lambda,j}\rangle_{L^2} e_{\lambda,j}(x),\quad x\in [0,1],
    $$
    so it follows that
    \be\label{eq-LQ:representation-Wf}
    (\boldsymbol{W}f)(x)=\sum_{\lambda\in\sigma(\boldsymbol{W})}\sum_{j=1}^{m(\lambda)} \lambda\langle f, e_{\lambda,j}\rangle_{L^2} e_{\lambda,j}(x),\quad x\in [0,1].
    \ee
    (See \cite{lovasz2012large}, Chapter~7.5.)
    \item[(iii)] For $j\in\N$, denote by $\|\cdot\|_{\R^j}$ the Euclidean norm on $\R^j$. Define the Hilbert space direct sum 
    $$
    X:=\bigoplus_{\lambda\in\sigma(\boldsymbol{W})}\R^{m(\lambda)}
    $$
    containing all tuples $g=(g(\lambda))_{\lambda\in\sigma(\boldsymbol{W})}$ such that $g(\lambda)\in\R^{m(\lambda)}$ and 
    $$    \| g\|_X^2:=\sum_{\lambda\in\sigma(\boldsymbol{W})}\|g(\lambda)\|_{\R^{m(\lambda)}}^2<\infty.
    $$
    In the case that $m(0)=\infty$, the space $\R^\infty$ denotes the square-summable sequences. Then $\boldsymbol{W}$ admits a decomposition $\boldsymbol{W}=\boldsymbol{U^*\boldsymbol{M}\boldsymbol{U}}$ where $\boldsymbol{U}:L^2([0,1],\R)\to X$ is the unitary operator given by
    $$
    (\boldsymbol{U} f)(\lambda)=(\langle f, e_{\lambda,1}\rangle_{L^2},\langle f, e_{\lambda,2}\rangle_{L^2},\ldots,\langle f, e_{\lambda,m(\lambda)}\rangle_{L^2})\in\R^{m(\lambda)},\quad \lambda\in\sigma(\boldsymbol{W}),
    $$
    and $\boldsymbol{M}:X\to X$ is the multiplication operator given by $(\boldsymbol{M} g)(\lambda)=\lambda  g(\lambda)$, $ \lambda\in\sigma(\boldsymbol{W})$. (See \cite{hall2013quantum}, Chapter~7.3.) 
\end{itemize}

For a function $ f \in L^2([0,1],\R)$, define $\smash{\underline{ f}} = \boldsymbol{U} f\in X$, which consists of the projections of $f$ onto the eigenspaces corresponding to the  eigenvalues in $\sigma(\boldsymbol{W})$. In particular, we call $\underline{f}(\lambda)=(\boldsymbol{U} f)(\lambda)\in\R^{m(\lambda)}$ the projection onto the principal component corresponding to $\lambda$. Substituting the expression $\boldsymbol{W} = \boldsymbol{U}^* \boldsymbol{M} \boldsymbol{U}$ into equation \eqref{eq-LQ:Nash-condition-graphon}, which characterizes the graphon Nash equilibrium, we obtain
$$
(\boldsymbol{I} - \beta \boldsymbol{U}^* \boldsymbol{M} \boldsymbol{U}) \bar{a} = \theta,\quad \P \otimes dt \otimes d\nu \textrm{-a.e.~on } \Omega \times [0,T]\times [0,1].
$$
Applying $\boldsymbol{U}$ to both sides of this equation  gives us an analogue of \eqref{eq-LQ:graphon-Nash-equilibrium} characterizing the solution of the game,
$$
(\boldsymbol{J} - \beta \boldsymbol{M}) \underline{\bar{a}} = \underline{\theta}\ \Leftrightarrow\ \underline{\bar{a}} = (\boldsymbol{J} - \beta \boldsymbol{M})^{-1} \underline{\theta},
$$
where $\boldsymbol{J}$ denotes the identity operator on $X$. Note that the operator $(\boldsymbol{J} - \beta \boldsymbol{M})^{-1}$ is a multiplication operator, given by
$$
((\boldsymbol{J}-\beta\boldsymbol{M})^{-1} g)(\lambda)=(1-\beta \lambda)^{-1} g(\lambda),\quad\textrm{for } g\in X \textrm{ and } \lambda\in\sigma(\boldsymbol{W}).
$$
Due to Assumption~\ref{assum:spectral-radius-graphon}, we have $1 - \beta \lambda > 0$ for all $\lambda\in \sigma( \boldsymbol{W})$. Define the scalars $$\al_{\lambda}: =(1 - \beta \lambda)^{-2}.$$ Then, for every $\lambda\in \sigma(\boldsymbol{W})$,
\be\label{eq-LQ:Nash-in-PC-graphon}
\bar{\underline{a}}(\lambda) = \sqrt{\al_{\lambda}} \underline{\theta}(\lambda).
\ee
As in the finite-player case, this shows that the equilibrium action $\underline{\bar{a}}(\lambda)$ in the principal component corresponding to $\lambda$ is the product of an amplification factor (determined by the strategic parameter $\beta$ and $\lambda$) and $\underline{\theta}(\lambda)$, the projection of $\theta$ onto that principal component. When $\beta > 0$ ($\beta < 0$), the amplification factor is increasing (decreasing) in $\lambda$. Equation \eqref{eq-LQ:Nash-in-PC-graphon} also yields a reconstruction formula for equilibrium actions in the original coordinates,
$$
\bar{a}^x = \sum_{\lambda\in\sigma(\boldsymbol{W})}\sqrt{\al_{\lambda}}\sum_{j=1}^{m(\lambda)}  \underline{\theta}(\lambda)_j  e_{\lambda,j}(x),\quad x\in [0,1].
$$

\paragraph{Optimal interventions.} We now study optimal interventions in the graphon game. The following assumption is the infinite-dimensional analogue of Assumption~\ref{assum:property}.
\begin{assumption}\label{assum:property-graphon}
Assume there exists $\tilde{w}\in\R$ such that the pure externalities satisfy 
$$
\int_0^1 P\left(z^x(\bar a),\theta^x\right)dx=\tilde{w} \|\bar a\|^2_{\mathcal{A}^\infty},\quad\textrm{for all }\,\theta\in\mathcal{A}^\infty.
$$
Suppose that either $\tilde{w}<-\frac{1}{2}$ and $\|\theta\|_{\mathcal{A}^\infty}^2> C_B$, or $\tilde{w}>-\frac{1}{2}$. Furthermore, assume that $\underline{\theta}(\lambda)_j\neq 0$, $\P \otimes dt$-a.e., for all $j=1,\ldots,m(\lambda)$ and all $\lambda\in\sigma(\boldsymbol{W})$.
\end{assumption}

\begin{remark}\label{rem:property-graphon}
\eqref{eq-LQ:Nash-condition-graphon} into \eqref{eq:Ux-LQ} shows that under the first part of Assumption~\ref{assum:property-graphon} it holds for  $w :=\tilde{w}+\frac{1}{2}$ that
$$
T(\hat\theta)=w\left\|\bar{a}\right\|_{\mathcal{A}^\infty}^2,\quad\textrm{for all }\,\hat{\theta}\in\mathcal{A}^\infty.
$$
That is, the average utility at equilibrium is proportional to the squared norm of the players' actions, ensuring the tractability of the graphon intervention problem in \eqref{eq:graphon-int-problem-LQ}. The second part of Assumption~\ref{assum:property-graphon} translates into the assumption that either $w<0$ and $\|\theta\|_{\mathcal{A}^\infty}^2> C_B$, or $w>0$. As in the finite-player setting, this excludes the trivial case of \eqref{eq:graphon-int-problem-LQ} where $w<0$ and $\|\theta\|_{\mathcal{A}^\infty}^2\leq C_B$, in which the planner will always choose the optimal intervention $\bar\theta=-\theta$.  The last part of Assumption~\ref{assum:property-graphon} is of technical nature, and will be needed for the proof of Theorem~\ref{thm:LQ-graphon}.
\end{remark}
\begin{definition}
The cosine similarity of two nonzero functions $f,g\in L^2([0,1],\R) $ is
$$
\rho(f,g) = \frac{\langle f,g\rangle_{L^2}}{\|f\|_{L^2} \|g\|_{L^2}}.
$$
\end{definition}
The following theorem solves the graphon intervention problem \eqref{eq:graphon-int-problem-LQ}.
\begin{theorem}\label{thm:LQ-graphon}
    Suppose Assumptions~ \ref{assum:spectral-radius-graphon} and \ref{assum:property-graphon} hold.
    At the optimal intervention, the cosine similarity between $\bar\theta$ and the eigenfunction $e_{\lambda,j}$ is given by
    \be\label{eq:main-graphon-cosine-similarity-LQ}
    \rho\big(\bar\theta,e_{\lambda,j}\big)
    =\frac{\|\theta\|_{L^2}}{\|\bar\theta\|_{L^2}}\rho\big(\theta,e_{\lambda,j}\big)\frac{w\al_\lambda}{\mu-w\al_\lambda},\quad \P \otimes dt \textrm{-a.e.},\quad \lambda\in\sigma(\boldsymbol{W}),\ j=1,\ldots, m(\lambda),
    \ee
    where $\mu$ is uniquely determined as the solution to
    \be\label{eqshadow-price-graphon-LQ}
    C_B=\sum_{\lambda\in\sigma(\boldsymbol{W})}\big(\frac{w\al_\lambda}{\mu-w\al_\lambda}\big)^2\big\|\|{\underline{\theta}}(\lambda)\|_{\R^{m(\lambda)}}\big\|_{\mathcal{A}}^2,
    \ee
    satisfying $\mu>w\al_\lambda$ for all $\lambda$. In particular, the optimal intervention in the principal component of $\boldsymbol{W}$ corresponding to $\lambda\in\sigma(\boldsymbol{W})$ is explicitly given by 
    \be\label{eq:optimal-zeta-graphon-LQ}
    \bar{\underline{\theta}}(\lambda)=\frac{w\al_\lambda}{\mu-w\al_\lambda}{\underline{\theta}}(\lambda),\quad \P \otimes dt \textrm{-a.e.},\quad \lambda\in\sigma(\boldsymbol{W}).
    \ee
\end{theorem}
The proof of Theorem~\ref{thm:LQ-graphon} is given in Section~\ref{sec:LQ-proofs}.
\begin{remark}
Theorem~\ref{thm:LQ-graphon} extends Theorem~1 from \citet{galeotti2020targeting} for static finite-player games to the dynamic infinite-player setting. As in the finite-player case (see Theorem~\ref{thm:LQ}), the projected optimal intervention $\bar{\underline{\theta}}(\lambda)$ is simply a scalar factor of the projected status quo standalone returns ${\underline{\theta}}(\lambda)$, without further dependence on $t\in [0,T]$ or $\omega\in\Omega$. The deterministic  factor $\frac{w\al_\lambda}{\mu-w\al_\lambda}$ is given explicitly in terms of the eigenvalues of the graphon operator $\boldsymbol{W}$, the parameter $\beta$ characterizing the strategic spillovers, the constant $w$ from Remark \ref{rem:property-graphon}, and the Lagrange multiplier $\mu$.
\end{remark}
Theorem~\ref{thm:LQ-graphon} allows for a detailed description of the optimal intervention in terms of the principal components of $\boldsymbol{W}$. Namely, as we vary $\lambda$, equation \eqref{eq:main-graphon-cosine-similarity-LQ} shows that $\rho(\bar\theta,e_{\lambda,j})$ is proportional to the cosine similarity to the status quo $\rho(\theta,e_{\lambda,j})$ and the factor $\frac{w\al_\lambda}{\mu-w\al_\lambda}$. As we vary $\lambda$, the similarity ratio
$$
\bar{r}_\lambda:=\frac{\rho(\bar\theta,e_{\lambda,j})}{\rho(\theta,e_{\lambda,j})}
$$
is proportional to $\frac{w\al_\lambda}{\mu-w\al_\lambda}$, and therefore larger for the principal components in which the optimal intervention makes the largest change relative to the status quo standalone returns. We obtain the following corollary, which extends Corollary~1 from \cite{galeotti2020targeting} to the dynamic infinite-player setting.
\begin{corollary}\label{corollary:similarity-ratio-graphon}
    Suppose Assumptions~ \ref{assum:spectral-radius-graphon} and \ref{assum:property-graphon} hold. If the game is one of strategic complements ($\beta>0$), then $|\bar{r}_\lambda|$ is increasing in $\lambda$, $\P\otimes dt$-a.e.; if the game is one of strategic substitutes ($\beta<0$), then $|\bar{r}_\lambda|$ is decreasing in $\lambda$, $\P\otimes dt$-a.e.
\end{corollary}
\paragraph{Small and large budgets.}
The conclusions we can draw from Theorem~\ref{thm:LQ-graphon} become especially salient, when we consider very small or very large budgets $C_B>0$. Namely, since by \eqref{eqshadow-price-graphon-LQ} $\mu$ is decreasing in $C_B$, if $w>0$ ($w<0$), $\frac{w\al_\lambda}{\mu-w\al_\lambda}$ is increasing (decreasing) in $C_B$. Furthermore, if $w>0$ ($w<0$), for all $\lambda,\lambda'$ with $\al_\lambda>\al_{\lambda'}$, it holds that $\bar{r}_\lambda/\bar{r}_{\lambda'}$ is increasing in $C_B$. The following proposition extends Proposition~1 from \cite{galeotti2020targeting} to the dynamic infinite-player setting.
\begin{proposition}\label{prop:convergence-graphon}
    Suppose that Assumptions~\ref{assum:spectral-radius-graphon} and \ref{assum:property-graphon} hold. Then:
    \begin{itemize}
        \item[(i)] As $C_B\to 0$, at the optimal intervention, $\frac{\bar{r}_\lambda}{\bar{r}_{\lambda'}}\to\frac{\al_\lambda}{\al_{\lambda'}}$.
        \item[(ii)] Let $\lambda_1$ ($\lambda_{s}$) denote the largest (smallest) eigenvalue of $\boldsymbol{W}$, and assume that it has multiplicity 1.  Then, if $\beta>0$ ($\beta<0$), as $C_B\to\infty$, it holds that $|\rho(\bar\theta,e_{\lambda_1,1})|\to 1$ ($|\rho(\bar\theta,e_{\lambda_{s},1})|\to 1$), $\P \otimes dt$-a.e.
    \end{itemize}
\end{proposition}
Proposition~\ref{prop:convergence-graphon} is a consequence of equation \eqref{eq:main-graphon-cosine-similarity-LQ}. As $C_B\to 0$, by equation \eqref{eqshadow-price-graphon-LQ}, $\mu$ goes to $\infty$, and the similarity ratio $\bar{r}_\lambda$ corresponding to each eigenvalue $\lambda$ becomes proportional to the corresponding amplification factor $\al_\lambda$. That means, that the optimal intervention is guided by all principal components of the network and weighted by the corresponding amplification factors. In the case where $C_B\to \infty$, by equation \eqref{eqshadow-price-graphon-LQ}, $\mu$ goes to $w\al_{\lambda_1}$ if $\beta>0$, and to $w\al_{\lambda_{s}}$ if $\beta<0$. Asymptotically, the optimal intervention is proportional to the one-dimensional principal component of $\boldsymbol{W}$ corresponding to $\lambda_1$ if $\beta>0$ and to $\lambda_{s}$ if $\beta<0$. This means that the optimal intervention is determined by only one of the infinitely many eigenfunctions of $\boldsymbol{W}$.

\paragraph{Simple interventions.}
While Proposition~\ref{prop:convergence-graphon} describes the convergence of the cosine similarities for asymptotic budgets $C_B$, it doesn't bound the corresponding convergence rate. In order to understand better how well the asymptotic case $C_B\to\infty$  approximates the case of large $C_B<\infty$, we next generalize Proposition~2 from \cite{galeotti2020targeting}, which, depending on the budget $C_B$, gives a bound on how close the aggregate utility and cosine similarity are to the asymptotic case. 
Recall that for $\lambda\in\sigma(\boldsymbol{W}),\ j\in\{1,\ldots, m(\lambda)\}$, the notation $\underline{\theta}(\lambda)_j$ denotes the $j$-th coordinate of the projection of $\theta$ onto the eigenspace corresponding to $\lambda$.

\begin{definition}
    Let $\lambda_1$ ($\lambda_{s}$) denote the largest (smallest) eigenvalue of $\boldsymbol{W}$, and assume that it has multiplicity 1. Define the $\F$-progressively measurable processes $\Omega\times[0,T]\to\R$ given by 
    $$
    c_t^+:=\sqrt{C_B}\frac{\underline{\theta}_t(\lambda_1)_1}{\|\underline{\theta}(\lambda_1)_1\|_{\mathcal{A}}},\quad 
    c_t^-:=\sqrt{C_B}\frac{\underline{\theta}_t(\lambda_{s})_1}{\|\underline{\theta}(\lambda_{s})_1\|_{\mathcal{A}}},\quad 0\leq t\leq T. 
    $$
    An intervention $\hat\theta\in\mathcal{A}^\infty$ is called simple if for $\nu$-a.e.~$x\in [0,1],$
    \begin{itemize}
    \item $\hat\theta^x_t=c_t^+e_{\lambda_1,1}(x)$, when the game has strategic complements ($\beta>0$),
    \item $\hat\theta^x_t=c^-_te_{\lambda_{s},1}(x)$, when the game has strategic substitutes ($\beta<0$),
    \end{itemize}
    where $e_{\lambda_{1},1}$ and $e_{\lambda_{s},1}$ are uniquely determined up to multiplication by $-1$.
\end{definition}
Consistent with \cite{galeotti2020targeting}, these interventions are called simple, since the intervention at each $x$, up to the scaling by the process $c^+$ ($c^-$), is only determined by the value of the eigenfunction $e_{\lambda_1,1}$ ($e_{\lambda_{s},1}$) depending on the underlying graphon. However, in contrast to the static case, where simple interventions factorize into an eigenvector of the underlying graph's adjacency matrix and a scalar, we obtain a factorization into an eigenfunction of the graphon operator and a process in $\mathcal{A}$ to account for the dynamicity.
Here the processes $c^+$ and $c^-$ are determined by equation \eqref{eq:optimal-zeta-graphon-LQ} and the fact that the constraint in \eqref{eq:graphon-int-problem-LQ} is binding at the optimal intervention. In particular, $\|c^+\|_{\mathcal{A}}=\|c^-\|_{\mathcal{A}}=\sqrt{C_B}$.

Let $T_{\operatorname{opt}}$ and $T_{\operatorname{sim}}$ denote the average utility under the optimal and simple interventions, respectively. Denote by $\lambda_2$ and $\lambda_{s-1}$ the second largest and second smallest eigenvalue of $\boldsymbol{W}$, respectively. 

\begin{proposition}\label{prop:convergence-rate-graphon}
    Suppose Assumptions~ \ref{assum:spectral-radius-graphon} and \ref{assum:property-graphon} hold with $\tilde w>-\frac{1}{2}$.
    \begin{itemize}
        \item If $\beta>0$ and $\lambda_1$ has multiplicity $1$, then for any $\delta>0$, if $\smash{C_B>\frac{2}{\delta}\|\theta\|_{\mathcal{A}^\infty}^2(\frac{\al_{\lambda_2}}{\al_{\lambda_1}-\al_{\lambda_2}})^2}$, then $T_{\operatorname{opt}}/T_{\operatorname{sim}}<1+\delta$ and $\rho(\|\bar\theta\|_{\mathcal{A}},c^+e_{\lambda_1,1})>\sqrt{1-\delta}$.
        \item If $\beta<0$ and $\lambda_{s}$ has multiplicity $1$, then for any $\delta>0$, if $\smash{C_B>\frac{2}{\delta}\|\theta\|_{\mathcal{A}^\infty}^2(\frac{\al_{\lambda_{s-1}}}{\al_{\lambda_s}-\al_{\lambda_{s-1}}})^2}$, then $T_{\operatorname{opt}}/T_{\operatorname{sim}}<1+\delta$ and $\rho(\|\bar\theta\|_{\mathcal{A}},c^-e_{\lambda_{s},1})>\sqrt{1-\delta}$.
    \end{itemize}
\end{proposition}
Proposition~\ref{prop:convergence-rate-graphon} characterizes the size of the budget $C_B$ beyond which simple interventions achieve most of the welfare and the optimal intervention resembles the simple intervention. This size depends on the status quo standalone returns and the graphon via the fraction ${\al_{\lambda_2}}/({\al_{\lambda_1}-\al_{\lambda_2}})$ and $\al_{\lambda_{s-1}}/({\al_{\lambda_s}-\al_{\lambda_{s-1}}})$, respectively. In general, the budget needs to be larger for large status quo standalone returns and small spectral gaps $\al_{\lambda_1}-\al_{\lambda_2}$ at the top of the spectrum (and $\al_{\lambda_s}-\al_{\lambda_{s-1}}$ at the bottom of the spectrum, respectively). This result extends the results of Proposition~2 from \cite{galeotti2020targeting} from static finite-player games to the dynamic infinite-player setting.

The proofs of Propositions~\ref{prop:convergence-graphon} and \ref{prop:convergence-rate-graphon} follow by the same generalization of Propositions~1 and 2 in \cite{galeotti2020targeting} that underlies our derivation of Theorem~\ref{thm:LQ-graphon} from their Theorem~1 (see Section~\ref{sec:LQ-proofs}); we therefore omit them for brevity.

\section{Proofs of the Results in Section~\ref{sec:underlying-game}}\label{sec:underlying-game-proofs}

\begin{proof}[Proof of Proposition~\ref{prop:correspondence}] 
Let $\bar a$ be a Nash equilibrium of the game $\mathcal{G}(\mathcal{A}_{ad}^\infty,U,\theta^N_{\operatorname{step}},W_{G^N})$, where $\mathcal{A}^x:=\mathcal{A}^{i,N}$ for all $x\in\mathcal{P}_i^N$. Then, since $W_{G^N}$ is a step function with respect to $\smash{\mathcal{P}^N}$, the local aggregate 
$$ z^x(\bar a)=\int_0^1W_{G^N}(x,y)\bar a^ydy$$
is a step function with respect to $\mathcal{P}^N$ too. Let $\bar z^{i,N}$ be the value of $z(\bar a)$ on $\mathcal{P}_i^N$. Then it follows from Definition~\ref{def:Nash-infinite-player} of a Nash equilibrium in the graphon game that
$$
\bar a^x=\argmax_{\tilde a\in\mathcal{A}^x}U(\tilde a, z^{x}(\bar a),\theta^{x,N}_{\operatorname{step}})=
\argmax_{\tilde a\in\mathcal{A}^{i,N}}U(\tilde a, \bar z^{i,N},\theta^{i,N}),\quad \textrm{ for all }x\in\mathcal{P}_i^N,
$$
which implies that $\bar a$ is a step function action profile with respect to $\mathcal{P}^N$. Let $\bar a^{i,N}$ be the value of $\bar a$ on $\mathcal{P}_i^N$.
Then 
$$\bar z^{i,N}=\int_0^1W_{G^N}(x,y)\bar a^ydy=\frac{1}{N}\sum_{j=1}^N G_{ij}^N\bar a^{j,N},
$$
and $\bar a$ is a graphon Nash equilibrium if and only if
$$
\bar a^{i,N}=\argmax_{\tilde a\in\mathcal{A}^{i,N}}U(\tilde a, \bar z^{i,N},\theta^{i,N}), \quad  \bar z^{i,N}=\frac{1}{N}\sum_{j=1}^N G_{ij}^N\bar a^{j,N},\quad \textrm{ for all } i\in\{1,\ldots,N\}.
$$
Since the latter aligns exactly with Definition~\ref{def:Nash-finite-player} of a Nash equilibrium in the network game $\mathcal{G}(\mathcal A_{ad}^N,U,\theta^N,G^N)$, this completes the proof.
\end{proof}

The following sensitivity result for solutions to variational inequalities on Hilbert spaces is needed for the proof of Theorem~\ref{thm:NE}. It follows from \cite{kinderlehrer2000introduction}, Chapter~2.2, Theorem~2.1.
\begin{lemma}\label{lemma:VI}
Let $H$ be a real Hilbert space with inner product $\langle\cdot,\cdot\rangle_H$ and norm $\|\cdot\|_H$. Let $F:H\times H\to\R$ be a bilinear form on $H$ which is coercive with constant $\alpha>0$, that is, $F(v,v)\geq \alpha\|v\|_H^2$ for all $v\in H$. Let $V\subset H$ be a closed and convex subset and $r_1,r_2\in H$. Then, each of the variational inequalities 
$$
F(u,v-u)\geq \langle r_i,v-u\rangle_H,\quad \textrm{for all }v\in V, \ i=1,2,
$$
admits a unique solution $u_i\in H$, $i=1,2$ (respectively), and it holds that 
$$
\|u_1-u_2\|_H\leq \frac{1}{\alpha}\|r_1-r_2\|_H.
$$
\end{lemma}

In order to prove Theorem~\ref{thm:NE}, we introduce an auxiliary operator. Namely, given a heterogeneity profile $\theta\in\mathcal{A^\infty}$, define the best-response operator $\boldsymbol{B}_\theta$ with domain $\mathcal{A^\infty}$ by
\be \label{def-b}
(\boldsymbol{B}_\theta z)(x):=\argmax_{\tilde{a}\in \mathcal{A}^x} U\big(\tilde{a}, z^x,\theta^x\big), \quad z\in\mathcal{A}^\infty,
\ee
which assigns to any fixed local aggregate $z^x$ (not necessarily of the form $z^x(a)$) the best response of player $x$. The argmax in \eqref{def-b} exists and is unique due to Assumption~\ref{assum:U}. As shown in the following Lemma~\ref{lemma:B_theta}, the image of $\boldsymbol{B}_\theta$ is contained in $\mathcal{A}^\infty$, turning it into a well-defined operator from $\mathcal{A}^\infty$ to $\mathcal{A}^\infty$.
\begin{lemma}\label{lemma:B_theta}
Under Assumptions~\ref{assum:U} and \ref{assum:graphon}, the best-response operator $\boldsymbol{B}_\theta$ satisfies the following:
\begin{itemize}
\item[(i)] $\boldsymbol{B}_\theta$ is jointly Lipschitz continuous, that is,
$$
\left\|\boldsymbol{B}_{\theta_1}z_1-\boldsymbol{B}_{\theta_2}z_2\right\|_{\mathcal{A}^\infty}\leq\frac{1}{\alpha_U}\big(\ell_U\|z_1-z_2\|_{\mathcal{A}^\infty}+\ell_\theta\|\theta_1-\theta_2\|_{\mathcal{A}^\infty}\big),
$$
for all $z_1,z_2\in\mathcal{A}^\infty$ and $\theta_1,\theta_2\in\mathcal{A}^\infty$. 
\item[(ii)]  The image of $\boldsymbol{B}_\theta$ is contained in $\mathcal{A}^\infty$, that is, $\boldsymbol{B}_\theta(\mathcal{A}^\infty)\subset\mathcal{A}^\infty$.  
\item[(iii)]  If additionally Assumption~\ref{assum:Ax} holds, the image of $\boldsymbol{B}_\theta$ is contained in 
$$\mathcal{A}^\infty_M:=\left\{a\in\mathcal{A}^\infty\Big| \|a\|_{\mathcal{A}^\infty}\leq M\right\}.$$
\end{itemize}
\end{lemma}

\begin{proof}[Proof of Lemma~\ref{lemma:B_theta}] Recall the definitions in \eqref{eq:A} and \eqref{eq:Ainfty} of the Hilbert spaces $\mathcal{A}$ and $\mathcal{A}^\infty$ and their norms.  

(i) Let $z_1,z_2, \theta_1,\theta_2\in\mathcal{A}^\infty$. By Assumption~\ref{assum:U}, for any $x\in[0,1]$, $U(\tilde{a},z_1^x,\theta_1^x)$ is $\alpha_U$-strongly concave and G\^ateaux differentiable in $\tilde{a}$. Therefore, the negative G\^ateaux gradient $-\nabla_{\tilde a}U(\cdot,z_1^x,\theta_1^x)$ is $\alpha_U$-strongly monotone, that is, 
\be\label{eq:monotone}
\big \langle \tilde{a}-\tilde{b}, -\nabla_{\tilde a}U(\tilde{a},z_1^x,\theta_1^x)-\big(-\nabla_{\tilde a}U(\tilde{b},z_1^x,\theta_1^x)\big)\big\rangle_\mathcal{A}\geq \alpha_U\|\tilde a-\tilde b \|_{\mathcal{A}}^2,\quad \textrm{for all }\tilde{a},\tilde{b}\in\mathcal{A},
\ee
(see \cite{bauschke2017}, Chapter~17, Exercise 17.5). Consider the corresponding bilinear form $F^x$ on $\mathcal{A}$ given by 
\be\label{eq:Fx}
F^x(\tilde a,\tilde b):= -\langle \nabla_{\tilde a}U(\tilde{a},z_1^x,\theta_1^x), \tilde{b}\rangle_{\mathcal{A}},\quad \tilde{a},\tilde{b}\in\mathcal{A}.
\ee
Setting $\tilde b=0$ in \eqref{eq:monotone} shows that this bilinear form is coercive with constant $\alpha_U$, that is,
\be
F^x(\tilde a,\tilde a)=-\langle \nabla_{\tilde a}U(\tilde{a},z_1^x,\theta_1^x), \tilde{a}\rangle_\mathcal{A}\geq \alpha_U \|\tilde{a}\|^2_\mathcal{A},\quad \textrm{for all }\tilde{a}\in\mathcal{A}.
\ee
Next, note that the unique best-responses $(\boldsymbol{B}_{\theta_1}z_1)(x)$ and $(\boldsymbol{B}_{\theta_2}z_2)(x)$ are given by the unique solutions to the variational inequalities
\be\label{eq:VI1}
\langle \nabla_{\tilde a}U(\tilde{a},z_1^x,\theta_1^x), \tilde{b}- \tilde{a}\rangle_\mathcal{A}\leq 0, \quad  \textrm{for all } \tilde{b}\in\mathcal{A}^x,
\ee
and 
\be\label{eq:VI2}
\langle \nabla_{\tilde a}U(\tilde{a},z_2^x,\theta_2^x), \tilde{b}- \tilde{a}\rangle_\mathcal{A}\leq 0, \quad  \textrm{for all } \tilde{b}\in\mathcal{A}^x,
\ee
respectively. Recalling \eqref{eq:Fx}, inequality \eqref{eq:VI1} can equivalently be written as 
\be\label{eq:VI1-rewritten}
F^x(\tilde a,\tilde b-\tilde a)\geq 0, \quad  \textrm{for all } \tilde{b}\in\mathcal{A}^x,
\ee
and, adding $F^x(\tilde a,\tilde b-\tilde a)$ to both of its sides, inequality \eqref{eq:VI2} can equivalently be written as 

\be
F^x(\tilde a,\tilde b-\tilde a)\geq \langle \nabla_{\tilde a}U(\tilde{a},z_2^x,\theta_2^x), \tilde{b}- \tilde{a}\rangle_\mathcal{A}+F^x(\tilde a,\tilde b-\tilde a),\quad  \textrm{for all } \tilde{b}\in\mathcal{A}^x,
\ee
which by \eqref{eq:Fx} simplifies to 
\be\label{eq:VI2-rewritten}
F^x(\tilde a,\tilde b-\tilde a)\geq \langle \nabla_{\tilde a}U(\tilde{a},z_2^x,\theta_2^x)-\nabla_{\tilde a}U(\tilde{a},z_1^x,\theta_1^x), \tilde{b}- \tilde{a}\rangle_\mathcal{A},\quad  \textrm{for all } \tilde{b}\in\mathcal{A}^x.
\ee
Now an application of Lemma~\ref{lemma:VI} to the variational inequalities \eqref{eq:VI1-rewritten} and \eqref{eq:VI2-rewritten} (which are equivalent to \eqref{eq:VI1} and \eqref{eq:VI2}, respectively) with $r_1=0$, $r_2=\nabla_{\tilde a}U(\tilde{a},z_2^x,\theta_2^x)-\nabla_{\tilde a}U(\tilde{a},z_1^x,\theta_1^x)$ implies for any $x\in [0,1]$, 
\be\begin{aligned}\label{eq:sensitivity}
\|(\boldsymbol{B}_{\theta_1}z_1)(x)-(\boldsymbol{B}_{\theta_2}z_2)(x)\|_\mathcal{A}
&\leq \frac{1}{\alpha_U}\big\|\nabla_{\tilde a}U\big((\boldsymbol{B}_{\theta_2}z_2)(x),z_2^x,\theta_2^x\big)-\nabla_{\tilde a}U\big((\boldsymbol{B}_{\theta_2}z_2)(x),z_1^x,\theta_1^x\big) \big\|_\mathcal{A}\\
&\leq \frac{1}{\alpha_U}\big(\ell_U\|z_1^x-z_2^x\|_\mathcal{A}+\ell_\theta\|\theta_1^x-\theta_2^x \|_\mathcal{A} \big),
\end{aligned}\ee
where the second inequality follows from the fact that $\nabla_{\tilde a} U(\cdot,\tilde z,\tilde \theta)$ is Lipschitz continuous in $\tilde z,\tilde\theta$ with constants $\ell_U,\ell_\theta$ by Assumption~\ref{assum:U}. Finally, applying the norm $\|\cdot\|_{L^2}$ on $L^2([0,1],\R)$ to both sides of inequality \eqref{eq:sensitivity} and using the triangle inequality completes the proof. 

(ii) Let $\tilde z_0, \tilde\theta_0 \in\mathcal{A}$ denote the processes from Assumption~\ref{assum:graphon}. Consider the aggregate $z_0\in\mathcal{A}^\infty$ defined by $z_0^x:=\tilde z_0$ for all $x\in [0,1]$ and the heterogeneity profile $\theta_0\in\mathcal{A}^\infty$ defined by $\theta_0^x:=\tilde\theta_0$ for all $x\in [0,1]$. Then,
$$
\|\boldsymbol{B}_\theta z_0\|_{\mathcal{A}^\infty}^2=\int_0^1 \big\| \argmax_{\tilde{a}\in \mathcal{A}^x} U(\tilde{a}, z_0^x,\theta_0^x)\big\|_\mathcal{A}^2dx<\infty.
$$
Now let $z,\theta\in\mathcal{A}^\infty$. Then it follows from (i) and the triangle inequality,
\be\begin{aligned}
\|\boldsymbol{B}_\theta z\|_{\mathcal{A}^\infty}&=\|\boldsymbol{B}_\theta z-\boldsymbol{B}_{\theta_0} z_0\|_{\mathcal{A}^\infty}+\|\boldsymbol{B}_{\theta_0} z_0\|_{\mathcal{A}^\infty}\\
&\leq \frac{\ell_U}{\alpha_U}\| z- z_0\|_{\mathcal{A}^\infty}+\frac{\ell_\theta}{\alpha_U}\| \theta- \theta_0\|_{\mathcal{A}^\infty}+\|\boldsymbol{B}_{\theta_0} z_0\|_{\mathcal{A}^\infty}\\
&\leq \frac{\ell_U}{\alpha_U}\big(\| z\|_{\mathcal{A}^\infty}+\| z_0\|_{\mathcal{A}^\infty}\big)+\frac{\ell_\theta}{\alpha_U}\big(\| \theta\|_{\mathcal{A}^\infty}+\| \theta_0\|_{\mathcal{A}^\infty}\big)+\|\boldsymbol{B}_{\theta_0} z_0\|_{\mathcal{A}^\infty}<\infty.
\end{aligned}\ee

(iii) Let $z\in\mathcal{A}^\infty$. By Assumption~\ref{assum:Ax}, it holds that $(\boldsymbol{B}_\theta z)(x)\in\mathcal{A}_M$ for all $x\in [0,1]$, and therefore that
$$
\|\boldsymbol{B}_\theta z\|_{\mathcal{A}^\infty}^2=\int_0^1\|(\boldsymbol{B}_\theta z)(x)\|_\mathcal{A}^2dx\leq M^2.
$$
This concludes the proof of the lemma. 
\end{proof}
Recall that $\|\cdot\|_{L^2}$ denotes the norm on $L^2([0,1],\R)$. For a bounded operator $\boldsymbol{S}$ on $L^2([0,1],\R)$, denote its operator norm by 
\be\label{eq:operator-norm}
\|\boldsymbol{S}\|_{\operatorname{op}}:= \sup\big\{\|\boldsymbol{S}f\|_{L^2} : f\in L^2([0,1],\R)\textrm{ with }\|f\|_{L^2}\leq 1\big\}.
\ee
\begin{proof}[Proof of Theorem~\ref{thm:NE}]
(i) Recall \eqref{eq:Ainfty}. It follows from \eqref{eq:graphon-operator} and \eqref{def-b} that an action profile $\bar{a}\in\mathcal{A}^\infty$ is a Nash equilibrium if and only if it satisfies,
$$
\bar{a}=\boldsymbol{B}_\theta \boldsymbol{W}\bar{a},
$$
that is, if the function $\bar{a}$ is a fixed point of the operator  $\boldsymbol{B}_\theta \boldsymbol{W}$ on $\mathcal{A}^\infty$. To show that such $\bar a$ exists, we will prove that $\boldsymbol{B}_\theta \boldsymbol{W}$ is a contraction. Namely, recalling \eqref{eq:operator-norm}, by Lemma~\ref{lemma:B_theta}(i) and the linearity of the graphon operator, we have for any $a_1,a_2\in \mathcal{A}^\infty$ that
\be\begin{aligned}
    \|\boldsymbol{B}_\theta \boldsymbol{W}a_1-\boldsymbol{B}_\theta \boldsymbol{W}a_2\|_{\mathcal{A}^\infty}&\leq \frac{\ell_U}{\alpha_U}\| \boldsymbol{W}a_1- \boldsymbol{W}a_2\|_{\mathcal{A}^\infty}\\
    &\leq \frac{\ell_U}{\alpha_U}\|\boldsymbol{W}\|_{\operatorname{op}}\| a_1- a_2\|_{\mathcal{A}^\infty}\\
    &=\frac{\ell_U}{\alpha_U}\lambda_{1}(\boldsymbol{W})\| a_1- a_2\|_{\mathcal{A}^\infty},
\end{aligned}\ee
where we used that $\|\boldsymbol{W}\|_{\operatorname{op}}=\lambda_{1}(\boldsymbol{W})$ for the last equality (see \cite{avella2018centrality}, Lemma~1).
Therefore, by Banach's fixed point theorem (\cite{bauschke2017}, Chapter~1.12, Theorem~1.50), the operator  $\boldsymbol{B}_\theta \boldsymbol{W}$ has a unique fixed point $\bar a\in\mathcal{A}^\infty$ (see \eqref{eq:Ainfty}), which by definition of $\boldsymbol{B}_\theta$ is also contained in $\mathcal{A}_{ad}^\infty$ (see \eqref{eq:Aadinfty}).  Thus, it is a Nash equilibrium. 

(ii) Since the largest eigenvalue of any graphon is bounded by 1 (see \cite{lovasz2012large}, Chapter~7.5, equation 7.20), the result follows directly from (i). 
\end{proof}

\begin{proof}[Proof of Corollary~\ref{cor:NE-network}]
By Proposition~\ref{prop:correspondence}, in order to study Nash equilibrium properties of the network game $\mathcal{G}(\mathcal A_{ad}^N,U,\theta^N,G^N)$, we can equivalently study the graphon game $\mathcal{G}(\mathcal{A}_{ad}^\infty,U,\theta^N_{\operatorname{step}},W_{G^N})$ with action sets $\mathcal{A}^x:=\mathcal{A}^{i,N}$ for all $x\in\mathcal{P}_i^N$, step function heterogeneity profile $\smash{\theta^{N}_{\operatorname{step}}}$ corresponding to $\theta^N$, and underlying step graphon $W_{G^N}$. Moreover, denote by $\lambda_1\geq\ldots\geq\lambda_N$ the eigenvalues of $G^N$ and note that the eigenvalues of the corresponding step graphon $W_{G^N}$ are then given by $\{\frac{1}{N}\lambda_i\}_{i=1}^N$ (see \cite{gao2019spectral}, Proposition~3). The result now follows directly from Theorem~\ref{thm:NE}.
\end{proof}
The following proposition is needed to prove Theorems~\ref{thm:convergence} and \ref{thm:convergence-sampled}.  
\begin{proposition}\label{prop:continuity}
Consider graphons $W,W'\in\mathcal{W}_0$ and heterogeneity profiles $\theta,\theta'\in\mathcal{A^\infty}$. Suppose that the associated graphon games $\mathcal{G}(\mathcal{A}_{ad}^\infty,U,\theta,W)$ and $\mathcal{G}(\mathcal{A}_{ad}^\infty,U,\theta',W')$ satisfy Assumptions~\ref{assum:U}, \ref{assum:graphon}, and \ref{assum:Ax} with $\lambda_{1}(\boldsymbol{W})\vee \lambda_{1}(\boldsymbol{W'}) <\alpha_U/\ell_U$ and denote by $\bar a$ and $\bar a'$ their unique Nash equilibria (which exist by Theorem~\ref{thm:NE}), respectively. Then the following holds, 
$$
\|\bar a-\bar a'\|_{\mathcal{A}^\infty}\leq \frac{1}{\alpha_U-\ell_U\lambda_{1}(\boldsymbol{W})}\big(\ell_UM\|\boldsymbol{W}-\boldsymbol{W'}\|_{\operatorname{op}}+\ell_\theta\|\theta-\theta'\|_{\mathcal{A}^\infty}\big).
$$
\end{proposition}
\begin{proof}
By the proof of Theorem~\ref{thm:NE}, it holds that $\bar{a}=\boldsymbol{B}_\theta \boldsymbol{W}\bar{a}$ and $\bar{a}'=\boldsymbol{B}_{\theta'} \boldsymbol{W'}\bar{a}'$. Therefore, by  Lemma~\ref{lemma:B_theta}(i) and the triangle inequality,
\be\begin{aligned}\label{eq:continuity}
\|\bar a-\bar a'\|_{\mathcal{A}^\infty}&= \|\boldsymbol{B}_\theta \boldsymbol{W}\bar{a}-\boldsymbol{B}_{\theta'} \boldsymbol{W'}\bar{a}'\|_{\mathcal{A}^\infty}\\
&\leq \frac{\ell_U}{\alpha_U}\|\boldsymbol{W}\bar{a}-\boldsymbol{W'}\bar{a}'\|_{\mathcal{A}^\infty}+\frac{\ell_\theta}{\alpha_U}\|\theta-\theta'\|_{\mathcal{A}^\infty}\\
&\leq \frac{\ell_U}{\alpha_U}\|\boldsymbol{W}\bar{a}-\boldsymbol{W}\bar{a}'\|_{\mathcal{A}^\infty}+\frac{\ell_U}{\alpha_U}\|\boldsymbol{W}\bar{a}'-\boldsymbol{W'}\bar{a}'\|_{\mathcal{A}^\infty}+\frac{\ell_\theta}{\alpha_U}\|\theta-\theta'\|_{\mathcal{A}^\infty}\\
&\leq \frac{\ell_U}{\alpha_U}\|\boldsymbol{W}\|_{\operatorname{op}}\|\bar{a}-\bar{a}'\|_{\mathcal{A}^\infty}+\frac{\ell_U}{\alpha_U}\|\boldsymbol{W}-\boldsymbol{W'}\|_{\operatorname{op}}\|\bar{a}'\|_{\mathcal{A}^\infty}+\frac{\ell_\theta}{\alpha_U}\|\theta-\theta'\|_{\mathcal{A}^\infty},
\end{aligned}\ee
where the last inequality follows from the definition \eqref{eq:operator-norm} of the operator norm. Recalling that $\|\boldsymbol{W}\|_{\operatorname{op}}=\lambda_{1}(\boldsymbol{W})$ (see \cite{avella2018centrality}, Lemma~1) and rearranging \eqref{eq:continuity} yields 
$$
\big(1-\frac{\ell_U}{\alpha_U}\lambda_{1}(\boldsymbol{W})\big)\|\bar{a}-\bar{a}'\|_{\mathcal{A}^\infty}\leq \Big(\frac{\ell_U}{\alpha_U}\|\boldsymbol{W}-\boldsymbol{W'}\|_{\operatorname{op}}\|\bar{a}'\|_{\mathcal{A}^\infty}+\frac{\ell_\theta}{\alpha_U}\|\theta-\theta'\|_{\mathcal{A}^\infty}\Big).
$$
Since $\frac{\ell_U}{\alpha_U}\lambda_{1}(\boldsymbol{W})<1$ by assumption and $\|\bar{a}'\|_{\mathcal{A}^\infty}\leq M$ by Lemma~\ref{lemma:B_theta}(iii), this concludes the proof. 
\end{proof}

Recall that the cut norm and the operator norm were defined in \eqref{eq:cut-norm} and  \eqref{eq:operator-norm}.
The following lemma shows how the two norms compare for graphons. It combines the results of Lemmas~E.2 and E.6 in \cite{janson2010graphons}. 
\begin{lemma}\label{lemma:relation-cut-op}
Let $W\in\mathcal{W}_0$, then it holds that
\be
\|W\|_\Box\leq \|\boldsymbol{W}\|_{\operatorname{op}}\leq  \sqrt{8\|W\|_\Box}.
\ee
\end{lemma}
  
To prove Theorem~\ref{thm:convergence}, we also need an “unlabeled” version of the cut norm (see \cite{lovasz2012large}, Chapter~8.2.2). Let $S_{[0,1]}$ be the set of all invertible measure preserving maps $[0,1] \to [0,1]$. We define the cut distance of two kernels $W,W'\in\mathcal{W}$ by
\begin{equation}
    \delta_\square(W,W') = \inf_{\varphi \in S_{[0,1]}} \|W^\varphi-W'\|_\Box,
\end{equation}
where $W^\varphi(x,y) = W(\varphi(x), \varphi(y))$. Note that $\delta_\square$ is only a pseudometric, as different kernels can
have distance zero.
\begin{proof}[Proof of Theorem~\ref{thm:convergence}]
We first prove the existence of the network game equilibria. For this, notice that $\|W-W_{G^N}\|_\Box\to 0$ implies $\delta_\square(W,W_{G^N})\to 0$. Therefore, it holds that
$$\frac{\lambda_1(G^N)}{ N}=\lambda_1(\boldsymbol{W_{G^N}})\to\lambda_1(\boldsymbol{W}),$$
where the equality follows from Proposition~3 in \cite{gao2019spectral} and the convergence follows from Theorem~11.54 in Chapter~11.6 of \cite{lovasz2012large}. In particular, since $\ell_U\lambda_{1}(\boldsymbol{W})<\alpha_U$ by assumption, there exists an $N_0\in\N$ such that 
$$
\ell_U\lambda_{1}(G^N)<\alpha_U  N,\quad \textrm{for all }N\geq N_0.
$$
Thus, by Corollary~\ref{cor:NE-network}, the network game $\mathcal{G}((\mathcal{A}^0)^N,U,\theta^N,G^N)$  admits a unique Nash equilibrium $\bar a^N$ for all $N\geq N_0$. Second, due to Proposition~\ref{prop:correspondence}, the equilibrium $\bar a^N$ of $\mathcal{G}((\mathcal{A}^0)^N,U,\theta^N,G^N)$ can be identified with the equilibrium $\bar a^N_{\operatorname{step}}$ of the corresponding step graphon game $\smash{\mathcal{G}((\mathcal{A}^0)^{[0,1]},U,\theta^N_{\operatorname{step}},W_{G^N})}$. Now, by Proposition~\ref{prop:continuity} and Lemma~\ref{lemma:relation-cut-op}, it holds for all $N\geq N_0$ that
\be\begin{aligned}
    \|\bar a-\bar a^N_{\operatorname{step}}\|_{\mathcal{A}^\infty}&\leq \frac{1}{\alpha_U-\ell_U\lambda_{1}(\boldsymbol{W})}\big(\ell_UM\|\boldsymbol{W}-\boldsymbol{W_{G^N}}\|_{\operatorname{op}}+\ell_\theta\|\theta-\theta^N_{\operatorname{step}}\|_{\mathcal{A}^\infty}\big)\\
    &\leq \frac{1}{\alpha_U-\ell_U\lambda_{1}(\boldsymbol{W})}\big(\ell_UM\sqrt{8\|W-W_{G^N}\|_\Box}+\ell_\theta\|\theta-\theta^N_{\operatorname{step}}\|_{\mathcal{A}^\infty}\big) \\
    &= C_W\|W-W_{G^N}\|_{\Box}^{1/2}+C_\theta\|\theta-\theta^N_{\operatorname{step}}\|_{\mathcal{A}^\infty},
\end{aligned}\ee
where the constants $C_W$ and $C_\theta$ are defined in Theorem~\ref{thm:convergence}. This concludes the proof. 
\end{proof}

Recall that the sampled graphs ${G_w^N(W)}$, ${G_s^N(W)}$ and the density parameters $(\kappa_N)_N$ were introduced in Definition~\ref{def:sampling}, and that $\smash{W_{G_w^N(W)},W_{G_s^N(W)}}$ and $\smash{\boldsymbol{W_{G_w^N(W)}},\boldsymbol{W_{G_s^N(W)}}}$ denote the corresponding step graphons (see \eqref{eq:step-graphon}) and induced integral operators (see \eqref{eq:graphon-operator}), respectively. Also, recall that $\Q$ denotes the probability measure according to which the sampling takes place (see Remark \ref{rem:sampling}).  The next lemma follows from Theorem~1 in \cite{avella2018centrality} and is needed for the proof of Theorem~\ref{thm:convergence-sampled}.
\begin{lemma}\label{lemma:avella}
For a graphon $W\in\mathcal{W}_0$ satisfying Assumption~\ref{assum:Lipschitz-graphon}, it holds with $\Q$-probability $1-\dl$ that 
\be\label{eq:rho(N)}
\big\|\boldsymbol{W}-\boldsymbol{W_{G_w^N(W)}}\big\|_{\operatorname{op}}\leq 2\sqrt{(L^2-K^2)d_N^2+Kd_N}=:\rho (N),
\ee
where $\dl\in (Ne^{-N/5},e^{-1})$ and $d_N=\tfrac{1}{N}+(\tfrac{8\log (N/\dl)}{N+1})^{0.5}$. Moreover, for sufficiently large $N$, it holds  with $\Q$-probability at least $1-2\dl$ that 
\be\label{eq:rho'(N)}
\big\|\boldsymbol{W}-\kappa_N^{-1}\boldsymbol{W_{G_s^N(W)}}\big\|_{\operatorname{op}}\leq \sqrt{\frac{4\kappa_N^{-1}\log (2N/\dl)}{N}}+\rho (N)=:\rho'(N).
\ee
\end{lemma} 
The following lemma is also needed for the proof of Theorem~\ref{thm:convergence-sampled}.
\begin{lemma}\label{lemma:aux-eigenvalues}
Suppose the assumptions of Theorem~\ref{thm:convergence-sampled} hold. Then, for any $0<\dl<e^{-1}$, there exists an $N_\dl\in\N$ such that for all $N\geq N_\dl$ it holds that $\ell_U\cdot\lambda_1(\kappa_N^{-1}G_s^N(W))<\alpha_U\cdot N$ with $\Q$-probability at least $1-2\dl$.
\end{lemma}
\begin{proof}[Proof of Lemma~\ref{lemma:aux-eigenvalues}]
First, notice that $G^N_w(W)\in [0,1]^{N\times N}$ by Definition~\ref{def:sampling}.
Let $\|\cdot\|_2$ denote the spectral norm of a matrix. Then, due to the fact that the spectral radius of a matrix is bounded by its spectral norm, an application of the triangle inequality, and the fact that the spectral norm is bounded by the Frobenius norm, it holds that
\be\begin{aligned}\label{eq:aux-lemma1}
\lambda_1(\kappa_N^{-1}G^N_s(W))&\leq\|\kappa_N^{-1}G^N_s(W)\|_2\\
&\leq \|\kappa_N^{-1}G^N_s(W)-G^N_w(W)\|_2+\|G^N_w(W)\|_2\\
&\leq \|\kappa_N^{-1}G^N_s(W)-G^N_w(W)\|_2+N.
\end{aligned}\ee
Moreover, it follows from the proof of Theorem~1 in \cite{avella2018centrality} that for any $0<\dl<e^{-1}$ there is an $\smash{\wt{N}_\dl\in\N}$ such that for all $\smash{N\geq \wt{N}_\dl}$ it holds that 
\be\label{eq:aux-lemma2}
\frac{1}{N}\left\|\kappa_N^{-1}G^N_s(W)-G^N_w(W)\right\|_2\leq\sqrt{\frac{\kappa_N^{-1}\log(2N/\dl)}{N}},
\ee
with $\Q$-probability at least $1-2\dl$. Since $\ell_U<\alpha_U$ and the right-hand side of \eqref{eq:aux-lemma2} converges to $0$ as $N\to\infty$ by assumption, it follows from 
\eqref{eq:aux-lemma1}  that there is an $\smash{N_\dl\geq \wt{N}_\dl}$ such that 
$$
\frac{\lambda_1(\kappa_N^{-1}G_s^N(W))}{N}<\frac{\alpha_U}{\ell_U},\quad\textrm{for all }N\geq N_\dl,
$$
with $\Q$-probability at least $1-2\dl$. This concludes the proof.
\end{proof}

\begin{proof}[Proof of Theorem~\ref{thm:convergence-sampled}]
(i) We start with the case of weighted sampled graphs. First notice that all sampled adjacency matrices  $G_w^N(W)$ are contained in $[0,1]^{N\times N}$ by Definition~\ref{def:sampling}. Thus, since $\ell_U<\alpha_U$ by assumption, the sampled network game $\mathcal{G}((\mathcal{A}^0)^N,U,\theta^N,G_w^N(W))$ admits a unique Nash equilibrium $\bar a^N$ for every $N\in\N$, by Corollary~\ref{cor:NE-network}(ii). Second, due to Proposition~\ref{prop:correspondence}, the equilibrium $\bar a^N$ of $\mathcal{G}((\mathcal{A}^0)^N,U,\theta^N,G_w^N(W))$ can be identified with the equilibrium $\bar a^N_{\operatorname{step}}$ of the corresponding step graphon game $\mathcal{G}((\mathcal{A}^0)^{[0,1]},U,\theta^N_{\operatorname{step}},W_{G_w^N(W)})$. Recall \eqref{eq:step-function}. By Proposition~\ref{prop:continuity}, it holds that
\be
    \|\bar a-\bar a^N_{\operatorname{step}}\|_{\mathcal{A}^\infty}\leq \frac{1}{\alpha_U-\ell_U\lambda_{1}(\boldsymbol{W})}\big(\ell_UM\|\boldsymbol{W}-\boldsymbol{W_{G_w^N(W)}}\|_{\operatorname{op}}+\ell_\theta\|\theta-\theta^N_{\operatorname{step}}\|_{\mathcal{A}^\infty}\big).
\ee
Thus, by Lemma~\ref{lemma:avella}, for every $0<\dl<e^{-1}$, it holds for all $N\in\N$ satisfying $Ne^{-N/5}<\dl$ with $\Q$-probability at least $1-\dl$ that
\be\begin{aligned}\label{eq:convergence-weighted}
    \|\bar a-\bar a^N_{\operatorname{step}}\|_{\mathcal{A}^\infty}&\leq \frac{C_W}{\sqrt{8}}\rho(N)+C_\theta\|\theta-\theta^N_{\operatorname{step}}\|_{\mathcal{A}^\infty}\\
    &=\mathcal{O}\Big(\Big(\frac{\log(N/\dl)}{N}\Big)^\frac{1}{4}\vee\|\theta-\theta^N_{\operatorname{step}}\|_{\mathcal{A}^\infty}\Big),
\end{aligned}\ee
where the equality follows from \eqref{eq:rho(N)}. In particular, assume that there exists an $\eps>0$ such that
\be\label{eq:2delta}
\bar\dl:=\Q\left(\limsup_{N\to\infty}\Big\{\|\bar a-\bar a^N_{\operatorname{step}}\|_{\mathcal{A}^\infty}>\eps\Big\}\right)>0.
\ee
Choose $0<\dl<\min\{\bar\dl,e^{-1}\}$. By \eqref{eq:convergence-weighted}, since $\rho(N)$ in \eqref{eq:rho(N)} and $\|\theta-\theta^N_{\operatorname{step}}\|_{\mathcal{A}^\infty}$ converge to 0 as $N\to\infty$, there exists an $N(\eps)\in\N$ such that
\be
\Q\left(\Big\{\|\bar a-\bar a^N_{\operatorname{step}}\|_{\mathcal{A}^\infty}\leq\eps\Big\}\right)\geq 1-\dl,\quad \textrm{for all }N\geq N(\eps),
\ee
which contradicts \eqref{eq:2delta}, and thus yields
$$\|\bar a-\bar a^N_{\operatorname{step}}\|_{\mathcal{A}^\infty}\xrightarrow{N\to\infty}0,\quad \Q\textrm{-almost surely.}
$$
(ii) Next, we focus on the case of simple sampled graphs. Here, notice that the sampled matrices $\kappa_N^{-1}G_s^N(W)$ from Definition~\ref{def:sampling} are not necessarily contained in $[0,1]^{N\times N}$ for density parameters smaller than 1, so Lemma~\ref{lemma:aux-eigenvalues} is needed.
Namely, by Corollary~\ref{cor:NE-network} and Lemma~\ref{lemma:aux-eigenvalues}, for any $0<\dl<e^{-1}$, there exists an $N_\dl\in\N$ such that the sampled network game $\mathcal{G}((\mathcal{A}^0)^N,U,\theta^N,\kappa_N^{-1}G_s^N(W))$ admits a unique Nash equilibrium $\bar b^N$ with $\Q$-probability at least $1-2\dl$ for all $N\geq N_\dl$. Second, due to Proposition~\ref{prop:correspondence} and Remark \ref{rem:correspondence-with-kappa_N}, the equilibrium $\bar b^N$ of $\mathcal{G}((\mathcal{A}^0)^N,U,\theta^N,\kappa_N^{-1}G_s^N(W))$ can be identified with the equilibrium $\bar b^N_{\operatorname{step}}$ of the corresponding step graphon game $\mathcal{G}((\mathcal{A}^0)^{[0,1]},U,\theta^N_{\operatorname{step}},W_{\kappa_N^{-1}G_s^N(W)})$. Now, by Proposition~\ref{prop:continuity}, it holds that
\be
    \|\bar b-\bar b^N_{\operatorname{step}}\|_{\mathcal{A}^\infty}\leq \frac{1}{\alpha_U-\ell_U\lambda_{1}(\boldsymbol{W})}\big(\ell_UM\|\boldsymbol{W}-\boldsymbol{W_{\kappa_N^{-1}G_s^N(W)}}\|_{\operatorname{op}}+\ell_\theta\|\theta-\theta^N_{\operatorname{step}}\|_{\mathcal{A}^\infty}\big).
\ee
Thus, by Lemma~\ref{lemma:avella}, for every $0<\dl<e^{-1}$, it holds for all $N\in\N$ satisfying $Ne^{-N/5}<\dl$ with $\Q$-probability at least $1-2\dl$ that
\be\begin{aligned}\label{eq:convergence-sampled}
    \|\bar b-\bar b^N_{\operatorname{step}}\|_{\mathcal{A}^\infty}&\leq \frac{C_W}{\sqrt{8}}\rho'(N)+C_\theta\|\theta-\theta^N_{\operatorname{step}}\|_{\mathcal{A}^\infty}\\
    &=\mathcal{O}\Big(\Big(\frac{\log(N/\dl)}{N}\Big)^\frac{1}{4}\vee\Big(\frac{\log(N/\dl)}{\kappa_NN}\Big)^\frac{1}{2}\vee\|\theta-\theta^N_{\operatorname{step}}\|_{\mathcal{A}^\infty}\Big),
\end{aligned}\ee
where the equality follows from \eqref{eq:rho'(N)}. Finally, we want to prove almost sure convergence. Assume that there exists an $\eps>0$ such that
\be\label{eq:delta-bar}
\bar\dl:=\Q\left(\limsup_{N\to\infty}\Big\{\|\bar b-\bar b^N_{\operatorname{step}}\|_{\mathcal{A}^\infty}>\eps\Big\}\right)>0.
\ee
Choose $0<\dl<\min\{\bar\dl/2,e^{-1}\}$. By \eqref{eq:convergence-sampled}, since $\rho'(N)$ in \eqref{eq:rho'(N)} and $\|\theta-\theta^N_{\operatorname{step}}\|_{\mathcal{A}^\infty}$ converge to 0 as $N\to\infty$ (because $\smash{\tfrac{\log N}{\kappa_NN}}$ does), there exists an $N(\eps)\in\N$ such that
\be
\Q\left(\Big\{\|\bar b-\bar b^N_{\operatorname{step}}\|_{\mathcal{A}^\infty}\leq\eps\Big\}\right)\geq 1-2\dl>1-\bar\dl,\quad \textrm{for all }N\geq N(\eps),
\ee
which contradicts \eqref{eq:2delta} and thus yields
$$\|\bar b-\bar b^N_{\operatorname{step}}\|_{\mathcal{A}^\infty}\xrightarrow{N\to\infty}0,\quad \Q\textrm{-almost surely.}
$$
This completes the proof.
\end{proof}

\section{Proofs of the Results in Section~\ref{sec:interventions}} \label{sec:interventions-proofs}
Recall the definitions in \eqref{eq:A} and \eqref{eq:Ainfty} of the Hilbert spaces $\mathcal{A}$ and $\mathcal{A}^\infty$ and their norms. 

\begin{proof}[Proof of Theorem~\ref{thm:intervention-existence}] By assumption of Theorem~\ref{thm:intervention-existence}, the players have homogeneous action sets, that is, their sets of admissible actions from \eqref{eq:Aadinfty} satisfy $\mathcal{A}^x=\mathcal{A}^0$ for all $x\in [0,1]$.
Define the image of $(\mathcal{A}^0)^{[0,1]}$ under $\boldsymbol{W}$ as
$$
\boldsymbol{W}\big((\mathcal{A}^0)^{[0,1]}\big):=\big\{\boldsymbol{W}a\mid  a\in (\mathcal{A}^0)^{[0,1]}\big\}\subset\mathcal{A}^\infty.
$$
Now, recalling \eqref{eq:product}, notice that under Assumption~\ref{assum:Ax} it holds that $\boldsymbol{W}((\mathcal{A}^0)^{[0,1]})\subset\mathcal{A}_M$. Together with Assumption~\ref{assum:U2} it follows that the functional
$$(a,z,\theta)\mapsto\int_0^1U\big(a^x, z^x,\theta^x)dx$$
is uniformly bounded from above on $(\mathcal{A}^0)^{[0,1]}\times \boldsymbol{W}((\mathcal{A}^0)^{[0,1]})\times\mathcal{A}^\infty$.  Let $C_B':=\sqrt{C_B}$, where $C_B>0$ is the budget from \eqref{eq:graphon-int-problem}. Recalling \eqref{ad-set-a-m}, given $\smash{z\in\boldsymbol{W}((\mathcal{A}^0)^{[0,1]})}$, define the intervention operator $\boldsymbol{T}_z:(\mathcal{A}^0)^{[0,1]}\to \smash{\mathcal{A}_{C_B'}^\infty}$ by
\be\label{eq:Tz}
\boldsymbol{T}_z (a):=\argmax_{\hat{\theta}\in \smash{\mathcal{A}_{C_B'}^\infty}} \int_0^1U\big(a^x, z^x,\theta^x+\hat{\theta}^x\big)dx,
\ee
which assigns to any fixed action profile $a\in(\mathcal{A}^0)^{[0,1]}$ the optimal intervention, subject to the budget constraint from \eqref{eq:graphon-int-problem}. As $U$ is $\beta_U$-strongly concave in $\tilde\theta$ by Assumption~\ref{assum:U2}, the map 
$$\hat{\theta}\mapsto\int_0^1U(a^x,z^x,\theta^x+\hat{\theta}^x)dx$$
is $\beta_U$-strongly concave as well. Thus, $\boldsymbol{T}_z$ is well-defined because  $\smash{\mathcal{A}_{C_B'}^\infty}$ is bounded, closed, and convex (see \cite{bauschke2017}, Chapters~11.3--11.4, Corollary~11.9 and Proposition~11.15). Moreover, it is Lipschitz continuous by the auxiliary Lemma~\ref{lemma:C_z} stated below this proof. Second, define the Nash equilibrium operator $\boldsymbol{N}:\smash{\mathcal{A}_{C_B'}^\infty}\to\mathcal{A}^\infty$ by
$$
\boldsymbol{N}(\hat\theta):=\textrm{Nash equilibrium of }\mathcal{G}((\mathcal{A}^0)^{[0,1]},U,\theta+\hat\theta,W),
$$
where the domain of $\boldsymbol{N}$ is chosen as $\smash{\mathcal{A}_{C_B'}^\infty}$ so that the budget constraint from \eqref{eq:graphon-int-problem} is satisfied.
This operator is well defined due to Theorem~\ref{thm:NE} and its image is contained in $(\mathcal{A}^0)^{[0,1]}$ by definition. Moreover, it is Lipschitz continuous by Proposition~\ref{prop:continuity}. Now consider the best-response product operator 
\be\label{eq:P}
\boldsymbol{P}:\smash{\mathcal{A}_{C_B'}^\infty}\times(\mathcal{A}^0)^{[0,1]}\to\smash{\mathcal{A}_{C_B'}^\infty}\times(\mathcal{A}^0)^{[0,1]},\quad \boldsymbol{P}(\hat{\theta},a):=\left(\boldsymbol{T}_{(\boldsymbol{W}a)}(a)\, , \boldsymbol{N}(\hat{\theta})\right),
\ee
where $\smash{\mathcal{A}_{C_B'}^\infty}\times(\mathcal{A}^0)^{[0,1]}\subset \mathcal{A}^\infty\times\mathcal{A}^\infty$ is equipped with the product topology and the norm on ${\mathcal{A}^\infty\times\mathcal{A}^\infty}$ is given by 
\be\label{eq:prod-norm}
\|(a_1,a_2)\|^2_{\mathcal{A}^\infty\times\mathcal{A}^\infty}:=\|a_1\|^2_{\mathcal{A}^\infty}+\|a_2\|^2_{\mathcal{A}^\infty},\quad a_1,a_2\in\mathcal{A}^\infty,
\ee
see \cite{conway2019course}, Chapter~1.6, Definition~6.1. Then, by \eqref{eq:prod-norm}, Proposition~\ref{prop:continuity}, Lemma~ \ref{lemma:C_z}, and the fact that $\|\boldsymbol{W}\|_{\operatorname{op}}=\lambda_{1}(\boldsymbol{W})$ (see \cite{avella2018centrality}, Lemma~1), it holds for any $\hat{\theta}_1,\hat{\theta}_2\in\smash{\mathcal{A}_{C_B'}^\infty}$ and $a_1,a_2\in(\mathcal{A}^0)^{[0,1]}$ that 
\be\begin{aligned}
&\big\|\boldsymbol{P}(\hat{\theta}_1,a_1)-\boldsymbol{P}(\hat{\theta}_2,a_2)\big\|_{\mathcal{A}^\infty\times\mathcal{A}^\infty}\\
&=\sqrt{\|\boldsymbol{T}_{(\boldsymbol{W}a_1)}(a_1)-\boldsymbol{T}_{(\boldsymbol{W}a_2)}(a_2)\|_{\mathcal{A}^\infty}^2+\|\boldsymbol{N}(\hat{\theta}_1)-\boldsymbol{N}(\hat{\theta}_2)\|_{\mathcal{A}^\infty}^2}\\
&\leq\sqrt{\Big(\frac{1}{\beta_U}\big(\ell_a\|a_1-a_2\|_{\mathcal{A}^\infty}+\ell_z\|\boldsymbol{W}a_1-\boldsymbol{W}a_2 \|_{\mathcal{A}^\infty}\big)\Big)^2+\Big(\frac{\ell_\theta}{\alpha_U-\ell_U\lambda_{1}(\boldsymbol{W})}\|\hat{\theta}_1-\hat{\theta}_2\|_{\mathcal{A}^\infty}\Big)^2}\\
&\leq\sqrt{\Big(\frac{\ell_a+\ell_z\lambda_1(\boldsymbol{W})}{\beta_U}\|a_1-a_2\|_{\mathcal{A}^\infty}\Big)^2+\Big(\frac{\ell_\theta}{\alpha_U-\ell_U\lambda_{1}(\boldsymbol{W})}\|\hat{\theta}_1-\hat{\theta}_2\|_{\mathcal{A}^\infty}\Big)^2}\\
&\leq \max \Big(\frac{\ell_a+\ell_z\lambda_1(\boldsymbol{W})}{\beta_U},\frac{\ell_\theta}{\alpha_U-\ell_U\lambda_{1}(\boldsymbol{W})}\Big)\cdot\sqrt{\|a_1-a_2\|_{\mathcal{A}^\infty}^2+\|\hat{\theta}_1-\hat{\theta}_2\|_{\mathcal{A}^\infty}^2}.
\end{aligned}\ee
That is, $\boldsymbol{P}$ is Lipschitz continuous. Moreover, $\boldsymbol{P}$ is a nonexpansive operator whenever 
$$
\max \Big(\frac{\ell_a+\ell_z\lambda_1(\boldsymbol{W})}{\beta_U},\frac{\ell_\theta}{\alpha_U-\ell_U\lambda_{1}(\boldsymbol{W})}\Big)\leq 1.
$$
Since $\smash{\mathcal{A}_{C_B'}^\infty}\times(\mathcal{A}^0)^{[0,1]}$ is a nonempty, bounded,
closed, convex subset of $\mathcal{A}^\infty\times\mathcal{A}^\infty$, it follows from Assumption~\ref{assum:U2} and Browder's fixed point theorem (see \cite{bauschke2017}, Chapter~4.4, Theorem~4.29) that $\boldsymbol{P}$ has at least one fixed point $(\bar\theta,\bar a)\in \smash{\mathcal{A}_{C_B'}^\infty}\times(\mathcal{A}^0)^{[0,1]}$, that is, by \eqref{eq:P},
$$
\bar\theta=\boldsymbol{T}_{(\boldsymbol{W}\bar a)}(\bar a),\quad \bar a= \boldsymbol{N}(\bar{\theta}).
$$
This implies that $\bar\theta$ is an optimal intervention. If the inequality in Assumption~\ref{assum:U2} is strict, the fixed point is unique by Banach's fixed point theorem (see \cite{bauschke2017}, Chapter~1.12, Theorem~1.50), so that the optimal intervention is unique as well.
\end{proof}
Next we state and prove the regularity lemma for the operator $\boldsymbol{T}_{z}$ from \eqref{eq:Tz}, which was used in the proof of Theorem~\ref{thm:intervention-existence}. 
\begin{lemma}\label{lemma:C_z}
Under the assumptions of Theorem~\ref{thm:intervention-existence}, the best-response map $\boldsymbol{T}_z$ introduced in \eqref{eq:Tz} is Lipschitz continuous. That is, for any $a_1,a_2\in(\mathcal{A}^0)^{[0,1]}$ and $z_1,z_2\in\boldsymbol{W}((\mathcal{A}^0)^{[0,1]})$, it holds that
$$
\|\boldsymbol{T}_{z_1}(a_1)-\boldsymbol{T}_{z_2}(a_2)\|_\mathcal{A^\infty}
\leq \frac{1}{\beta_U}\big(\ell_a\|a_1-a_2\|_{\mathcal{A}^\infty}+\ell_z\|z_1-z_2 \|_{\mathcal{A}^\infty}\big).
$$
\end{lemma}
\begin{proof}
Recall that the heterogeneity profile $\theta\in\mathcal{A}^\infty$ from \eqref{eq:graphon-int-problem} is fixed.
Let $a_1,a_2\in(\mathcal{A}^0)^{[0,1]}$ and $z_1,z_2\in\boldsymbol{W}((\mathcal{A}^0)^{[0,1]})$. Then, the map $\hat\theta\mapsto\int_0^1 U(a_1^x,z_1^x,\theta^x+\hat\theta^x)dx$ is $\beta_U$-strongly concave. Moreover, its G\^ateaux derivative is given by
\be\begin{aligned}\label{eq:gateaux}
&\langle\nabla_{\hat\theta}\big( \int_0^1U(a_1^x,z_1^x,\cdot)dx\big)({\theta}+\hat\theta),\xi\rangle_{\mathcal{A}^\infty}\\
&:=\lim_{\eps\to 0}\frac{\int_0^1U(a_1^x,z_1^x,\theta^x+\hat\theta^x+\eps\xi^x)dx-\int_0^1U(a_1^x,z_1^x,\theta^x+\hat\theta^x)dx}{\eps},\quad \xi\in\mathcal{A}^\infty.
\end{aligned}\ee
To apply the dominated convergence theorem to the right-hand side of \eqref{eq:gateaux}, recall that by Assumption~\ref{assum:U2}, there exists a constant $\ell_0$  such that for any constant $M'$, the utility functional $U(\tilde a,\tilde z,\tilde \theta)$ is Lipschitz continuous in $\tilde\theta$ on $\mathcal{A}_M\times \mathcal{A}_M\times\mathcal{A}_{M'}$ with Lipschitz constant $\ell_{M'}=\ell_0(1+M')$. Therefore, for $\xi\in\mathcal{A}^\infty$ and $\eps>0$, it follows from the Cauchy-Schwarz inequality that
\be\begin{aligned}\label{eq:DCT}
&\frac{1}{\eps}\left| \int_0^1U(a_1^x,z_1^x,\theta^x+\hat\theta^x+\eps\xi^x)dx-\int_0^1U(a_1^x,z_1^x,\theta^x+\hat\theta^x)dx \right|\\
&\leq \frac{1}{\eps}\int_0^1\left|U(a_1^x,z_1^x,\theta^x+\hat\theta^x+\eps\xi^x)-U(a_1^x,z_1^x,\theta^x+\hat\theta^x)\right|dx \\
&\leq \frac{1}{\eps}\int_0^1 \ell_0\big(1+\|\theta^x\|_{\mathcal{A}}+\|\hat\theta^x\|_{\mathcal{A}}+\|\eps\xi^x\|_{\mathcal{A}}\big)\|\eps\xi^x\|_{\mathcal{A}}dx\\
&\leq \ell_0\big\|1+\|\theta^x\|_{\mathcal{A}}+\|\hat\theta^x\|_{\mathcal{A}}+\|\eps\xi^x\|_{\mathcal{A}}\big\|_{L^2}\|\xi\|_{\mathcal{A}^\infty}\\
&\leq \ell_0\big(1+\|\theta\|_{\mathcal{A}^\infty}+\|\hat\theta\|_{\mathcal{A}^\infty}+\eps\|\xi\|_{\mathcal{A}^\infty}\big)\|\xi\|_{\mathcal{A}^\infty}\\
&<\infty,
\end{aligned}\ee
where we used the triangle inequality in the end. For $(\tilde a,\tilde z,\tilde \theta)\in\mathcal{A}_M\times\mathcal{A}_M\times\mathcal{A}$, denote by 
$$\langle\nabla_{\tilde\theta} U(\tilde a,\tilde z,\tilde \theta),\tilde\xi\rangle_{\mathcal{A}}:=\lim_{\eps\to 0}\frac{1}{\eps}\big(U(\tilde a,\tilde z,\tilde \theta+\eps\tilde\xi)-U(\tilde a,\tilde z,\tilde \theta)\big),\quad\tilde\xi\in\mathcal{A},$$ 
the G\^ateaux derivative on $\mathcal{A}$. Then, it follows from \eqref{eq:gateaux} and \eqref{eq:DCT} that
\be\begin{aligned}\label{eq:gateaux2}
&\langle\nabla_{\hat\theta}\big( \int_0^1U(a_1^x,z_1^x,\cdot)dx\big)({\theta}+\hat\theta),\xi\rangle_{\mathcal{A}^\infty}\\
&=\int_0^1\lim_{\eps\to 0}\frac{U(a_1^x,z_1^x,\theta^x+\hat\theta^x+\eps\xi^x)-U(a_1^x,z_1^x,\theta^x+\hat\theta^x)}{\eps}dx\\
&=\int_0^1 \langle\nabla_{\tilde\theta} U(a_1^x,z_1^x,\theta^x+\hat\theta^x),\xi^x\rangle_\mathcal{A}dx\\
&=\langle\nabla_{\tilde\theta} U(a_1^\cdot,z_1^\cdot,{\theta}^\cdot+\hat\theta^\cdot),\xi\rangle_\mathcal{A^\infty},\quad \xi\in\mathcal{A}^\infty.
\end{aligned}\ee
Therefore, by \eqref{eq:gateaux2}, the map $\hat\theta\mapsto\int_0^1 U(a_1^x,z_1^x,\theta^x+\hat\theta^x)dx$ is G\^ateaux differentiable, and together with its $\beta_U$-strong concavity, we get that $\smash{-\nabla_{\hat\theta}\big( \int_0^1U(a_1^x,z_1^x,\theta^x+\cdot)dx\big)}$ is $\beta_U$-strongly monotone on $\mathcal{A}^\infty$, that is, 
\be\begin{aligned}\label{eq:monotone2}
&\big \langle \hat\theta-\xi, -\nabla_{\hat\theta} (\int_0^1 U(a_1^x,z_1^x,\theta+\hat\theta)dx)-\big(-\nabla_{\hat\theta} (\int_0^1 U(a_1^x,z_1^x,\theta+\xi)dx)\big)\big\rangle_{\mathcal{A}^\infty}\\
&\geq \beta_U\|\hat\theta-\xi  \|_{\mathcal{A}^\infty}^2,\quad \textrm{for all }\hat\theta,\xi\in\mathcal{A}^\infty,
\end{aligned}\ee
(see \cite{bauschke2017}, Chapter~17, Exercise 17.5). Consider the corresponding the bilinear form $F$ on $\mathcal{A}^\infty$ given by 
\be\label{eq:F}
F(\hat\theta,\xi):= -\nabla_{\hat\theta} (\int_0^1 U(a_1^x,z_1^x,\theta^x+\hat\theta^x)dx), \xi\rangle_{\mathcal{A}^\infty},\quad \hat\theta,\xi\in\mathcal{A}^\infty.
\ee
Setting $\xi=0$ in \eqref{eq:monotone2} shows that this bilinear form is coercive with constant $\beta_U$, that is,
$$
F(\hat\theta,\hat\theta)=-\langle \nabla_{\hat\theta} \big(\int_0^1U(a_1^x,z_1^x,\cdot)dx\big)({\theta}+\hat\theta), {\hat{\theta}}\rangle_\mathcal{A^\infty}\geq \beta_U \|{\hat{\theta}}\|^2_\mathcal{A^\infty},\quad \textrm{for all }{\hat{\theta}}\in\mathcal{A}^\infty.
$$
Next, note that the unique best-responses $\boldsymbol{T}_{z_1}(a_1)$ and $\boldsymbol{T}_{z_2}(a_2)$ are given by the unique solutions to the variational inequalities 
\be\label{eq:VI1-2}
\langle  \nabla_{\hat\theta} \big(\int_0^1U(a_1^x,z_1^x,\cdot)dx\big)(\theta+\hat{\theta}), \xi-\hat{\theta}\rangle_\mathcal{A^\infty}\leq 0,\quad \textrm{ for all } \xi \in\mathcal{A}^\infty_{C_B'},
\ee
and
\be\label{eq:VI2-2}
\langle  \nabla_{\hat\theta} \big(\int_0^1U(a_2^x,z_2^x,\cdot)dx\big)(\theta+\hat{\theta}), \xi-\hat{\theta}\rangle_\mathcal{A^\infty}\leq 0,\quad \textrm{ for all } \xi \in\mathcal{A}^\infty_{C_B'},
\ee
respectively.
 Recalling \eqref{eq:F}, inequality \eqref{eq:VI1-2} can equivalently be written as 
\be\label{eq:VI1-2-rewritten}
F(\hat\theta,\xi-\hat\theta)\geq 0, \quad  \textrm{ for all } \xi \in\mathcal{A}^\infty_{C_B'},
\ee
and, adding $F(\hat\theta,\xi-\hat\theta)$ to both of its sides, inequality \eqref{eq:VI2-2} can equivalently be written as 
\be
F(\hat\theta,\xi-\hat\theta)\geq \langle  \nabla_{\hat\theta} \big(\int_0^1U(a_2^x,z_2^x,\cdot)dx\big)(\theta+\hat{\theta}), \xi-\hat{\theta}\rangle_\mathcal{A^\infty}+F(\hat\theta,\xi-\hat\theta),\quad  \textrm{ for all } \xi \in\mathcal{A}^\infty_{C_B'},
\ee
which by \eqref{eq:F} simplifies to 
\be\begin{aligned}\label{eq:VI2-2-rewritten}
&F(\hat\theta,\xi-\hat\theta)\geq \langle  \nabla_{\hat\theta} \big(\int_0^1U(a_2^x,z_2^x,\cdot)dx\big)(\theta+\hat{\theta})- \nabla_{\hat\theta} \big(\int_0^1U(a_1^x,z_1^x,\cdot)dx\big)(\theta+\hat{\theta}), \xi-\hat{\theta}\rangle_\mathcal{A^\infty},\\
&\textrm{for all } \xi \in\mathcal{A}^\infty_{C_B'},
\end{aligned}\ee
Now, recalling \eqref{eq:gateaux2}, an application of Lemma~\ref{lemma:VI} to the variational inequalities \eqref{eq:VI1-2-rewritten} and \eqref{eq:VI2-2-rewritten} (which are equivalent to \eqref{eq:VI1-2} and \eqref{eq:VI2-2}, respectively) with 
$$r_1=0,\quad r_2=\nabla_{\hat\theta} \big(\int_0^1U(a_2^x,z_2^x,\cdot)dx\big)(\theta+\hat{\theta})- \nabla_{\hat\theta} \big(\int_0^1U(a_1^x,z_1^x,\cdot)dx\big)(\theta+\hat{\theta}),$$
the fact that $\nabla_{\tilde\theta} U(\tilde a,\tilde z,\cdot)$ is Lipschitz continuous in $\tilde a,\tilde z$ with constants $\ell_a,\ell_z$ by assumption, and denoting by $\boldsymbol{T}_{z_2}(a_2)(x)$ the intervention $\boldsymbol{T}_{z_2}(a_2)$ at player $x\in[0,1]$ yield that
\be\begin{aligned}\label{eq:sensitivity2}
&\|\boldsymbol{T}_{z_1}(a_1)-\boldsymbol{T}_{z_2}(a_2)\|_\mathcal{A^\infty}\\
&\leq \frac{1}{\beta_U}\big\| \nabla_{\tilde\theta} U\big(a_1^\cdot,z_1^\cdot,\theta^\cdot+\boldsymbol{T}_{z_2}(a_2)(\cdot)\big)- \nabla_{\tilde\theta} U\big(a_2^\cdot,z_2^\cdot,\theta^\cdot+\boldsymbol{T}_{z_2}(a_2)(\cdot)\big) \big\|_\mathcal{A^\infty}\\
&= \frac{1}{\beta_U}\Big\|\big\| \nabla_{\tilde\theta} U\big(a_1^\cdot,z_1^\cdot,\theta^\cdot+\boldsymbol{T}_{z_2}(a_2)(\cdot)\big)- \nabla_{\tilde\theta} U\big(a_2^\cdot,z_2^\cdot,\theta^\cdot+\boldsymbol{T}_{z_2}(a_2)(\cdot)\big) \big\|_\mathcal{A}\Big\|_{L^2}\\
&\leq \frac{1}{\beta_U}\big\|\ell_a\|a_1^\cdot-a_2^\cdot\|_\mathcal{A}+\ell_z\|z_1^\cdot-z_2^\cdot \|_\mathcal{A} \big\|_{L^2}\\
&\leq \frac{1}{\beta_U}\big(\ell_a\|a_1-a_2\|_\mathcal{A^\infty}+\ell_z\|z_1-z_2 \|_\mathcal{A^\infty}\big),
\end{aligned}\ee
where the last inequality follows from the triangle inequality.  
\end{proof}

The following lemma is needed for the proof of Theorem~\ref{thm:intervention-convergence}.
\begin{lemma}\label{lemma:U-Lipschitz}
Assume that the utility functional $U(\tilde a,\tilde z,\tilde \theta)$ is jointly Lipschitz in $(\tilde a,\tilde z,\tilde\theta)$ with Lipschitz constant $L_U$. Then it holds for all $a_1,a_2,z_1,z_2,\theta_1,\theta_2\in\mathcal{A}^\infty$ that 
$$
\left|\int_0^1\hspace{-1mm}U(a^x_1,z^x_1,\theta^x_1)dx-\hspace{-1.9mm}\int_0^1 U(a^x_2,z^x_2,\theta^x_2)dx\right|\leq L_U\sqrt{\|a_1-a_2\|_{\mathcal{A}^\infty}^2\hspace{-1mm}+\|z_1-z_2\|_{\mathcal{A}^\infty}^2\hspace{-1mm}+\|\theta_1-\theta_2\|_{\mathcal{A}^\infty}^2}.
$$
\end{lemma}
\begin{proof} By the linearity of the integral, the Lipschitz continuity of $U$, and the concavity of the square-root function, it holds that
\be\begin{aligned}
    \Big|&  \int_0^1  U(a^x_1,z^x_1,\theta^x_1)dx-\int_0^1 U(a^x_2,z^x_2,\theta^x_2)dx\Big|\\
    &\leq L_U\int_0^1\sqrt{\|a^x_1-a^x_2\|_{\mathcal{A}}^2+\|z^x_1-z^x_2\|_{\mathcal{A}}^2+\|\theta^x_1-\theta^x_2\|_{\mathcal{A}}^2}dx\\
    &\leq L_U\sqrt{\int_0^1\|a^x_1-a^x_2\|_{\mathcal{A}}^2+\|z^x_1-z^x_2\|_{\mathcal{A}}^2+\|\theta^x_1-\theta^x_2\|_{\mathcal{A}}^2dx}\\
    &= L_U\sqrt{\|a_1-a_2\|_{\mathcal{A}^\infty}^2+\|z_1-z_2\|_{\mathcal{A}^\infty}^2+\|\theta_1-\theta_2\|_{\mathcal{A}^\infty}^2}.
\end{aligned}\ee
\end{proof}

\begin{proof}[Proof of Theorem~\ref{thm:intervention-convergence}] Let $\theta^N \in \mathcal A^N $ and $\theta \in \mathcal A^\infty$ be heterogeneity processes for problems \eqref{eq:network-int-problem} and \eqref{eq:graphon-int-problem}. 
First, notice that there is an $N_0\in\N$ such that for all $N\geq N_0$ the network game $\mathcal{G}((\mathcal{A}^0)^N,U,\theta^N+\hat{\theta}^N,G^N)$ admits a unique Nash equilibrium for all $\hat\theta^N\in\mathcal{A}^N$ by Theorem~\ref{thm:convergence}. Throughout the proof, let $N\geq N_0$. We first embed the network intervention problem \eqref{eq:network-int-problem} into the graphon framework. For this, define
$$A^{N,\infty}:=\left\{a\in\mathcal{A}^\infty\big|a\textrm{ is a step function w.r.t. }\mathcal{P}^N\right\},
$$
where $\mathcal{P}^N=\{\mathcal{P}_i^N\}_{i=1}^N$ is the partition from Section~\ref{subsec:correspondence}. Then  problem \eqref{eq:network-int-problem} can be reformulated as
\be\begin{aligned}\label{eq:network-int-problem-reformulated}
    \bar\theta^N\in&\argmax_{\hat{\theta}^N\in\mathcal{A}^{N,\infty}}\ T^N(\hat{\theta}^N)= \argmax_{\hat{\theta}^N\in\mathcal{A}^{N,\infty}}\int_0^1 U\big(\bar a_{\hat{\theta}^N}^{x,N},\bar z_{\hat{\theta}^N}^{x,N},\theta_{\operatorname{step}}^{x,N}+\hat{\theta}^{x,N}\big)dx,\\
    \quad \textrm{s.t.}\quad & \bar a_{\hat{\theta}^N}^{N}\textrm{ is a Nash equilibrium of }\mathcal{G}((\mathcal A^0)^{[0,1]},U,{\theta}^N_{\operatorname{step}}+\hat{\theta}^N,W_{G^N}),\quad \bar z_{\hat{\theta}^N}^{N}=\boldsymbol{W_{G^N}}\bar a_{\hat{\theta}^N}^{N},\\
    & \|\hat{\theta}^N\|^2_{\mathcal{A}^\infty}\leq C_B,
\end{aligned}\ee
Note that it is not clear in general whether problem \eqref{eq:network-int-problem-reformulated} has an optimizer $\bar\theta^N$. However, recalling that $\smash{T^N_{\operatorname{opt}}}$ denotes the optimal value of problem \eqref{eq:network-int-problem-reformulated}, for a given $\eps>0$, we can find an $\eps$-optimizer $\smash{\bar\theta_\eps^N}$, that is, $T^N(\hat{\theta}_\eps^N)\geq T^N_{\operatorname{opt}}-\eps$. Recalling the definition \eqref{eq:approx-intervention} of the approximate intervention, it follows from the fact that $\bar\theta$ is a maximizer of the graphon intervention problem \eqref{eq:graphon-int-problem} that  
\be\begin{aligned}\label{eq:T_opt}
    T^N(\bar\theta^N_W)&\geq T(\bar\theta)-|T^N(\bar\theta^N_W)- T(\bar\theta)|\\
    &\geq T(\bar\theta^N_\eps)-|T^N(\bar\theta^N_W)- T(\bar\theta)|\\
    &\geq T^N(\bar\theta^N_\eps)-\underbrace{|T^N(\bar\theta^N_W)- T(\bar\theta)|}_{=:T_1}-\underbrace{|T(\bar\theta^N_\eps)-T^N(\bar\theta^N_\eps)|}_{=:T_2}\\
    &\geq T^N_{\operatorname{opt}}-\eps-T_1-T_2,
\end{aligned}\ee
where we used  the $\eps$-optimality of $\bar\theta^N_\eps$ for the last inequality. Next, we need to bound the terms $T_1$ and $T_2$. Using Lemma~\ref{lemma:U-Lipschitz} and recalling the fixed heterogeneity processes from \eqref{eq:convergence-ass-2} , we get
\be\label{eq:T_1}
T_1\leq L_U\sqrt{\|\bar a_{\bar\theta^N_W}^N-\bar a_{\bar\theta}\|_{\mathcal{A}^\infty}^2+\|\bar z_{\bar\theta^N_W}^N-\bar z_{\bar\theta}\|_{\mathcal{A}^\infty}^2+\|\bar\theta^N_W-\bar\theta\|_{\mathcal{A}^\infty}^2+\|\theta_{\operatorname{step}}^N-\theta\|_{\mathcal{A}^\infty}^2}
\ee
and 
\be\label{eq:T_2}
T_2\leq L_U \sqrt{\|\bar a_{\bar\theta^N_\eps}^N-\bar a_{\bar\theta^N_\eps}\|_{\mathcal{A}^\infty}^2+\|\bar z_{\bar\theta^N_\eps}^N-\bar z_{\bar\theta^N_\eps}\|_{\mathcal{A}^\infty}^2+\|\theta_{\operatorname{step}}^N-\theta\|_{\mathcal{A}^\infty}^2}.
\ee
By Proposition~\ref{prop:continuity}, for any $\hat\theta_1,\hat\theta_2\in\mathcal{A}^\infty$, there is a constant $C_0>0$ such that  
\be\begin{aligned}\label{eq:int-convergence-aux}
    \|\bar a^N_{\hat\theta_1}-\bar a_{\hat\theta_2}\|_{\mathcal{A}^\infty}&\leq C_0\left(\|\boldsymbol{W_{G^N}}-\boldsymbol{W}\|_{\operatorname{op}}+\|(\theta_{\operatorname{step}}^N+\hat\theta_1)-(\theta+\hat\theta_2)\|_{\mathcal{A}^\infty}\right)\\
    &\leq C_0\left(\sqrt{8\|{W_{G^N}}-W\|_\Box}+\|\hat\theta_1-\hat\theta_2\|_{\mathcal{A}^\infty}+\|\theta_{\operatorname{step}}^N-\theta\|_{\mathcal{A}^\infty}\right),\\
\end{aligned}\ee
where the second inequality follows from Lemma~\ref{lemma:relation-cut-op} and the triangle inequality. In a similar way, we obtain from the triangle inequality, the definition \eqref{eq:operator-norm} of the operator norm,  Lemma~\ref{lemma:B_theta}, Lemma~\ref{lemma:relation-cut-op}, and \eqref{eq:int-convergence-aux} that
\be\begin{aligned}\label{eq:int-convergence-aux2}
    \|\bar z^N_{\hat\theta_1}-\bar z_{\hat\theta_2}\|_{\mathcal{A}^\infty}&=\|\boldsymbol{W_{G^N}}\bar a^N_{\hat\theta_1}-\boldsymbol{W}\bar a_{\hat\theta_2}\|_{\mathcal{A}^\infty}\\
    &\leq \|\boldsymbol{W_{G^N}}\bar a^N_{\hat\theta_1}-\boldsymbol{W}\bar a^N_{\hat\theta_1}\|_{\mathcal{A}^\infty}+\|\boldsymbol{W}\bar a^N_{\hat\theta_1}-\boldsymbol{W}\bar a_{\hat\theta_2}\|_{\mathcal{A}^\infty} \\
    &\leq \|\boldsymbol{W_{G^N}}-\boldsymbol{W}\|_{\operatorname{op}}\|\bar a^N_{\hat\theta_1}\|_{\mathcal{A}^\infty}+\|\boldsymbol{W}\|_{\operatorname{op}}\|\bar a^N_{\hat\theta_1}-\bar a_{\hat\theta_2}\|_{\mathcal{A}^\infty}\\
    &\leq \sqrt{8\|{W_{G^N}}-W\|_\Box}M\\
    &\quad+\lambda_1(\boldsymbol{W})C_0\left(\sqrt{8\|{W_{G^N}}-W\|_\Box}+\|\hat\theta_1-\hat\theta_2\|_{\mathcal{A}^\infty}+\|\theta_{\operatorname{step}}^N-\theta\|_{\mathcal{A}^\infty}\right)\\
    &\leq C_1\left(\sqrt{8\|{W_{G^N}}-W\|_\Box}+\|\hat\theta_1-\hat\theta_2\|_{\mathcal{A}^\infty}+\|\theta_{\operatorname{step}}^N-\theta\|_{\mathcal{A}^\infty}\right),\\    
\end{aligned}\ee
where the last inequality follows from defining the constant $C_1:=\max(M,\lambda_1(\boldsymbol{W})C_0)$. Furthermore, we need to bound $\|\bar\theta^N_W-\bar\theta\|_{\mathcal{A}^\infty}$. For this, consider the step function process $\bar\theta^N:=((\bar\theta(\frac{i}{N}))_{i=1}^N)_{\operatorname{step}}$ corresponding to $(\bar\theta(\frac{i}{N}))_{i=1}^N$. Then, by assumption, $\bar\theta$ satisfies $\bar\theta^x\in \mathcal{A}_{\bar\theta_{\operatorname{max}}}$ for all $x\in[0,1]$ and a constant $\bar\theta_{\operatorname{max}}$, and there exist $L_{\bar\theta}\in\R$ and a finite partition $\{I_1,\ldots,I_{K_{\bar\theta}+1}\}$ of $[0,1]$ such that for any $1\leq k\leq K_{\bar\theta}+1$ and $x,x'\in I_k$ it holds that $\|\bar\theta^x-\bar\theta^{x'}\|_{\mathcal{A}}\leq L_{\bar\theta}|x-x'|$.  
Thus, by the triangle inequality, we get that
\be\begin{aligned}\label{eq:int-convergence-aux3}
    \|\bar\theta^N-\bar\theta\|_{\mathcal{A}^\infty}^2&=\int_0^1\|\bar\theta^N(x)-\bar\theta(x)\|_\mathcal{A}^2dx\\
    &=\sum_{i=1}^N\int_{\mathcal{P}_i^N}\|\bar\theta(\frac{i}{N})-\bar\theta(x)\|_\mathcal{A}^2dx\\
    &\leq \left(\sum_{i=1}^N\int_{\mathcal{P}_i^N}L_{\bar\theta}^2\big|\frac{i}{N}-x\big|^2dx\right)+K_{\bar\theta}\frac{1}{N}(2\bar \theta_{\operatorname{max}}^2+2\bar\theta_{\operatorname{max}}^2)\\
    &\leq \frac{L_{\bar\theta}^2}{N^2}+\frac{4K_{\bar\theta}\bar\theta_{\operatorname{max}}^2}{N}.
\end{aligned}\ee
Moreover, by definition \eqref{eq:approx-intervention} of the approximate intervention, we have
\be\label{eq:theta_W^N}
\bar\theta_W^N=\bar\theta^N\frac{\|\bar\theta\|_{\mathcal{A}^\infty}}{\|\bar\theta^N\|_{\mathcal{A}^\infty}},
\ee
with $\bar\theta_W^N=0$ if $\bar\theta^N=0$. Hence, by the triangle inequality and by plugging in \eqref{eq:theta_W^N}, we get for $\bar\theta^N\neq 0$ that
\be\begin{aligned}\label{eq:aux-ineq}
    \|\bar\theta_W^N-\bar\theta\|_{\mathcal{A}^\infty}&\leq \|\bar\theta_W^N-\bar\theta^N\|_{\mathcal{A}^\infty} +\|\bar\theta^N-\bar\theta\|_{\mathcal{A}^\infty}\\
    &\leq \left|\frac{\|\bar\theta\|_{\mathcal{A}^\infty}}{\|\bar\theta^N\|_{\mathcal{A}^\infty}}-1\right|\|\bar\theta^N\|_{\mathcal{A}^\infty}+\|\bar\theta^N-\bar\theta\|_{\mathcal{A}^\infty}\\
    &=\left|\|\bar\theta\|_{\mathcal{A}^\infty}-\|\bar\theta^N\|_{\mathcal{A}^\infty}\right|+\|\bar\theta^N-\bar\theta\|_{\mathcal{A}^\infty}\\
    &\leq 2\|\bar\theta^N-\bar\theta\|_{\mathcal{A}^\infty},
\end{aligned}\ee
where the last inequality follows from the lower triangle inequality. For $\bar\theta_W^N=0=\bar\theta^N$, inequality \eqref{eq:aux-ineq} holds as well. Finally, combining \eqref{eq:T_1},  \eqref{eq:int-convergence-aux}, \eqref{eq:int-convergence-aux2}, \eqref{eq:aux-ineq} we obtain
\be\begin{aligned}\label{eq:T_1-bound}
T_1&\leq L_U\bigg((C_0^2+C_1^2)\left(\sqrt{8\|{W_{G^N}}-W\|_\Box}+2\|\bar\theta^N-\bar\theta\|_{\mathcal{A}^\infty}+\|\theta_{\operatorname{step}}^N-\theta\|_{\mathcal{A}^\infty}\right)^2\\
&\quad+4\|\bar\theta^N-\bar\theta\|_{\mathcal{A}^\infty}^2+\|\theta_{\operatorname{step}}^N-\theta\|_{\mathcal{A}^\infty}^2\bigg)^\frac{1}{2}\\
&=\mathcal{O}\left(\|{W_{G^N}}-W\|_\Box^\frac{1}{2}\vee N^{-\frac{1}{2}}\vee\|\theta_{\operatorname{step}}^N-\theta\|_{\mathcal{A}^\infty}\right),
\end{aligned}\ee
where the equality follows from \eqref{eq:int-convergence-aux3}. Analogously, we get get from \eqref{eq:T_2},  \eqref{eq:int-convergence-aux}, \eqref{eq:int-convergence-aux2}, \eqref{eq:int-convergence-aux3}, \eqref{eq:aux-ineq} that 
\be\begin{aligned}\label{eq:T_2-bound}
T_2=\mathcal{O}\left(N^{-\frac{1}{2}}\vee\|\theta_{\operatorname{step}}^N-\theta\|_{\mathcal{A}^\infty}\right).
\end{aligned}\ee
Now it follows from \eqref{eq:T_opt}, \eqref{eq:T_1-bound}, \eqref{eq:T_2-bound} that 
\be
T^N_{\operatorname{opt}}-T^N(\bar\theta^N_W)-\eps=\mathcal{O}\left(\|{W_{G^N}}-W\|_\Box^\frac{1}{2}\vee\|\theta_{\operatorname{step}}^N-\theta\|_{\mathcal{A}^\infty}\vee N^{-\frac{1}{2}}\right),
\ee
for all $\eps>0$, so that letting $\eps\to 0$ completes the proof.
\end{proof}

\begin{proof}[Proof of Corollary~\ref{cor:intervention-convergence-sampled}]
The bound on the convergence rate follows from adopting the setting from Theorem~\ref{thm:convergence-sampled} and plugging the bounds from Lemma~\ref{lemma:avella} into Theorem~\ref{thm:intervention-convergence}, where $\kappa_N\equiv 1$ for the statement of the corollary. The almost sure convergence then follows as in the proof of Theorem~\ref{thm:convergence-sampled}.   
\end{proof}

\section{Proofs of the Results in Section~\ref{sec:LQ}}\label{sec:LQ-proofs} 
In this section, we extend the proof of Theorem~1 in \cite{galeotti2020targeting} to the dynamic setting. We start with the finite-player case.
\begin{proof}[Proof of Theorem~\ref{thm:LQ}]
Letting $\eta^N=\theta^N+\hat\theta^N$ and recalling Remark \ref{rem:property}, the goal is to solve the optimization problem
    \be\begin{aligned}
\max_{\eta^N\in\mathcal{A}^N}w\cdot \E\bigg[&\int_0^T (\bar{a}^N_t)^\top\bar{a}^N_tdt\bigg],\\
\textrm{s.t.}\quad & \bar{a}^N=(I^N-\frac{\beta}{N}G^N)^{-1}\eta^N,\\
    & \frac{1}{N}\|\eta^N-\theta^N\|^2_{\mathcal{A}^N}=\frac{1}{N}\|\hat{\theta}^N\|^2_{\mathcal{A}^N}\leq C_B.
\end{aligned}\ee
To this end, we reformulate the maximization problem by means of the basis of $\R^N$ given by the principal components of $G^N$. Then, using the fact that norms do not change under the orthogonal transformation $(U^N)^\top$ introduced before \eqref{eq:theta-LQ}, we get, 
$$
\frac{1}{N}\|\hat\theta^N\|_{\mathcal{A}^N}^2=\frac{1}{N}\E\bigg[\int_0^T\sum_{k=1}^N (\hat{\underline{\theta}}^N_{k,t})^2dt\bigg],
$$
and
$$
w\cdot \E\bigg[\int_0^T (\bar{a}^N_t)^\top\bar{a}^N_tdt\bigg]=w\cdot \|\underline{\bar{a}}^N\|_{\mathcal{A}^N}^2=w\cdot\E\bigg[\int_0^T\sum_{k=1}^N (\bar{\underline{a}}^N_{k,t})^2dt\bigg].
$$
Together with \eqref{eq:Nash-in-PC-LQ} and the definition of $\al^N_k$ in \eqref{eq:alpha-LQ}, the optimization problem can therefore be rewritten as 
\be\begin{aligned}\label{eq-LQ:IT-in-PC}
\max_{\underline{\eta}^N\in\mathcal{A}^N}w\cdot \E\bigg[&\int_0^T \sum_{k=1}^N\al^N_k(\underline{\eta}^N_{k,t})^2dt\bigg],\\
\textrm{s.t.}\quad
    &\frac{1}{N}\E\bigg[\int_0^T\sum_{k=1}^N (\hat{\underline{\theta}}^N_{k,t})^2dt\bigg]\leq C_B.
\end{aligned}\ee
Recalling that $\underline{\theta}^N_k\neq 0$, $\P \otimes dt$-a.e., for each $k$, let $\zeta^N:\Omega\times[0,T]\to\R^N$ be given by $\zeta^N_k=\underline{\hat\theta}^N_k / \underline{\theta}^N_k$. Then \eqref{eq-LQ:IT-in-PC} can be rewritten as
\be\begin{aligned}\label{eq-LQ:IT-with-x}
\max_{\zeta^N}w\cdot \E\bigg[&\int_0^T \sum_{k=1}^N\al^N_k(1+\zeta^N_{k,t})^2(\underline{\theta}_{k,t}^N)^2dt\bigg],\\
\textrm{s.t.}\quad
    &\frac{1}{N}\E\bigg[\int_0^T\sum_{k=1}^N (\underline{\theta}_{k,t}^N)^2(\zeta^N_{k,t})^2dt\bigg]\leq C_B.
\end{aligned}\ee
\textbf{Observations.} Given an optimal solution $\bar{\zeta}^N$, under Assumption~\ref{assum:property}, either $w<0$ and $\frac{1}{N}\|\underline{\theta}^N\|_{\mathcal{A}^N}^2> C_B$, or $w>0$ (see Remark \ref{rem:property}), hence it follows from \eqref{eq-LQ:IT-with-x} that there exist a $k_0\in\{1,\ldots,N\}$  and a set $A\in\mathcal{F}\otimes\mathcal{B}([0,T])$ with positive measure such that $\bar\zeta^N_{k_0}>-1$ on $A$. 
We can draw the following two conclusions.
\begin{itemize}
\item First, at the optimal solution $\bar\zeta^N$ (respectively $\underline{\bar{\eta}}^N$), the budget $C_B$ is fully used. Namely, suppose that it is not fully used. Then,  there exists a sufficiently small $0<\varepsilon<1$ such that an increase or decrease of $\bar\zeta^N_{k_0}$ on $A$ by $\varepsilon \cdot\min\{1,\bar\zeta^N_{k_0}+1\}$ still satisfies the budget constraint. Now if $w>0$ (respectively $w<0$), an increase (respectively decrease) will surely increase (respectively decrease) $(1+\bar\zeta^N_{k_0})^2$ on $A$, by the strict monotonicity of the function $\zeta\mapsto (1+\zeta)^2$ on $[-1,\infty)$.  Together with the facts that $(\underline{\theta}^N_{k_0})^2>0$, $\P\otimes dt$-a.e., and $\al^N_{k_0}>0$ this implies suboptimality and thus yields a contradiction.

\item Second, it holds for the optimal solution $\bar\zeta^N$ $\P \otimes dt$-a.e.~that
\be\label{eq-LQ:optimal-x-constraint}
\bar\zeta^N_{k}\geq 0,\quad  \textrm{for all } k,\textrm{ if }w>0,\quad \textrm{and}\quad \bar\zeta^N_{k}\in [-1,0],\quad  \textrm{for all }k,\textrm{ if }w<0.
\ee
Namely, if $w>0$ and $\bar\zeta^N_k<0$, the aggregate utility in the first part of \eqref{eq-LQ:IT-with-x} can be raised without any costs by flipping the sign of $\bar\zeta^N_k$ in the second part of \eqref{eq-LQ:IT-with-x}. If $w<0$, it follows analogously from \eqref{eq-LQ:IT-with-x} that flipping the sign of $\bar\zeta^N_{k}> 0$ increases the aggregate utility without any costs. Moreover, if $\bar\zeta^N_{k}< 1$, setting $\bar\zeta^N_{k}=-1$ increases the aggregate utility while simultaneously lowering the cost.
\end{itemize}
Next, notice that the Lagrangian $\mathcal{L}:\mathcal{A}^N\times\R\to\R$ corresponding to the optimization problem \eqref{eq-LQ:IT-in-PC} is given by 
$$
\mathcal{L}(\underline{\eta}^N,\mu)=w \cdot\E\bigg[\int_0^T \sum_{k=1}^N\al^N_k(\underline{\eta}_{k,t}^N)^2dt\bigg]
+\mu \bigg(NC_B-\E\bigg[\int_0^T\sum_{k=1}^N ( \underline{\eta}^N_{k,t}-\underline{\theta}^N_{k,t})^2dt\bigg]\bigg).
$$
Consider an arbitrary process $h\in\mathcal{A}$. For $k=1,\ldots,N$, define $h_k:=he_{k}\in\mathcal{A}^N$ with $e_{k}\in\R^N$ being the $k$-th unit vector and note that the corresponding G\^ateaux derivative of $\mathcal{L}$ is given by
\be \begin{aligned}\label{eq-LQ:gateaux}
    \langle  \nabla  \mathcal{L}(\underline{{\eta}}^N,\mu) , h_k  \rangle &= \lim_{\varepsilon \to 0} \frac{ \mathcal{L}(\underline{{\eta}}^N+\varepsilon h_{k},\mu) -  \mathcal{L}(\underline{{\eta}}^N,\mu)}{\varepsilon}\\
    &= \E\bigg[\int_0^T h_t \big(2w\al^N_k\underline{\eta}^N_{k,t}-2\mu( \underline{\eta}^N_{k,t}-\underline{\theta}^N_{k,t})\big)dt\bigg]\bigg).
\end{aligned}\ee
Recalling that the constraint is binding at the optimal intervention, the G\^ateaux derivative \eqref{eq-LQ:gateaux} must vanish at the optimum for any $h\in\mathcal{A}$, which implies 
\be\label{LQ-aux}
w\al^N_k\underline{\bar{\eta}}^N_{k,t}-\mu( \underline{\bar{\eta}}^N_{k,t}-\underline{\theta}^N_{k,t})=0,\quad \P \otimes dt \textrm{-a.e.}
\ee
Thus, it follows from \eqref{LQ-aux}, $w\al^N_k\neq 0$ (see \eqref{eq:alpha-LQ}), and $\underline{\theta}^N_k\neq 0$, $\P \otimes dt$-a.e. (see Assumption~\ref{assum:property}) that $\mu\neq w\al^N_k$, and therefore  that
\be\label{eq-LQ:optimal-bl}
\underline{\bar{\eta}}^N_{k,t}=\frac{\mu\underline{\theta}^N_{k,t}}{\mu-w\al^N_k},\quad \P \otimes dt \textrm{-a.e.}, \ \textrm{for all } k=1,...,N, 
\ee
which in turn implies
\be\label{eq-LQ:optimal-xl}
\bar\zeta^N_{k,t}=\frac{\frac{\mu\underline{\theta}^N_{k,t}}{\mu-w\al^N_k}-\underline{\theta}^N_{k,t}}{\underline{\theta}^N_{k,t}}=\frac{w\al^N_k}{\mu-w\al^N_k},\quad \P \otimes dt \textrm{-a.e.}, \ \textrm{for all } k=1,...,N. 
\ee 
To complete the proof, note that \eqref{eq-LQ:optimal-x-constraint} and \eqref{eq-LQ:optimal-xl} together imply that $\mu>w\al^N_k$ for all $k$, so that all denominators in \eqref{eq-LQ:optimal-xl} are positive. Plugging \eqref{eq-LQ:optimal-xl} into \eqref{eq-LQ:IT-with-x}, yields that $\mu$ is uniquely determined by the equation 
\be
C_B=\frac{1}{N}\E\bigg[\int_0^T\sum_{k=1}^N (\underline{\theta}^N_{k,t})^2\big(\frac{w\al^N_k}{\mu-w\al^N_k}\big)^2dt\bigg]=\frac{1}{N}\sum_{k=1}^N\big(\frac{w\al^N_k}{\mu-w\al^N_k}\big)^2\|\underline{\theta}^N_k\|_{\mathcal{A}}^2.
\ee
Here, letting $\mu_{min}:=\max_{k}(w\al^N_k)$, the existence and uniqueness follow from the fact that $C_B>0$ and that, since $\underline{\theta}^N_k\neq 0$, $\P \otimes dt$-a.e., for all $k$ (see Assumption~\ref{assum:property}), the function 
$$
\mu\mapsto\frac{1}{N}\sum_{k=1}^N\big(\frac{w\al^N_k}{\mu-w\al^N_k}\big)^2\|\underline{\theta}^N_k\|_{\mathcal{A}}^2,
$$
is strictly decreasing on the interval $(\mu_{min},\infty)$ with $\lim_{\mu\downarrow \mu_{min}}=\infty$ and $\lim_{\mu\to \infty}=0$. Finally, by Definition~\ref{eq-LQ:main-cosine-similarity}, the fact that the matrix $U^N$ introduced before \eqref{eq:theta-LQ} is orthonormal, and the definition of $\zeta_k^N$ above \eqref{eq-LQ:IT-with-x},
\be
\rho\big({\bar\theta^N},U^N_{\bullet k}\big)
=\frac{\langle\bar\theta^N,U^N_{\bullet k}\rangle_{\R^N}}{\|{\bar\theta^N}\|_{\R^N}\|U^N_{\bullet k}\|_{\R^N}}
=\frac{\underline{\bar\theta}^N_k}{\|{\bar\theta^N}\|_{\R^N}}
=\frac{\underline{\theta}^N_k \bar\zeta^N_k}{\|{\bar\theta^N}\|_{\R^N}}
=\frac{\|\theta^N\|_{\R^N}}{\|{\bar\theta^N}\|_{\R^N}}\rho\big(\theta^N,U^N_{\bullet k}\big)\bar\zeta^N_k.
\ee
\end{proof}
\begin{remark}
Note that due to the division in the definition of $\zeta$ above \eqref{eq-LQ:IT-with-x}, $\zeta$ is not contained in the Hilbert space $L^2(\Omega\times[0,T],\R^N)$ in general. Therefore, in contrast to \cite{galeotti2020targeting}, in order to circumvent technical intricacies, we solve the constrained optimization problem \eqref{eq-LQ:IT-in-PC} instead of \eqref{eq-LQ:IT-with-x}.
\end{remark}
We now prove the corresponding infinite-player version of Theorem~\ref{thm:LQ}, which follows along similar lines.
\begin{proof}[Proof of Theorem~\ref{thm:LQ-graphon}]
Letting $\eta=\theta+\hat\theta$, the goal is to solve the optimization problem
    \be\begin{aligned}
\max_{\eta\in\mathcal{A}^\infty}w\cdot \E\bigg[&\int_0^T \int_0^1(\bar{a}^x_t)^2dxdt\bigg],\\
    \quad \textrm{s.t.}\quad & \bar{a}=(\boldsymbol{I-\beta\boldsymbol{W}})^{-1}\eta, \\
    & \|\eta-\theta\|^2_{\mathcal{A}^\infty}=\|\hat{\theta}\|^2_{\mathcal{A}^\infty}\leq C_B.
\end{aligned}\ee
For this, we reformulate the maximization problem by means of the spectral decomposition of $\boldsymbol{W}=\boldsymbol{U}^*\boldsymbol{M}\boldsymbol{U}$ introduced after \eqref{eq:graphon-int-problem-LQ}. Using the fact that norms do not change under the unitary  operator $\boldsymbol{U}$,
$$
\|\hat{\theta}\|^2_{\mathcal{A}^\infty}=\E\bigg[\int_0^T\int_0^1 (\hat{\theta}_t^x)^2dxdt\bigg]=\E\bigg[\int_0^T \|\underline{\hat{\theta}}_t\|_X^2 dt\bigg],
$$
and
$$
w\cdot\E\bigg[\int_0^T\int_0^1 (\bar{a}_t^x)^2dxdt\bigg]=w\cdot\E\bigg[\int_0^T \|\bar{\underline{a}}_t\|_X^2 dt\bigg]=w\cdot\E\bigg[\int_0^T \sum_{\lambda\in\sigma(\boldsymbol{W})}\|\bar{\underline{a}}_t(\lambda)\|_{\R^{m(\lambda)}}^2 dt\bigg].
$$
Define $\mathcal{B}^\infty:=\mathcal{A}\otimes X$ with the standard inner product. Together with \eqref{eq-LQ:Nash-in-PC-graphon} and the definition of the scalars $\al_\lambda$ above, the optimization problem can therefore be rewritten as 
\be\begin{aligned}\label{eq-LQ:IT-in-PC-graphon}
\max_{\underline{\eta}\in\mathcal{B}^\infty}w \cdot\E\bigg[&\int_0^T \sum_{\lambda\in\sigma(\boldsymbol{W})}\al_\lambda\|\underline{\eta}_t(\lambda)\|_{\R^{m(\lambda)}}^2dt\bigg],\\
\textrm{s.t.}\quad
    &\|\underline{\hat{\theta}}\|_{\mathcal{B}^\infty}^2=\E\bigg[\int_0^T \|\underline{\eta}_t-\underline{\theta}_t\|_X^2 dt\bigg]\leq C_B.
\end{aligned}\ee
Recalling the last part of Assumption~\ref{assum:property-graphon}, let $\zeta$ be given by $\smash{\zeta(\lambda)_j=\underline{\hat{\theta}}(\lambda)_j / \underline{\theta}(\lambda)_j}$ for $j=1,\ldots,m(\lambda)$ and $\lambda\in\sigma(\boldsymbol{W})$. Then \eqref{eq-LQ:IT-in-PC-graphon} can be rewritten as
\be\begin{aligned}\label{eq-LQ:IT-with-z-graphon}
\max_{\zeta}w\cdot \E\bigg[&\int_0^T \sum_{\lambda\in\sigma(\boldsymbol{W})}\al_\lambda\sum_{j=1}^{m(\lambda)}(1+\zeta_t(\lambda)_j)^2\underline{\theta}_t(\lambda)_j^2dt\bigg],\\
\textrm{s.t.}\quad
    &\E\bigg[\int_0^T\sum_{\lambda\in\sigma(\boldsymbol{W})} \sum_{j=1}^{m(\lambda)}\underline{\theta}_t(\lambda)_j^2\zeta_t(\lambda)_j^2 dt\bigg]\leq C_B.
\end{aligned}\ee
\textbf{Observations.} Given an optimal solution $\bar{\zeta}$, under Assumption~\ref{assum:property-graphon}, either $w<0$ and $\|\underline{\theta}\|_{\mathcal{B}^\infty}^2> C_B$, or $w>0$ (see Remark \ref{rem:property-graphon}), hence it follows from \eqref{eq-LQ:IT-with-z-graphon} that there exist $\lambda\in\sigma(\boldsymbol{W})$, $j\in\{1,\ldots,m(\lambda)\}$, and a set $A\in\mathcal{F}\otimes\mathcal{B}([0,T])$ with positive measure such that $\zeta(\lambda)_j>-1$ on $A$. We can draw the following two conclusions.
\begin{itemize}
\item First, at the optimal solution $\bar{\zeta}$ (respectively~$\bar{\underline{\eta}}$), the budget $C_B$ is fully used. Namely, suppose that it is not fully used. Then, there exists a sufficiently small $0<\varepsilon<1$ such that an increase or decrease of $\bar{\zeta}(\lambda)_j$ on $A$ by $\varepsilon \cdot\min\{1,\bar{\zeta}(\lambda)_j+1\}$ still satisfies the budget constraint. Now if $w>0$ (respectively~$w<0$), an increase (respectively~decrease) will surely increase (respectively~decrease) $(1+\bar{\zeta}(\lambda)_j)^2$ on $A$, by the strict monotonicity of the function $\zeta\mapsto (1+\zeta)^2$ on $[-1,\infty)$.  Together with the facts that $\underline{\theta}(\lambda)_j^2>0$, $\P\otimes dt$-a.e., and $\al_\lambda>0$ this implies suboptimality and thus yields a contradiction.

\item Second, it holds for the optimal solution $\bar{\zeta}$ $\P \otimes dt$-a.e.~that 
\be\label{eq-LQ:optimal-z-constraint}
\bar{\zeta}(\lambda)_j\geq 0\ \textrm{ for all } \lambda,j,\textrm{ if }w>0,\quad \textrm{and}\quad \bar{\zeta}(\lambda)_j\in [-1,0]\ \textrm{ for all } \lambda,j,\textrm{ if }w<0.
\ee
Namely, if $w>0$ and $\bar{\zeta}(\lambda)_j<0$, the aggregate utility in the first part of \eqref{eq-LQ:IT-with-z-graphon} can be raised without any costs by flipping the sign of $\bar{\zeta}(\lambda)_j$ in the second part of \eqref{eq-LQ:IT-with-z-graphon}. If $w<0$, it follows analogously from \eqref{eq-LQ:IT-with-z-graphon} that flipping the sign of $\bar{\zeta}(\lambda)_j> 0$ increases the aggregate utility without any costs. Moreover, if $\bar{\zeta}(\lambda)_j< 1$, setting $\bar{\zeta}(\lambda)_j=-1$ increases the aggregate utility while simultaneously lowering the cost.
\end{itemize}

Next, notice that the Lagrangian $\mathcal{L}:\mathcal{B}^\infty\times\R\to\R$ corresponding to the optimization problem \eqref{eq-LQ:IT-in-PC-graphon} is given by 
$$
\mathcal{L}(\underline{\eta},\mu)=w \cdot\E\bigg[\int_0^T \hspace{-2mm}\sum_{\lambda\in\sigma(\boldsymbol{W})}\al_\lambda\|\underline{\eta}_t(\lambda)\|_{\R^{m(\lambda)}}^2dt\bigg]
+\mu \bigg(C_B-\E\bigg[\int_0^T \hspace{-2mm}\sum_{\lambda\in\sigma(\boldsymbol{W})}\hspace{-2mm}\|\underline{\eta}_t(\lambda)-\underline{\theta}_t(\lambda)\|_{\R^{m(\lambda)}}^2dt\bigg]\bigg).
$$
Consider an arbitrary process $h\in\mathcal{A}$. For $\lambda\in {\sigma(\boldsymbol{W})}$ and $ j\in\{1,\ldots, m(\lambda)\}$, define $h^{\lambda, j}\in\mathcal{B}^\infty$ by
$$
h^{\lambda, j}(\tilde\lambda)_{\tilde j}=\begin{cases}
    h,\quad \textrm{for }(\tilde\lambda,\tilde j)=(\lambda,j),\\
    0,\quad \textrm{otherwise}.
\end{cases}
$$
Note that the corresponding G\^ateaux derivative of $\mathcal{L}$ is given by
\be \begin{aligned}\label{eq-LQ:gateaux2}
    \langle  \nabla  \mathcal{L}(\underline{\eta},\mu) , h^{\lambda,j}  \rangle &= \lim_{\varepsilon \to 0} \frac{ \mathcal{L}(\underline{\eta}+\varepsilon h^{\lambda,j},\mu) -  \mathcal{L}(\underline{\eta},\mu)}{\varepsilon}\\
    &= \E\bigg[\int_0^T h_t \big(2w\al_\lambda\underline{\eta}_t(\lambda)_j-2\mu( \underline{\eta}_t(\lambda)_j-\underline{\theta}_{t}(\lambda)_j)\big)dt\bigg]\bigg).
\end{aligned}\ee
Recalling that the constraint is binding at the optimal intervention,  the G\^ateaux derivative \eqref{eq-LQ:gateaux2} must vanish at the optimum for any $h\in\mathcal{A}$, which implies 
\be\label{LQ-aux2}
w\al_\lambda\bar{\underline{\eta}}_t(\lambda)_j-\mu( \bar{\underline{\eta}}_t(\lambda)_j-\underline{\theta}_{t}(\lambda)_j)=0,\quad \P \otimes dt \textrm{-a.e.~for all }\lambda\in {\sigma(\boldsymbol{W})},\ j=1,\ldots, m(\lambda).
\ee
Thus, it follows from \eqref{LQ-aux2}, $w\al_\lambda\neq 0$ (see before \eqref{eq-LQ:Nash-in-PC-graphon}), and $\underline{\theta}(\lambda)_j\neq 0$, $\P \otimes dt$-a.e. (see Assumption~\ref{assum:property-graphon}) that $\mu\neq w\al_\lambda$, and therefore  that
\be
\bar{\underline{\eta}}_t(\lambda)_j=\frac{\mu\underline{\theta}_t(\lambda)_j}{\mu-w\al_\lambda},\quad \P \otimes dt \textrm{-a.e.~for all }\lambda\in {\sigma(\boldsymbol{W})},\ j=1,\ldots, m(\lambda),
\ee
which in turn implies
\be\label{eq-LQ:optimal-z-lambda}
\bar{\zeta}_t(\lambda)_j=\frac{\frac{\mu\underline{\theta}_t(\lambda)_j}{\mu-w\al_\lambda}-\underline{\theta}_t(\lambda)_j}{\underline{\theta}_t(\lambda)_j}=\frac{w\al_\lambda}{\mu-w\al_\lambda},\quad \P \otimes dt \textrm{-a.e.~for all }\lambda\in {\sigma(\boldsymbol{W})},\ j=1,\ldots, m(\lambda).
\ee
To complete the proof, note that \eqref{eq-LQ:optimal-z-constraint} and \eqref{eq-LQ:optimal-z-lambda} together imply that $\mu>w\al_\lambda$ for all $\lambda$, so that all denominators in \eqref{eq-LQ:optimal-z-lambda} are positive. Plugging \eqref{eq-LQ:optimal-z-lambda} into \eqref{eq-LQ:IT-with-z-graphon}, yields that $\mu$ is uniquely determined by the equation 
\be
C_B=\E\bigg[\int_0^T\sum_{\lambda\in\sigma(\boldsymbol{W})} \sum_{j=1}^{m(\lambda)}\underline{\theta}_t(\lambda)_j^2\big(\frac{w\al_\lambda}{\mu-w\al_\lambda}\big)^2 dt\bigg]=\sum_{\lambda\in\sigma(\boldsymbol{W})}\big(\frac{w\al_\lambda}{\mu-w\al_\lambda}\big)^2\big\|\|\underline{\theta}(\lambda)\|_{\R^{m(\lambda)}}\big\|_{\mathcal{A}}^2.
\ee
Here, letting $\mu_{min}:=\max_{\lambda}(w\al_\lambda)$, the existence and uniqueness follow from the fact that $C_B>0$ and that, since $\underline{\theta}(\lambda)_j\neq 0$, $\P \otimes dt$-a.e., for all $\lambda$ and $j$ (see Assumption~\ref{assum:property-graphon}), the function 
$$
\mu\mapsto\sum_{\lambda\in\sigma(\boldsymbol{W})}\big(\frac{w\al_\lambda}{\mu-w\al_\lambda}\big)^2\big\|\|\underline{\theta}(\lambda)\|_{\R^{m(\lambda)}}\big\|_{\mathcal{A}}^2,
$$
is strictly decreasing on the interval $(\mu_{min},\infty)$ with $\lim_{\mu\downarrow \mu_{min}}=\infty$ and $\lim_{\mu\to \infty}=0$. Finally, by Definition~\ref{eq:main-graphon-cosine-similarity-LQ}, the fact that the eigenfunctions $e_{\lambda,j}$ introduced after \eqref{eq:graphon-int-problem-LQ} are orthonormal, and the definition of $\zeta(\lambda)_j$ above \eqref{eq-LQ:IT-with-z-graphon},
\be
\rho\big(\bar{\theta},e_{\lambda,j}\big)
=\frac{\langle\bar{\theta},e_{\lambda,j}\rangle_{L^2}}{\|\bar{\theta}\|_{L^2}\|e_{\lambda,j}\|_{L^2}}
=\frac{\bar{\underline{\theta}}(\lambda)_j}{\|\bar{\theta}\|_{L^2}}
=\frac{\underline{\theta}(\lambda)_j\bar{\zeta}(\lambda)_j}{\|\bar{\theta}\|_{L^2}}
=\frac{\|\theta\|_{L^2}}{\|\bar{\theta}\|_{L^2}}\rho\big(\theta,e_{\lambda,j}\big)\bar{\zeta}(\lambda)_j.
\ee
\end{proof}


\begin{thebibliography}{28}
\providecommand{\natexlab}[1]{#1}
\providecommand{\url}[1]{\texttt{#1}}
\expandafter\ifx\csname urlstyle\endcsname\relax
  \providecommand{\doi}[1]{doi: #1}\else
  \providecommand{\doi}{doi: \begingroup \urlstyle{rm}\Url}\fi

\bibitem[Abi~Jaber et~al.(2023)Abi~Jaber, Neuman, and Vo{\ss}]{abijaber2023equilibrium}
E.~Abi~Jaber, E.~Neuman, and M.~Vo{\ss}.
\newblock Equilibrium in functional stochastic games with mean-field interaction.
\newblock \emph{arXiv preprint arXiv:2306.05433}, 2023.

\bibitem[Aurell et~al.(2022)Aurell, Carmona, and Lauri{\`e}re]{aurell2022stochastic}
A.~Aurell, R.~Carmona, and M.~Lauri{\`e}re.
\newblock Stochastic graphon games: {II}. the linear-quadratic case.
\newblock \emph{Applied Mathematics \& Optimization}, 85\penalty0 (3):\penalty0 39, 2022.

\bibitem[Avella-Medina et~al.(2018)Avella-Medina, Parise, Schaub, and Segarra]{avella2018centrality}
M.~Avella-Medina, F.~Parise, M.~T. Schaub, and S.~Segarra.
\newblock Centrality measures for graphons: Accounting for uncertainty in networks.
\newblock \emph{IEEE Transactions on Network Science and Engineering}, 7\penalty0 (1):\penalty0 520--537, 2018.

\bibitem[Bauschke and Combettes(2017)]{bauschke2017}
H.~H. Bauschke and P.~L. Combettes.
\newblock \emph{Convex analysis and monotone operator theory in Hilbert spaces}.
\newblock Springer, 2017.

\bibitem[Bayraktar et~al.(2023)Bayraktar, Wu, and Zhang]{bayraktar2023propagation}
E.~Bayraktar, R.~Wu, and X.~Zhang.
\newblock Propagation of chaos of forward--backward stochastic differential equations with graphon interactions.
\newblock \emph{Applied Mathematics \& Optimization}, 88\penalty0 (1):\penalty0 25, 2023.

\bibitem[Borgs et~al.(2008)Borgs, Chayes, Lov{\'a}sz, S{\'o}s, and Vesztergombi]{borgs2008convergent}
C.~Borgs, J.~Chayes, L.~Lov{\'a}sz, V.~S{\'o}s, and K.~Vesztergombi.
\newblock Convergent sequences of dense graphs {I}: Subgraph frequencies, metric properties and testing.
\newblock \emph{Advances in Mathematics}, 219\penalty0 (6):\penalty0 1801--1851, 2008.

\bibitem[Borgs et~al.(2012)Borgs, Chayes, Lov{\'a}sz, S{\'o}s, and Vesztergombi]{borgs2012convergent}
C.~Borgs, J.~Chayes, L.~Lov{\'a}sz, V.~S{\'o}s, and K.~Vesztergombi.
\newblock Convergent sequences of dense graphs {II}. multiway cuts and statistical physics.
\newblock \emph{Annals of Mathematics}, 176:\penalty0 151--219, 2012.

\bibitem[Carmona et~al.(2022)Carmona, Cooney, Graves, and Lauriere]{carmona2022stochastic}
R.~Carmona, D.~Cooney, C.~Graves, and M.~Lauriere.
\newblock Stochastic graphon games: I. the static case.
\newblock \emph{Mathematics of Operations Research}, 47\penalty0 (1):\penalty0 750--778, 2022.

\bibitem[Conway(2019)]{conway2019course}
J.~B. Conway.
\newblock \emph{A course in functional analysis}.
\newblock Springer, 2019.

\bibitem[Galeotti et~al.(2020)Galeotti, Golub, and Goyal]{galeotti2020targeting}
A.~Galeotti, B.~Golub, and S.~Goyal.
\newblock Targeting interventions in networks.
\newblock \emph{Econometrica}, 88\penalty0 (6):\penalty0 2445--2471, 2020.

\bibitem[Gao and Caines(2019)]{gao2019spectral}
S.~Gao and P.~E. Caines.
\newblock Spectral representations of graphons in very large network systems control.
\newblock In \emph{2019 IEEE 58th conference on decision and Control (CDC)}, pages 5068--5075. IEEE, 2019.

\bibitem[Gao et~al.(2020)Gao, Tchuendom, and Caines]{gao2020linear}
S.~Gao, R.~F. Tchuendom, and P.~E. Caines.
\newblock Linear quadratic graphon field games.
\newblock \emph{arXiv preprint arXiv:2006.03964}, 2020.

\bibitem[Gripenberg et~al.(1990)Gripenberg, Londen, and Staffans]{gripenberg1990volterra}
G.~Gripenberg, S.-O. Londen, and O.~Staffans.
\newblock \emph{Volterra integral and functional equations}.
\newblock Cambridge University Press, 1990.

\bibitem[Hall(2013)]{hall2013quantum}
B.~C. Hall.
\newblock \emph{Quantum theory for mathematicians}.
\newblock Springer Science \& Business Media, 2013.

\bibitem[Hang et~al.(2020)Hang, Lv, Xing, Huang, and Hu]{hang2020human}
P.~Hang, C.~Lv, Y.~Xing, C.~Huang, and Z.~Hu.
\newblock Human-like decision making for autonomous driving: A noncooperative game theoretic approach.
\newblock \emph{IEEE Transactions on Intelligent Transportation Systems}, 22\penalty0 (4):\penalty0 2076--2087, 2020.

\bibitem[Janson(2010)]{janson2010graphons}
S.~Janson.
\newblock Graphons, cut norm and distance, couplings and rearrangements.
\newblock \emph{arXiv preprint arXiv:1009.2376}, 2010.

\bibitem[Kinderlehrer and Stampacchia(2000)]{kinderlehrer2000introduction}
D.~Kinderlehrer and G.~Stampacchia.
\newblock \emph{An introduction to variational inequalities and their applications}.
\newblock SIAM, 2000.

\bibitem[Lacker and Soret(2023)]{lacker2023label}
D.~Lacker and A.~Soret.
\newblock A label-state formulation of stochastic graphon games and approximate equilibria on large networks.
\newblock \emph{Mathematics of Operations Research}, 48\penalty0 (4):\penalty0 1987--2018, 2023.

\bibitem[Leng and Parlar(2005)]{leng2005game}
M.~Leng and M.~Parlar.
\newblock Game theoretic applications in supply chain management: a review.
\newblock \emph{INFOR: Information Systems and Operational Research}, 43\penalty0 (3):\penalty0 187--220, 2005.

\bibitem[Lov{\'a}sz(2012)]{lovasz2012large}
L.~Lov{\'a}sz.
\newblock \emph{Large networks and graph limits}.
\newblock American Mathematical Society, 2012.

\bibitem[Lov{\'a}sz and Szegedy(2006)]{lovasz2006limits}
L.~Lov{\'a}sz and B.~Szegedy.
\newblock Limits of dense graph sequences.
\newblock \emph{Journal of Combinatorial Theory, Series B}, 96\penalty0 (6):\penalty0 933--957, 2006.

\bibitem[Neuman and Tuschmann(2024)]{neuman2024stochastic}
E.~Neuman and S.~Tuschmann.
\newblock Stochastic graphon games with memory.
\newblock \emph{arXiv preprint arXiv:2411.05896}, 2024.

\bibitem[Parise and Ozdaglar(2021)]{parise2021analysis}
F.~Parise and A.~Ozdaglar.
\newblock Analysis and interventions in large network games.
\newblock \emph{Annual Review of Control, Robotics, and Autonomous Systems}, 4\penalty0 (1):\penalty0 455--486, 2021.

\bibitem[Parise and Ozdaglar(2023)]{parise2023graphon}
F.~Parise and A.~Ozdaglar.
\newblock Graphon games: A statistical framework for network games and interventions.
\newblock \emph{Econometrica}, 91\penalty0 (1):\penalty0 191--225, 2023.

\bibitem[Sayedi(2018)]{sayedi2018real}
A.~Sayedi.
\newblock Real-time bidding in online display advertising.
\newblock \emph{Marketing Science}, 37\penalty0 (4):\penalty0 553--568, 2018.

\bibitem[Sun(2006)]{sun2006exact}
Y.~Sun.
\newblock The exact law of large numbers via {Fubini} extension and characterization of insurable risks.
\newblock \emph{Journal of Economic Theory}, 126\penalty0 (1):\penalty0 31--69, 2006.

\bibitem[Sun and Zhang(2009)]{sun2009individual}
Y.~Sun and Y.~Zhang.
\newblock Individual risk and {Lebesgue} extension without aggregate uncertainty.
\newblock \emph{Journal of Economic Theory}, 144\penalty0 (1):\penalty0 432--443, 2009.

\bibitem[Tangpi and Zhou(2024)]{tangpi2024optimal}
L.~Tangpi and X.~Zhou.
\newblock Optimal investment in a large population of competitive and heterogeneous agents.
\newblock \emph{Finance and Stochastics}, 28\penalty0 (2):\penalty0 497--551, 2024.

\end{thebibliography}
\end{document}